\documentclass[11pt,a4paper]{amsart}
\pdfoutput=1
\usepackage[centering, margin=1in]{geometry}
\usepackage{accents,amsmath,amsfonts,amsthm,amssymb,enumitem,graphicx,mathrsfs,mathtools,pgfplots}
\usepackage[stretch=10]{microtype}
\usepackage[unicode]{hyperref}
\usepackage{xcolor}
\usepackage[normalem]{ulem}
\definecolor{burntorange}{rgb}{0.8, 0.33, 0.0}
\definecolor{dark-red}{rgb}{0.4,0.15,0.15}
\definecolor{dark-blue}{rgb}{0.15,0.15,0.4}
\definecolor{darkpowderblue}{rgb}{0.0, 0.2, 0.6}
\definecolor{darkspringgreen}{rgb}{0.09, 0.45, 0.27}
\definecolor{medium-blue}{rgb}{0,0,0.5}
\hypersetup{
	pdftitle={\texorpdfstring{$L^p$}{Lp}-Norm Bounds for Automorphic Forms via Spectral Reciprocity},
	pdfauthor={Peter Humphries and Rizwanur Khan},
	pdfnewwindow=true,
	colorlinks, linkcolor={dark-red},
	citecolor={dark-blue}, urlcolor={medium-blue}
}
\allowdisplaybreaks
\newcommand{\BB}{\mathcal{B}}
\newcommand{\C}{\mathbb{C}}
\newcommand{\CC}{\mathcal{C}}
\newcommand{\dee}{\partial}
\newcommand{\e}{\varepsilon}
\newcommand{\Gscr}{\mathscr{G}}

\newcommand{\Hb}{\mathbb{H}}
\newcommand{\HH}{\mathcal{H}}
\newcommand{\hol}{\mathrm{hol}}
\newcommand{\JJ}{\mathcal{J}}
\newcommand{\Kscr}{\mathscr{K}}
\newcommand{\Lscr}{\mathscr{L}}
\newcommand{\N}{\mathbb{N}}
\newcommand{\Nscr}{\mathscr{N}}
\newcommand{\Q}{\mathbb{Q}}
\newcommand{\R}{\mathbb{R}}
\newcommand{\spec}{\mathrm{spec}}
\newcommand{\T}{\mathbb{T}}
\newcommand{\Z}{\mathbb{Z}}

\DeclareMathOperator{\ad}{ad}
\DeclareMathOperator{\GL}{GL}
\DeclareMathOperator{\Ogp}{O}
\DeclareMathOperator{\PGL}{PGL}
\DeclareMathOperator{\PSL}{PSL}
\DeclareMathOperator*{\Res}{Res}
\DeclareMathOperator{\sech}{sech}
\DeclareMathOperator{\SL}{SL}
\DeclareMathOperator{\sym}{sym}
\DeclareMathOperator{\Zgp}{Z}
\numberwithin{equation}{section}
\newtheorem{theorem}[equation]{Theorem}
\newtheorem{conjecture}[equation]{Conjecture}
\newtheorem{corollary}[equation]{Corollary}
\newtheorem{lemma}[equation]{Lemma}
\newtheorem{proposition}[equation]{Proposition}
\theoremstyle{remark}
\newtheorem{remark}[equation]{Remark}
\begin{document}

\title{$L^p$-Norm Bounds for Automorphic Forms via Spectral Reciprocity}

\author{Peter Humphries}

\address{Department of Mathematics, University of Virginia, Charlottesville, VA 22904, USA}

\email{\href{mailto:pclhumphries@gmail.com}{pclhumphries@gmail.com}}

\urladdr{\href{https://sites.google.com/view/peterhumphries/}{https://sites.google.com/view/peterhumphries/}}

\author{Rizwanur Khan}

\address{Department of Mathematical Sciences, University of Texas at Dallas, Richardson, TX 75080, USA}

\email{\href{mailto:rizwanur.khan@utdallas.edu}{rizwanur.khan@utdallas.edu}}

\urladdr{\href{https://profiles.utdallas.edu/rizwanur.khan}{https://profiles.utdallas.edu/rizwanur.khan}}


\thanks{The first author was supported by the National Science Foundation grant DMS-2302079 and the Simons Foundation (award 965056). The second author was supported by the National Science Foundation grants DMS-2001183/DMS-2344044 and DMS-2140604/DMS-2341239 and the Simons Foundation (award 630985).}

\begin{abstract}
Let $g$ be a Hecke--Maa\ss{} cusp form on the modular surface $\SL_2(\Z) \backslash \Hb$, namely an $L^2$-normalised nonconstant Laplacian eigenfunction on $\SL_2(\Z) \backslash \Hb$ that is additionally a joint eigenfunction of every Hecke operator. We prove the $L^4$-norm bound $\|g\|_4 \ll_{\e} \lambda_g^{3/304 + \e}$, where $\lambda_g$ denotes the Laplacian eigenvalue of $g$, which improves upon Sogge's $L^4$-norm bound $\|g\|_4 \ll \lambda_g^{1/16}$ for Laplacian eigenfunctions on a compact Riemann surface by more than a six-fold power-saving. Interpolating with the sup-norm bound $\|g\|_{\infty} \ll_{\e} \lambda_g^{5/24 + \e}$ due to Iwaniec and Sarnak, this yields $L^p$-norm bounds for Hecke--Maa\ss{} cusp forms that are power-saving improvements on Sogge's bounds for all $p > 2$. Our paper marks the first improvement of Sogge's result on the modular surface. Furthermore, these methods yield for compact arithmetic surfaces the best $L^4$-norm bound to date.

Via the Watson--Ichino triple product formula, bounds for the $L^4$-norm of $g$ are reduced to bounds for certain mixed moments of $L$-functions. We bound these using two forms of spectral reciprocity: identities between two different moments of central values of $L$-functions. The first is a form of $\GL_3 \times \GL_2 \leftrightsquigarrow \GL_4 \times \GL_1$ spectral reciprocity, which relates a $\GL_2$ moment of $\GL_3 \times \GL_2$ Rankin--Selberg $L$-functions to a $\GL_1$ moment of $\GL_4 \times \GL_1$ Rankin--Selberg $L$-functions; this can be seen as a cuspidal analogue of Motohashi's formula relating the fourth moment of the Riemann zeta function to the third moment of central values of Hecke $L$-functions. The second is a form of $\GL_4 \times \GL_2 \leftrightsquigarrow \GL_4 \times \GL_2$ spectral reciprocity, which is a cuspidal analogue of a formula of Kuznetsov for the fourth moment of central values of Hecke $L$-functions.
\end{abstract}

\maketitle

\section{Introduction}

\subsection{\texorpdfstring{$L^p$}{Lp}-Norm Bounds for Hecke--Maa\ss{} Cusp Forms}

A fundamental problem in analysis is understanding the distribution of mass of Laplacian eigenfunctions via bounds for their $L^p$-norms in terms of the size of their Laplacian eigenvalue. We study this problem for \emph{arithmetic} Laplacian eigenfunctions on the modular surface $\Gamma \backslash \Hb$, where $\Hb \coloneqq \{z = x + iy \in \C : y > 0\}$ is the upper half-plane upon which the modular group $\Gamma \coloneqq \SL_2(\Z)$ acts via M\"{o}bius transformations.

Let $g$ be a Hecke--Maa\ss{} cusp form on $\Gamma \backslash \Hb$, namely a nonconstant Laplacian eigenfunction lying in the discrete spectrum of the Laplacian on $\Gamma \backslash \Hb$ that is additionally a joint eigenfunction of every Hecke operator\footnote{Since the Laplacian commutes with each Hecke operator, every nonconstant Laplacian eigenfunction on $\Gamma \backslash \Hb$ is a linear combination of Hecke--Maa\ss{} cusp forms. Moreover, this Hecke assumption ought to be automatic since the discrete spectrum of the Laplacian on $\Gamma \backslash \Hb$ is expected to be simple \cite{Ste94}.}. Thus $\Delta g = \lambda_g g$, where $\Delta \coloneqq -y^2 (\frac{\dee^2}{\dee x^2} + \frac{\dee^2}{\dee y^2})$ denotes the Laplace--Beltrami operator on $\Gamma \backslash \Hb$ and $\lambda_g \in (0,\infty)$ is the Laplacian eigenvalue of $g$. We scale $g$ to be $L^2$-normalised with respect to the probability Haar measure $\frac{3}{\pi} \frac{dx \, dy}{y^2}$ on $\Gamma \backslash \Hb$. The main result of this paper is the following bound for the $L^4$-norm of a Hecke--Maa\ss{} cusp form $g$ on $\Gamma \backslash \Hb$ in terms of $\lambda_g$.

\begin{theorem}
\label{thm:L4}
Let $g$ be a Hecke--Maa\ss{} cusp form on $\Gamma \backslash \Hb$ of Laplacian eigenvalue $\lambda_g$. Then
\[\|g\|_4 \coloneqq \left(\int_{\Gamma \backslash \Hb} |g(z)|^4 \, \frac{3}{\pi} \frac{dx \, dy}{y^2}\right)^{\frac{1}{4}} \ll_{\e} \lambda_g^{\frac{3}{304} + \e}.\]
\end{theorem}

Interpolating between the $L^2$-norm normalisation $\|g\|_2 = 1$ and the $L^{\infty}$-norm bound $\|g\|_{\infty} \ll_{\e} \lambda_g^{5/24 + \e}$ of Iwaniec and Sarnak \cite[Theorem 0.1]{IS95} via the log-convexity of $L^p$-norms, we deduce the following $L^p$-norm bounds for a Hecke--Maa\ss{} cusp form $g$.

\begin{corollary}
\label{cor:Lp}
Let $g$ be a Hecke--Maa\ss{} cusp form on $\Gamma \backslash \Hb$ of Laplacian eigenvalue $\lambda_g$. Then for $p \in [2,\infty]$, we have that $\|g\|_p \ll_{\e} \lambda_g^{\delta(p) + \e}$, where
\begin{equation}
\label{eqn:delta(p)}
\delta(p) = \begin{dcases*}
\frac{3}{152} - \frac{3}{76p} & for $2 \leq p \leq 4$,	\\
\frac{5}{24} - \frac{181}{228p} & for $4 \leq p \leq \infty$.
\end{dcases*}
\end{equation}
\end{corollary}

The method of proof of \hyperref[thm:L4]{Theorem \ref*{thm:L4}} is quite general and applies to Hecke--Maa\ss{} cusp forms on arithmetic surfaces other than the modular surface, leading to the following result.

\begin{theorem}
\label{thm:L4modified}
Let $q$ be squarefree and fixed and let $\Gamma'$ either be the Hecke congruence subgroup $\Gamma_0(q)$ or the congruence subgroup $\Gamma^D$ corresponding to the norm one units of a maximal order of an indefinite quaternion division algebra $D$ over $\Q$ of discriminant $q$. Let $g$ be a Hecke--Maa\ss{} newform on $\Gamma' \backslash \Hb$ of Laplacian eigenvalue $\lambda_g$. Then
\[\|g\|_4 \ll_{\e} \lambda_g^{\frac{3}{304} + \e}.\]
\end{theorem}

We sketch in \hyperref[sect:proofmodifiedsketch]{Section \ref*{sect:proofmodifiedsketch}} how the method of proof of \hyperref[thm:L4]{Theorem \ref*{thm:L4}} extends to yield \hyperref[thm:L4modified]{Theorem \ref*{thm:L4modified}}.

After this paper was written, Ki announced a proof of the essentially sharp upper bound $\|g\|_4 \ll_{\e} \lambda_g^{\e}$ for a Hecke--Maa\ss{} cusp form $g$ on $\Gamma \backslash \Hb$ \cite[Theorem 2]{Ki23}. The proof is via completely different methods: instead of relating $\|g\|_4^4$ to moments of $L$-functions, as we do, Ki uses the Fourier--Whittaker expansion of $g$ over a Siegel set. This method would potentially extend to Hecke--Maa\ss{} newforms on arithmetic surfaces other than the modular surface for which there exists a Fourier--Whittaker expansion. However, no such Fourier--Whittaker expansion exists for Hecke--Maa\ss{} newforms on a compact congruence arithmetic surface arising from a quaternion division algebra; nonetheless, our method remains valid in this setting.

We sketch the method of proof of \hyperref[thm:L4]{Theorem \ref*{thm:L4}} in \hyperref[sect:proofmethod]{Section \ref*{sect:proofmethod}}: broadly speaking, we relate $\| g\|_4^4$ to a mixed moment of central values of $L$-functions via Parseval's identity and the Watson--Ichino triple product formula and then proceed to bound this moment. To experts, it may come as no surprise that current conventional machinery in the analytic theory of automorphic forms (approximate functional equations, the Kuznetsov and Petersson formul\ae{}, the Vorono\u{\i} summation formul\ae{}, spectral large sieve inequalities, etc.) leads to \emph{some} nontrivial $L^4$-norm bound for Hecke--Maa\ss{} cusp forms. In this paper, we do not simply push such methods to their limit. The novelty of our method is the development and implementation, for the first time, of spectral reciprocity identities for moments of $L$-functions in the context of the $L^4$-norm problem. This opens up new avenues of approach that would otherwise be completely unavailable if one were working with approximate functional equations and other standard techniques. Moreover, as we discuss in \hyperref[sect:spectralreciprocityformulae]{Section \ref*{sect:spectralreciprocityformulae}}, the usage of these spectral reciprocity formul\ae{} \emph{cannot} be substituted with the method of approximate functional equations without majorly weakening the result.

\subsection{Related Results}

\subsubsection{$L^p$-Norm Bounds for Laplacian Eigenfunctions}

\hyperref[thm:L4]{Theorem \ref*{thm:L4}} and \hyperref[cor:Lp]{Corollary \ref*{cor:Lp}} fall under the umbrella of a large swathe of results concerning $L^p$-norm bounds for Laplacian eigenfunctions on manifolds. The fundamental result in this area is due to Sogge \cite{Sog88}, who has shown for $p \in [2,\infty]$ the $L^p$-norm bounds $\|g\|_p \ll \lambda_g^{\delta(n,p)}$, where
\begin{equation}
\label{eqn:delta(n,p)}
\delta(n,p) = \begin{dcases*}
\frac{n - 1}{8} - \frac{n - 1}{4p} & for $2 \leq p \leq \frac{2(n + 1)}{n - 1}$,	\\
\frac{n - 1}{4} - \frac{n}{2p} & for $\frac{2(n + 1)}{n - 1} \leq p \leq \infty$,
\end{dcases*}
\end{equation}
for an $L^2$-normalised Laplacian eigenfunction $g$ with Laplacian eigenvalue $\lambda_g$ on a compact $n$-dimensional Riemannian manifold $M$. These bounds are sharp on the $n$-sphere $S^n$ and should be thought of as the \emph{convexity bounds} for $L^p$-norms; thus \hyperref[cor:Lp]{Corollary \ref*{cor:Lp}} should be viewed as giving \emph{subconvex} $L^p$-norm bounds on $\Gamma \backslash \Hb$ for all $p > 2$.

\begin{figure}[h]
\centering
\begin{tikzpicture}[
declare function={
	d1(\x)= and(\x >= 0, \x < 1/6) * (1/4-\x) +
	and(\x > 1/6, \x < 1/2) * (1/8-\x/4)
	;
},
declare function={
d2(\x)= and(\x >= 0, \x < 1/4) * (5/24-181*\x/228) +
and(\x > 1/4, \x < 1/2) * (3/152-3*\x/76)
;
}
]
\begin{axis}[
legend pos=north east,
axis x line*=middle, axis y line*=middle,
hide obscured x ticks=false,
y=25cm,
x=25cm,
ymin=0, ymax=1/4, ytick={3/304,1/12,5/24,1/4}, yticklabels={
$\frac{3}{304}$, $\frac{1}{12}$, $\frac{5}{24}$, $\frac{1}{4}$
},
xmin=0, xmax=1/2, xtick={0,1/6,1/4,1/2}, xticklabels={
$0$, $\frac{1}{6}$, $\frac{1}{4}$, $\frac{1}{2}$
},
domain=0:1/2,samples=301,
]
\addplot [darkpowderblue,very thick,dash pattern={on 6pt off 2pt on 3pt off 2pt}] {d2(x)};
\addlegendentry{$\delta(1/p)$}
\addplot [darkspringgreen,very thick] {d1(x)};
\addlegendentry{$\delta(2,1/p)$}
\end{axis}
\end{tikzpicture}
\caption{A comparison of the exponent $\delta(p)$ given by \eqref{eqn:delta(p)} to Sogge's exponent $\delta(2,p)$ given by \eqref{eqn:delta(n,p)}.}
\label{fig:Soggediagram}
\end{figure}
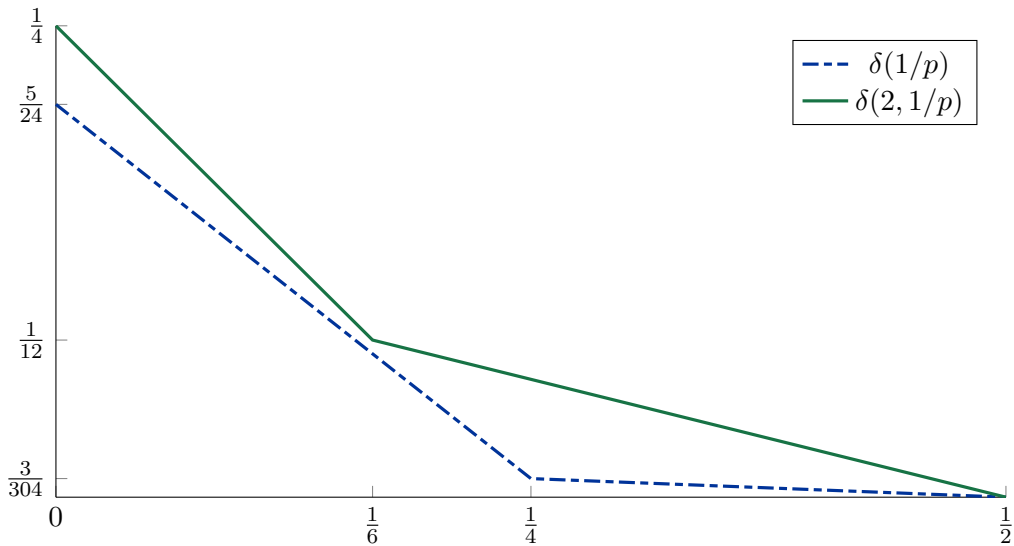

Logarithmic improvements to Sogge's $L^p$-norm bounds have been shown under various geometric assumptions on the underlying manifold, such as nonpositive sectional curvature \cite{BS18,BS19,CG23,HT15}. Furthermore, power-saving improvements to Sogge's bounds are known for certain manifolds: Zygmund \cite{Zyg74} proved that $\|g\|_p \ll 1$ for $2 \leq p \leq 4$ with $M = \T^2$, while for $M = \T^n$, Bourgain and Demeter \cite{BD15} (cf.~\cite[Theorem 13.12]{Dem20}) proved more generally the improved bounds $\|g\|_p \ll_{\e} \lambda_g^{\delta'(n,p) + \e}$ with $\delta'(n,p) = 0$ for $2 \leq p \leq \frac{2(n + 1)}{n - 1}$ and $n \geq 3$ and $\delta'(n,p) = \frac{n - 2}{4} - \frac{n}{2p}$ for $p \geq \frac{2(n - 1)}{n - 3}$ and $n \geq 4$. Finally, Marshall proved power-saving improvements to Sogge's bounds for Laplacian eigenfunctions on certain compact locally symmetric spaces that are additionally eigenfunctions of the full ring of invariant differential operators \cite[Theorem 1.1]{Mar16b}.

\subsubsection{The Iwaniec--Sarnak Conjecture}

For \emph{negatively curved surfaces}, Iwaniec and Sarnak have conjectured that Sogge's $L^p$-norm bounds fall well shy of the truth.

\begin{conjecture}[Iwaniec--Sarnak {\cite[Conjecture 4]{Sar03}}]
\label{conj:IS}
Let $M$ be a negatively curved surface and let $K \subseteq M$ be compact. Then for all $p \in [2,\infty]$,
\[\|g|_K\|_p \ll_{K,\e} \lambda_g^{\e}.\]
\end{conjecture}

This conjecture is quite strong: if $M$ is a compact arithmetic hyperbolic surface arising from a quaternion division algebra over $\Q$, then the bound $\|g\|_{\infty} \ll_{\e} \lambda_g^{\e}$ implies the generalised Lindel\"{o}f hypothesis for certain $L$-functions, since Hecke--Maa\ss{} cusp forms evaluated at distinguished points are essentially equal to central values of $L$-functions by Waldspurger's formula \cite{Wal85}. If true, \hyperref[conj:IS]{Conjecture \ref*{conj:IS}} is essentially sharp \cite[Theorem 1]{Mil10}; moreover, the assumption that $M$ be a \emph{surface} is necessary, since there are compact Riemannian manifolds with negative sectional curvature of dimension greater than $2$ for which there exist subsequences of Laplacian eigenfunctions whose $L^{\infty}$-norms exhibit power growth \cite[Theorem 1]{Mil11}.

Towards this conjecture, the only direct unconditional progress that has been made up until now is for Hecke--Maa\ss{} cusp forms on arithmetic hyperbolic surfaces, for which Iwaniec and Sarnak have proven the sup-norm bound $\|g\|_{\infty} \ll_{\e} \lambda_g^{5/24 + \e}$ \cite[Theorem 0.1]{IS95}, while Marshall has proven the $L^4$-norm bound $\|g\|_4 \ll \lambda_g^{3/56}$ \cite[Corollary 1.2]{Mar16a} for compact arithmetic hyperbolic surfaces arising from quaternion division algebras\footnote{Marshall's $L^4$-norm bound is a corollary of bounds he attains for the $L^2$-norm of geodesic restrictions of Hecke--Maa\ss{} cusp forms, as work of Blair and Sogge \cite[Theorem 1.1]{BS17} (building on earlier work of Bourgain \cite{Bou09}) relates such geodesic restriction bounds to $L^4$-norm bounds.}. The $L^4$-norm bound with exponent $\delta(4) = 3/304$ obtained in \hyperref[thm:L4]{Theorem \ref*{thm:L4}} for $\Gamma \backslash \Hb$ and in \hyperref[thm:L4modified]{Theorem \ref*{thm:L4modified}} for $\Gamma_0(q) \backslash \Hb$ and $\Gamma^D \backslash \Hb$ gives more than a six-fold improvement on the exponent $\delta(2,4) = 1/16$ of Sogge's $L^4$-norm bound for compact surfaces and more than a five-fold improvement on the exponent $3/56$ of Marshall's $L^4$-norm bound for compact arithmetic hyperbolic surfaces arising from quaternion division algebras.

\subsubsection{Conditional Improvements}
\label{sect:improvementsimprovementsubsect}

We highlight that the strengths of \hyperref[thm:L4]{Theorems \ref*{thm:L4}} and \ref{thm:L4modified} lie in the fact that these are \emph{unconditional} power-saving improvements upon Sogge's bound, as these bounds can be greatly strengthened \emph{conditionally}. Watson observed that under the assumption of the generalised Lindel\"{o}f hypothesis for $\GL_3 \times \GL_2$ Rankin--Selberg $L$-functions and $\GL_2$ standard $L$-functions, we have the almost sharp upper bound $\|g\|_4 \ll_{\e} \lambda_g^{\e}$ \cite[Corollary 2]{Wat08}\footnote{Sarnak \cite[Theorem 3]{Sar03} announced a proof, joint with Watson, of the almost sharp upper bound $\|g\|_4 \ll_{\e} \lambda_g^{\e}$ for Hecke--Maa\ss{} cusp forms $g$ on $\Gamma \backslash \Hb$ conditional only on the generalised Ramanujan conjecture (see additionally \cite[p.~62]{Wat08}, and compare \hyperref[fig:Soggediagram]{Figure \ref*{fig:Soggediagram}} to \cite[Figure 6]{Sar03}). However, this claim has subsequently been withdrawn (cf.~\cite[Remark 3.3]{Hum18}).}. Buttcane and the second author improved this to the \emph{asymptotic formula} $\|g\|_4^4 \sim 3$ under the same assumption \cite[Theorem 1.1]{BuK17}, in accordance with the random wave conjecture.

Asymptotics for the $L^4$-norm are known unconditionally for certain distinguished \emph{sparse} (in particular, density zero) subsequences of automorphic forms with additional arithmetic structure, namely for \emph{dihedral} Hecke--Maa\ss{} cusp forms \cite[Theorem 1.9]{HK20} and (in a regularised form) for \emph{Eisenstein series} \cite[Theorem 1.1]{DK20}. We explain why these strong results are possible unconditionally yet \hyperref[thm:L4]{Theorems \ref*{thm:L4}} and \ref{thm:L4modified} fall short of such an asymptotic formula in \hyperref[rem:Eisensteindihedral]{Remark \ref*{rem:Eisensteindihedral}}. In \hyperref[sect:GLH]{Sections \ref*{sect:GLH}} and \ref{sect:fifthmoment}, we discuss how \hyperref[thm:L4]{Theorems \ref*{thm:L4}} and \ref{thm:L4modified} may be improved under various conditional assumptions.

\subsubsection{Weight-Aspect and Level-Aspect Generalisations}

There are natural generalisations of these $L^p$-norm bounds to more general families of automorphic forms rather than just Hecke--Maa\ss{} cusp forms on the modular surface $\Gamma \backslash \Hb$. In particular, one can instead study $L^p$-norm bounds for \emph{holomorphic modular forms} (so that the Laplacian is replaced by the weight $k$ Laplacian with $k$ varying); moreover, one can investigate $L^p$-norm bounds in the \emph{level aspect} (so that the underlying orbifold $\Gamma_0(q) \backslash \Hb$ is varying); finally, one can study this in \emph{both} aspects simultaneously. We direct the reader to \cite{HS20,Sah17,Xia07} and the references therein for results on $L^{\infty}$-norm bounds, \cite{Blo13,BKY13,BuK15,DK20,Hum18,HK20,Kha14,KhSt24,Liu15,Luo14,Spi03} for results on $L^4$-norm bounds, and \cite{Mar16c} for $L^p$-norm bounds.

Notably, Blomer, the second author, and Young have proven that a holomorphic Hecke cusp form $G$ of weight $k \in 2\N$ satisfies $\|y^{k/2} G\|_4 \ll_{\e} k^{1/12 + \e}$ \cite[Theorem 1.1]{BKY13}; as $k^2$ is the analogue of the Laplacian eigenvalue $\lambda_g$, this should be thought of as a Weyl-strength subconvex improvement upon the \emph{convexity bound} $\|y^{k/2} G\|_4 \ll_{\e} k^{1/8 + \e}$. Under the assumption of the generalised Riemann hypothesis for $\GL_3 \times \GL_2$ Rankin--Selberg $L$-functions and $\GL_2$ standard $L$-functions, Zenz has shown the sharp upper bound $\|y^{k/2} G\|_4 \ll 1$ \cite[Theorem 1.1]{Zen23}, while it is conjectured that the \emph{asymptotic formula} $\|y^{k/2} G\|_4^4 \sim 2$ holds \cite[Conjecture 1.2]{BKY13}.

\subsection{Method of Proof}
\label{sect:proofmethod}

\subsubsection{Reduction to Mixed Moments of $L$-Functions}

The initial manoeuvres of our proof of \hyperref[thm:L4]{Theorem \ref*{thm:L4}} follow a well-trodden path pioneered by Sarnak \cite[p.~461]{Sar03}: we spectrally expand $\|g\|_4^4$ and apply the Watson--Ichino triple product formula. More precisely, we write $\|g\|_4^4 = \langle |g|^2,|g|^2\rangle$ and spectrally expand this via Parseval's theorem for $L^2(\Gamma \backslash \Hb)$. The resulting spectral expansion involves a sum over Hecke--Maa\ss{} cusp forms $f$ of terms of the form $|\langle |g|^2,f\rangle|^2$ and an integral over $t \in \R$ of terms of the form $|\langle |g|^2, E(\cdot,1/2 + it)\rangle|^2$, where $E(z,s)$ denotes the real analytic Eisenstein series. We then invoke the Watson--Ichino triple product formula, namely \cite[Theorem 1.1]{Ich08} and \cite[Theorem 3]{Wat08}, to relate these triple products to $L$-functions. As explicated in \cite[Section 2]{BuK17}, we arrive at the identity
\begin{multline}
\label{eqn:L4toLfunctions}
\int_{\Gamma \backslash \Hb} |g(z)|^4 \, \frac{3}{\pi} \frac{dx \, dy}{y^2} = 1 + \sum_{f \in \BB_0} \frac{L\left(\frac{1}{2},f\right) L\left(\frac{1}{2},\ad g \otimes f\right)}{L(1,\ad f) L(1,\ad g)^2} H(t_f)	\\
+ \frac{1}{2\pi} \int_{-\infty}^{\infty} \left|\frac{\zeta\left(\frac{1}{2} + it\right) L\left(\frac{1}{2} + it,\ad g\right)}{\zeta(1 + 2it) L(1,\ad g)}\right|^2 H(t) \, dt,
\end{multline}
where $\BB_0$ denotes an orthonormal basis of Hecke--Maa\ss{} cusp forms on $\Gamma \backslash \Hb$ and
\begin{equation}
\label{eqn:Htdefeq}
H(t) \coloneqq \frac{\pi^2}{24} \frac{\prod_{\pm_1} \Gamma\left(\frac{1}{4} \pm_1 \frac{it}{2}\right)^2 \prod_{\pm_2,\pm_3} \Gamma\left(\frac{1}{4} \pm_2 \frac{it}{2} \pm_3 it_g\right)}{\prod_{\pm_1} \Gamma\left(\frac{1}{2} \pm_1 it\right) \prod_{\pm_2} \Gamma\left(\frac{1}{2} \pm_2 it_g\right)^2}.
\end{equation}
Here $t_g \in [0,\infty)$ denotes the spectral parameter of $g$, so that $\lambda_g = \frac{1}{4} + t_g^2$. This process reduces the problem to bounding the mixed moment of $L$-functions appearing in \eqref{eqn:L4toLfunctions}.

\subsubsection{Bounds for Mixed Moments of $L$-Functions}

To proceed further, we break the sum over $f \in \BB_0$ and the integral over $t \in \R$ in \eqref{eqn:L4toLfunctions} into different ranges based on the size of the spectral parameter $t_f \in [0,\infty)$ of $f$ and on the size of $|t| \in [0,\infty)$ relative to $t_g$. As in \cite{BuK17,DK20,Hum18,HK20}, we fix $\alpha > 0$ and consider the following four ranges:
\begin{enumerate}[leftmargin=*,label=\textup{(\arabic*)}]
\item the short initial range $[0,t_g^{1 - \alpha}]$;
\item the bulk range $[t_g^{1 - \alpha},2t_g - t_g^{1 - \alpha}]$;
\item the short transition range $[2t_g - t_g^{1 - \alpha},2t_g]$;
\item the tail range $[2t_g,\infty)$.
\end{enumerate}
\hyperref[thm:L4]{Theorem \ref*{thm:L4}} is an immediate consequence of the spectral expansion \eqref{eqn:L4toLfunctions} together with the following collection of estimates for these four ranges, which constitute the main content of this paper.

\begin{proposition}
\label{prop:fourranges}
Let $g$ be a Hecke--Maa\ss{} cusp form on $\Gamma \backslash \Hb$ with spectral parameter $t_g$, let $H(t)$ be as in \eqref{eqn:Htdefeq}, and fix $\alpha \in (0,\frac{1}{100})$.
\begin{enumerate}[leftmargin=*,label=\textup{(\arabic*)}]
\item\label{item:initial} For the short initial range $[0,t_g^{1 - \alpha}]$, we have the bounds
\begin{equation}
\label{eqn:initial}
\sum_{\substack{f \in \BB_0 \\ t_f \leq t_g^{1 - \alpha}}} \frac{L\left(\frac{1}{2},f\right) L\left(\frac{1}{2},\ad g \otimes f\right)}{L(1,\ad f) L(1,\ad g)^2} H(t_f) + \frac{1}{2\pi} \int\limits_{|t| \leq t_g^{1 - \alpha}} \left|\frac{\zeta\left(\frac{1}{2} + it\right) L\left(\frac{1}{2} + it,\ad g\right)}{\zeta(1 + 2it) L(1,\ad g)}\right|^2 H(t) \, dt \ll_{\e} t_g^{\frac{3}{38} + \e}.
\end{equation}
\item\label{item:bulk} For the bulk range $[t_g^{1 - \alpha},2t_g - t_g^{1 - \alpha}]$, there exists a continuous function $c(\alpha)$ of $\alpha$ satisfying $\lim_{\alpha \to 0} c(\alpha) = 0$ such that
\begin{multline}
\label{eqn:bulk}
\sum_{\substack{f \in \BB_0 \\ t_g^{1 - \alpha} \leq t_f \leq 2t_g - t_g^{1 - \alpha}}} \frac{L\left(\frac{1}{2},f\right) L\left(\frac{1}{2},\ad g \otimes f\right)}{L(1,\ad f) L(1,\ad g)^2} H(t_f)	\\
+ \frac{1}{2\pi} \int\limits_{t_g^{1 - \alpha} \leq |t| \leq 2t_g - t_g^{1 - \alpha}} \left|\frac{\zeta\left(\frac{1}{2} + it\right) L\left(\frac{1}{2} + it,\ad g\right)}{\zeta(1 + 2it) L(1,\ad g)}\right|^2 H(t) \, dt \ll_{\e} t_g^{c(\alpha) + \e}.
\end{multline}
\item\label{item:transition} For the short transition range $[2t_g - t_g^{1 - \alpha},2t_g]$, we have that
\begin{multline}
\label{eqn:transition}
\sum_{\substack{f \in \BB_0 \\ 2t_g - t_g^{1 - \alpha} \leq t_f \leq 2t_g}} \frac{L\left(\frac{1}{2},f\right) L\left(\frac{1}{2},\ad g \otimes f\right)}{L(1,\ad f) L(1,\ad g)^2} H(t_f)	\\
+ \frac{1}{2\pi} \int\limits_{2t_g - t_g^{1 - \alpha} \leq |t| \leq 2t_g} \left|\frac{\zeta\left(\frac{1}{2} + it\right) L\left(\frac{1}{2} + it,\ad g\right)}{\zeta(1 + 2it) L(1,\ad g)}\right|^2 H(t) \, dt \ll_{\e} t_g^{\frac{3}{38} + \e}.
\end{multline}
\item\label{item:tail} For the tail range $[2t_g,\infty)$, we have that
\begin{equation}
\label{eqn:tail}
\sum_{\substack{f \in \BB_0 \\ t_f \geq 2t_g}} \frac{L\left(\frac{1}{2},f\right) L\left(\frac{1}{2},\ad g \otimes f\right)}{L(1,\ad f) L(1,\ad g)^2} H(t_f) + \frac{1}{2\pi} \int\limits_{|t| \geq 2t_g} \left|\frac{\zeta\left(\frac{1}{2} + it\right) L\left(\frac{1}{2} + it,\ad g\right)}{\zeta(1 + 2it) L(1,\ad g)}\right|^2 H(t) \, dt \ll_{\e} t_g^{\e}.
\end{equation}
\end{enumerate}
\end{proposition}

It is instructive at this point to consider which bounds are attainable for these four ranges under the assumption of the generalised Lindel\"{o}f hypothesis. Stirling's formula implies that the function $H(t)$ given by \eqref{eqn:Htdefeq} satisfies the asymptotic formula
\begin{multline}
\label{eqn:Htasymp}
H(t) = \frac{\pi^3}{3} \frac{1}{(1 + |t|)(1 + |2t_g + t|)^{1/2} (1 + |2t_g - t|)^{1/2}} e^{-\pi \Omega(t,t_g)}	\\
\times \left(1 + O\left(\frac{1}{1 + |t|} + \frac{1}{1 + |2t_g + t|} + \frac{1}{1 + |2t_g - t|}\right)\right),
\end{multline}
where
\[\Omega(t,t_g) \coloneqq \begin{dcases*}
0 & if $|t| \leq 2t_g$,	\\
|t| - 2t_g & if $|t| \geq 2t_g$.
\end{dcases*}\]
Thus the generalised Lindel\"{o}f hypothesis and the Weyl law combine to give the conditional bounds $O_{\e}(t_g^{-\alpha + \e})$ for the short initial range, $O_{\e}(t_g^{\e})$ for the bulk range, $O_{\e}(t_g^{-\alpha/2 + \e})$ for the short transition range, and $O_{\e}(t_g^{-1/2 + \e})$ for the tail range. Buttcane and the second author \cite{BuK17} showed that with further effort, one can obtain the conditional \emph{asymptotic formula} $2 + o(1)$ for the bulk range.

Without appealing to the generalised Lindel\"{o}f hypothesis, such strong bounds are no longer easily attained. Nonetheless, we shall show that the bound \eqref{eqn:tail} for the tail range is readily achieved due to the fact that $H(t)$ decays exponentially for $|t| \geq 2t_g$; moreover, the requisite bound \eqref{eqn:bulk} for the bulk range can be deduced with a modicum of effort from earlier work of Buttcane and the second author \cite{BuK17}. The bounds \eqref{eqn:initial} and \eqref{eqn:transition} for the short initial and short transition ranges, however, are far from immediate and require several new ideas.

A standard approach to bound the mixed moments of $L$-functions in these ranges would be to apply the Cauchy--Schwarz inequality to separate the $L$-functions, write these $L$-functions in terms of Dirichlet polynomials via the approximate functional equation, and apply the spectral large sieve. Unfortunately, this only yields the bounds $O(t_g^{1/2})$ for the short initial and short transition ranges, which merely recovers Sogge's $L^4$-norm bound. To improve upon these weaker bounds, we prove new forms of \emph{spectral reciprocity}: identities between two different moments of central values of $L$-functions. We apply these with a mix of other ideas, as we describe below.

\subsection{Spectral Reciprocity Formul\ae{}}
\label{sect:spectralreciprocityformulae}

\subsubsection{$\GL_3 \times \GL_2 \leftrightsquigarrow \GL_4 \times \GL_1$ Spectral Reciprocity}
\label{sect:GL3xGL2intro}

By an appropriate application of H\"{o}lder's inequality, we are led to the problem of determining bounds for the moments of $L$-functions
\begin{equation}
\label{eqn:firstmomentdyadic}
\sum_{f \in \BB_0} \frac{L\left(\frac{1}{2},\ad g \otimes f\right)}{L(1,\ad f)} h(t_f) + \frac{1}{2\pi} 
\int_{-\infty}^{\infty} \left|\frac{L\left(\frac{1}{2} + it,\ad g\right)}{\zeta(1 + 2it)}\right|^2 h(t) \, dt
\end{equation}
with $h(t)$ an appropriately chosen test function, such as a smooth approximation of the indicator function of a dyadic interval $[-2T,-T] \cup [T,2T]$.

The analogous moment with $g$ replaced by an Eisenstein series is
\begin{equation}
\label{eqn:Mformula}
\sum_{f \in \BB_0} \frac{L\left(\frac{1}{2},f\right)^3}{L(1,\ad f)} h(t_f) + \frac{1}{2\pi} \int_{-\infty}^{\infty} \left|\frac{\zeta\left(\frac{1}{2} + it\right)^3}{\zeta(1 + 2it)}\right|^2 h(t) \, dt.
\end{equation}
Via work of Motohashi \cite[Theorem 4.2]{Mot97}, given a sufficiently well-behaved test function $h$, there exists a corresponding transform $\HH$\footnote{Motohashi's formulation of this identity involves \emph{first} specifying $\HH$ and \emph{then} determining $h$ as a transform involving $\HH$. This process can be reversed; see, for example, work of Motohashi \cite{Mot99} and of Nelson \cite{Nel19b}.} such that there is an exact \emph{equality} between the moment \eqref{eqn:Mformula} and the sum of a main term dependent on $h$ together with a \emph{dual} moment
\[\int_{-\infty}^{\infty} \left|\zeta\left(\frac{1}{2} + it\right)\right|^4 \HH(t) \, dt.\]
For example, if one chooses $h(t)$ in \eqref{eqn:Mformula} to localise to an interval of the form $[-T - U,-T] \cup [T,T + U]$ with $1 \leq U \leq T$, then with some effort one can show that $\HH(t)$ is essentially localised to $|t| \leq T/U$, where it is of size $\approx U$.

In \hyperref[thm:3x2reciprocity]{Theorem \ref*{thm:3x2reciprocity}}, we prove a cuspidal analogue of Motohashi's formula, namely that given a sufficiently well-behaved test function $h$, there exists a corresponding transform $\HH$ (given as an explicit integral transform in \eqref{eqn:HHmuFt}), such that the moment \eqref{eqn:firstmomentdyadic} is \emph{exactly} equal to the sum of a main term dependent on $h$ together with a dual moment
\begin{equation}
\label{eqn:3x2dualmoment}
\int_{-\infty}^{\infty} L\left(\frac{1}{2} + it,\ad g\right) \zeta\left(\frac{1}{2} - it\right) \HH(t) \, dt.
\end{equation}

\subsubsection{$\GL_4 \times \GL_2 \leftrightsquigarrow \GL_4 \times \GL_2$ Spectral Reciprocity}
\label{sect:GL4xGL2intro}

We additionally prove a new form of spectral reciprocity for the mixed moment
\begin{equation}
\label{eqn:mixedmoment}
\sum_{f \in \BB_0} \frac{L\left(\frac{1}{2},f\right) L\left(\frac{1}{2},\ad g \otimes f\right)}{L(1,\ad f)} h(t_f) + \frac{1}{2\pi} \int_{-\infty}^{\infty} \left|\frac{\zeta\left(\frac{1}{2} + it\right) L\left(\frac{1}{2} + it,\ad g\right)}{\zeta(1 + 2it)}\right|^2 h(t) \, dt
\end{equation}
with $h(t)$ an appropriately chosen test function, such as a smooth approximation of the indicator function of a dyadic interval $[-2T,-T] \cup [T,2T]$. The analogous moment with $g$ replaced by an Eisenstein series is
\begin{equation}
\label{eqn:KMformula}
\sum_{f \in \BB_0} \frac{L\left(\frac{1}{2},f\right)^4}{L(1,\ad f)} h(t_f) + \frac{1}{2\pi} \int_{-\infty}^{\infty} \left|\frac{\zeta\left(\frac{1}{2} + it\right)^4}{\zeta(1 + 2it)}\right|^2 h(t) \, dt.
\end{equation}
Given a sufficiently well-behaved test function $h$, there exists a corresponding transform $\widetilde{h}$ such that there is an exact \emph{equality} between the moment \eqref{eqn:KMformula} and the sum of a main term dependent on $h$ together with a \emph{dual} moment
\[\sum_{f \in \BB_0} \frac{L\left(\frac{1}{2},f\right)^4}{L(1,\ad f)} \widetilde{h}(t_f) + \frac{1}{2\pi} \int_{-\infty}^{\infty} \left|\frac{\zeta\left(\frac{1}{2} + it\right)^4}{\zeta(1 + 2it)}\right|^2 \widetilde{h}(t) \, dt\]
as well as an additional dual moment of the same form involving a sum over holomorphic cusp forms. Such an identity is due to Kuznetsov (in an incomplete form in \cite{Kuz89} and completed in the unpublished preprint \cite{Kuz99}) and Motohashi \cite{Mot03}. If one chooses $h(t)$ in \eqref{eqn:KMformula} to localise to an interval of the form $[-T - U,-T] \cup [T,T + U]$ with $1 \leq U \leq T$, then with some effort one can show that $\widetilde{h}(t)$ is essentially localised to $|t| \leq T/U$, where it is of size $\approx U (1 + |t|)^{-1/2}$.

In \hyperref[thm:4x2reciprocity]{Theorem \ref*{thm:4x2reciprocity}}, we prove a cuspidal analogue of this formula of Kuznetsov, namely that given a sufficiently well-behaved test function $h$, there exists a corresponding transform $\widetilde{h}$ (given as an explicit integral transform in \eqref{eqn:tildehpmdefeq}) such that the mixed moment \eqref{eqn:mixedmoment} is equal to the sum of a main term dependent on $h$ together with a dual moment
\[\sum_{f \in \BB_0} \frac{L\left(\frac{1}{2},f\right) L\left(\frac{1}{2},\ad g \otimes f\right)}{L(1,\ad f)} \widetilde{h}(t_f) + \frac{1}{2\pi} \int_{-\infty}^{\infty} \left|\frac{\zeta\left(\frac{1}{2} + it\right) L\left(\frac{1}{2} + it,\ad g\right)}{\zeta(1 + 2it)}\right|^2 \widetilde{h}(t) \, dt\]
as well as an additional dual moment of the same form involving a sum over holomorphic cusp forms.

\subsubsection{Applications of Spectral Reciprocity for the Short Initial Range}

We apply these forms of spectral reciprocity to prove the bound \eqref{eqn:initial} for the short initial range. By a dyadic subdivision, we are led to bounding the mixed moment
\begin{equation}
\label{eqn:mixedmomentdyadic}
\sum_{\substack{f \in \BB_0 \\ T \leq t_f \leq 2T}} \frac{L\left(\frac{1}{2},f\right) L\left(\frac{1}{2},\ad g \otimes f\right)}{L(1,\ad f)} + \frac{1}{2\pi} \int\limits_{T \leq |t| \leq 2T} \left|\frac{\zeta\left(\frac{1}{2} + it\right) L\left(\frac{1}{2} + it,\ad g\right)}{\zeta(1 + 2it)}\right|^2 \, dt
\end{equation}
with $T \leq t_g^{1 - \alpha}$.

For $T \leq t_g^{3/19}$, so that the mixed moment \eqref{eqn:mixedmomentdyadic} is particularly short, we are unable to do better than simply applying the Cauchy--Schwarz inequality and the spectral large sieve. For $t_g^{1/2} \leq T \leq t_g^{16/19}$, on the other hand, we obtain improved bounds for \eqref{eqn:mixedmomentdyadic} via H\"{o}lder's inequality and an application of $\GL_3 \times \GL_2 \leftrightsquigarrow \GL_4 \times \GL_1$ spectral reciprocity to bound
\[\sum_{\substack{f \in \BB_0 \\ T \leq t_f \leq 2T}} \frac{L\left(\frac{1}{2},\ad g \otimes f\right)}{L(1,\ad f)} + \frac{1}{2\pi} \int\limits_{T \leq |t| \leq 2T} \left|\frac{L\left(\frac{1}{2} + it,\ad g\right)}{\zeta(1 + 2it)}\right|^2 \, dt,\]
which crucially relies on the nonnegativity of $L(1/2,\ad g \otimes f)$ \cite[Theorem 1.1]{Lap03}. After choosing our test function $h$ to approximate the indicator function of $[-2T,-T] \cup [T,2T]$, we are left with determining the behaviour of the corresponding transform $\HH$. This is no easy task: unlike most previous work involving similar spectral reciprocity formul\ae{} (such as \cite{LNQ23}), we require \emph{hybrid} bounds with explicit dependence not only on the dyadic parameter $T$ but additionally on the spectral parameter $t_g$, which is further complicated by the fact that $T < t_g$, so that our moment is short relative to the conductor of the $L$-function $L(1/2,\ad g \otimes f)$. Instead of being of size $\approx T$ and localised to $|t| \leq 1$, as is the case for $t_g$ \emph{bounded} relative to $T$, we find that $\HH(t)$ is essentially localised to $|t| \leq t_g^2 / T^2$, where it is of size $\approx T^2 / t_g$. We subsequently bound the dual mixed moment \eqref{eqn:3x2dualmoment} of $L(1/2 + it,\ad g) \zeta(1/2 - it)$ via the Cauchy--Schwarz inequality and the Montgomery--Vaughan mean value theorem for Dirichlet polynomials.

Finally, for the remaining ranges $t_g^{3/19} \leq T \leq t_g^{1/2}$ and $t_g^{16/19} \leq T \leq t_g^{1 - \alpha}$, we obtain strongest bounds for \eqref{eqn:mixedmomentdyadic} by using $\GL_4 \times \GL_2 \leftrightsquigarrow \GL_4 \times \GL_2$ spectral reciprocity. We once more choose $h$ to approximate the indicator function of $[-2T,-T] \cup [T,2T]$ and determine the behaviour of the corresponding transform $\widetilde{h}$. Again, this is a cumbersome task due to the hybrid bounds required, unlike other previous applications of spectral reciprocity (such as \cite{BLM19}). Instead of being of size $\approx T$ and localised to $|t| \leq 1$, as is the case for $t_g$ \emph{bounded} relative to $T$, we find that $\widetilde{h}(t)$ is essentially localised to $|t| \leq t_g / T$, where it is of size $\approx T^2 / t_g$. Thus we treat the range $t_g^{16/19} \leq T \leq t_g^{1 - \alpha}$ by using $\GL_4 \times \GL_2 \leftrightsquigarrow \GL_4 \times \GL_2$ spectral reciprocity to reduce the problem to the range $T \leq t_g^{3/19}$, which we previously treated by the spectral large sieve. Similarly, the range $t_g^{3/19} \leq T \leq t_g^{1/2}$ is treated by using $\GL_4 \times \GL_2 \leftrightsquigarrow \GL_4 \times \GL_2$ spectral reciprocity to reduce the problem to the range $t_g^{1/2} \leq T \leq t_g^{16/19}$, which we previously treated via $\GL_3 \times \GL_2 \leftrightsquigarrow \GL_4 \times \GL_1$ spectral reciprocity. Note that our treatment of this latter range is highly unusual and seemingly counter-intuitive, since we move from a relatively short moment to a \emph{longer} moment, and yet this process nonetheless yields improved bounds.

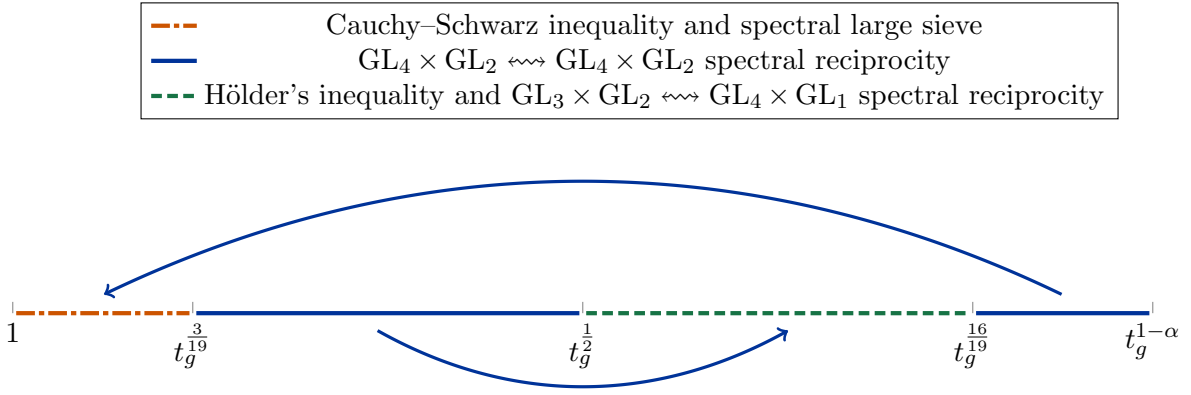
\begin{figure}[h]
\centering
\begin{tikzpicture}[
declare function={
	d1(\x)=0
	;
},
declare function={
	d2(\x)= sqrt(2-x*x+x)-1.43
	;
},
declare function={
d3(\x)= 0.52-sqrt((1/16)-x*x+x)
;
}
]
\begin{axis}[
legend pos=north east,
axis x line*=middle, axis y line*=middle,
separate axis lines,
hide obscured x ticks=false,
y axis line style={draw opacity=0},
x axis line style={draw opacity=0},
y=25cm,
x=15cm,
ymin=-1/16, ymax=11/64, ymajorticks=false,
xmin=0, xmax=1, xtick={0,3/19,1/2,16/19,1}, xticklabels={$1$,$t_g^{\frac{3}{19}}$, $t_g^{\frac{1}{2}}$, $t_g^{\frac{16}{19}}$, $t_g^{1 - \alpha}$},
domain=0:1,samples=301,
]
\addplot [ultra thick, color=burntorange, dash pattern={on 7pt off 1.5pt on 2pt off 1.5pt},domain=0.003:0.155] {d1(x)};
\addlegendentry{Cauchy--Schwarz inequality and spectral large sieve}
\addplot [ultra thick, color=darkpowderblue,domain=0.161:0.497] {d1(x)};
\addlegendentry{$\GL_4 \times \GL_2 \leftrightsquigarrow \GL_4 \times \GL_2$ spectral reciprocity}
\addplot [ultra thick, color=darkspringgreen, dash pattern={on 4pt off 2pt},domain=0.503:0.839] {d1(x)};
\addlegendentry{H\"{o}lder's inequality and $\GL_3 \times \GL_2 \leftrightsquigarrow \GL_4 \times \GL_1$ spectral reciprocity}
\addplot [ultra thick, color=darkpowderblue,domain=0.845:0.997] {d1(x)};
\addplot [<-, very thick, color=darkpowderblue, domain=0.08:0.92] {d2(x)};
\addplot [->, very thick, color=darkpowderblue, domain=0.32:0.68] {d3(x)};
\end{axis}
\end{tikzpicture}
\caption{The short initial range divided into portions based on the techniques used. The arrows indicate how the treatment of a portion via $\GL_4 \times \GL_2 \leftrightsquigarrow \GL_4 \times \GL_2$ spectral reciprocity reduces the problem to a different portion.}
\end{figure}

It is crucial to note that this last step \emph{cannot} be achieved without using $\GL_4 \times \GL_2 \leftrightsquigarrow \GL_4 \times \GL_2$ in its form as an \emph{exact} identity of two moments of $L$-functions. Indeed, by modifying the proof to use approximate functional equations, this spectral reciprocity formula may instead be proven in the form of an \emph{approximate} identity of moments of $L$-functions, where each moment involves Dirichlet polynomials in place of $L$-functions. Such an approximate identity, however, is insufficient for our applications. We make vital use of the nonnegativity of the $L$-function $L(1/2,\ad g \otimes f)$ appearing in the dual moment, for this allows us to use H\"{o}lder's inequality with exponents that lead to \emph{odd} moments of these central $L$-values. With Dirichlet polynomials, we no longer have nonnegativity, and so this avenue of approach is closed to us; in particular, we would only be able to prove strictly weaker $L^4$-norm bounds for Hecke--Maa\ss{} cusp forms via this alternative method using approximate functional equations.

\subsubsection{Applications of Spectral Reciprocity for the Short Transition Range}\label{sect:shorttransitionrange}

To prove the bound \eqref{eqn:transition} for the short transition range, we must work over shorter intervals, namely with mixed moments of the form
\begin{equation}
\label{eqn:mixedmomentshort}
\sum_{\substack{f \in \BB_0 \\ T - U \leq t_f \leq T + U}} \frac{L\left(\frac{1}{2},f\right) L\left(\frac{1}{2},\ad g \otimes f\right)}{L(1,\ad f)} + \frac{1}{2\pi} \int\limits_{T - U \leq |t| \leq T + U} \left|\frac{\zeta\left(\frac{1}{2} + it\right) L\left(\frac{1}{2} + it,\ad g\right)}{\zeta(1 + 2it)}\right|^2 \, dt
\end{equation}
with $2t_g - t_g^{1 - \alpha} \leq T \leq 2t_g$ and $U = t_g - \frac{T}{2} + 1$. This is the \emph{conductor-dropping} range: the conductor of $L(1/2,f) L(1/2,\ad g \otimes f)$ in this range is $\asymp t_g^6 (1 + 2t_g - t_f)^2$, which \emph{decreases} as $t_f$ approaches $2t_g$.

For the range $t_g^{1/3} \leq U \leq t_g^{1 - \alpha}$, we bound \eqref{eqn:mixedmomentshort} via $\GL_4 \times \GL_2 \leftrightsquigarrow \GL_4 \times \GL_2$ spectral reciprocity. We begin by choosing $h$ to approximate the indicator function of $[-T - U,-T + U] \cup [T - U,T + U]$ and determining the behaviour of the transform $\widetilde{h}$. Once more, the hybrid nature of the problem alters the behaviour of the transform: instead of being of size $U (1 + |t|)^{-1/2}$ and localised to $|t| \leq T / U$, as is the case for $t_g$ \emph{bounded} relative to $T$ and $U$, the transform $\widetilde{h}(t)$ is instead localised to the much shorter range $|t| \leq (T/U)^{1/2}$, where it is of larger size $\approx U$. Since $T \asymp t_g$, we have thereby reduced the problem back to the short initial range, and so we can simply appeal to our previously determined bounds for this range.

For the remaining range $1 \leq U \leq t_g^{1/3}$, we first apply H\"{o}lder's inequality, which leaves us with the problem of bounding
\begin{equation}
\label{eqn:firstmomentshort}
\sum_{\substack{f \in \BB_0 \\ T - U \leq t_f \leq T + U}} \frac{L\left(\frac{1}{2},\ad g \otimes f\right)}{L(1,\ad f)} + \frac{1}{2\pi} \int\limits_{T - U \leq |t| \leq T + U} \left|\frac{L\left(\frac{1}{2} + it,\ad g\right)}{\zeta(1 + 2it)}\right|^2 \, dt.
\end{equation}
While it is likely that we could bound this adequately using $\GL_3 \times \GL_2 \leftrightsquigarrow \GL_4 \times \GL_1$ spectral reciprocity, the short length of the moment (in particular, the fact that $U \leq T^{1/3}$) means that the analysis of the size and length of the transform $\HH$ becomes significantly more challenging. Instead, we bound \eqref{eqn:firstmomentshort} via a more classical approach using approximate functional equations and the Kuznetsov formula, which is more straightforward and yields sufficiently strong bounds for our purposes.

\section{Automorphic Preliminaries}

\subsection{Spectral Summation Formul\ae{}}

The central tools that we make use of are the Kuznetsov and Petersson formul\ae{}. The former involves sums of Hecke eigenvalues $\lambda_f(n)$ of Hecke--Maa\ss{} cusp forms $f$ on $\Gamma \backslash \Hb$ over an orthonormal basis $\BB_0$ of such cusp forms together with the root number $\epsilon_f \in \{1,-1\}$, as well as an integral involving the Hecke eigenvalues $\lambda(n,t) \coloneqq \sum_{ab = n} a^{it} b^{-it}$ of the real analytic Eisenstein series $E(z,1/2 + it)$. The latter involves sums of Hecke eigenvalues $\lambda_f(n)$ of holomorphic Hecke cusp forms $f$ on $\Gamma \backslash \Hb$ of weight $k_f \in 2\N$ over an orthonormal basis $\BB_{\hol}$ of such cusp forms. Both formul\ae{} express sums of Hecke eigenvalues in terms of sums of Kloosterman sums weighted by Bessel functions.

\begin{theorem}[{Kuznetsov formula \cite[Theorem 9.3]{Iwa02}}]
Let $\delta > 0$, and let $h^{\pm}(t)$ be a function that is even, holomorphic in the horizontal strip $|\Im(t)| \leq 1/2 + \delta$, and satisfies $h^{\pm}(t) \ll (1 + |t|)^{-2 - \delta}$. Then for $m,n \in \N$,
\begin{multline}
\label{eqn:Kuznetsovformula}
\sum_{f \in \BB_0} \epsilon_f^{\frac{1 \mp 1}{2}} \frac{\lambda_f(m) \lambda_f(n)}{L(1,\ad f)} h^{\pm}(t_f) + \frac{1}{2\pi} \int_{-\infty}^{\infty} \frac{\lambda(m,t) \lambda(n,t)}{\zeta(1 + 2it) \zeta(1 - 2it)} h^{\pm}(t) \, dt	\\
= \delta_{m, \pm n} \Nscr^{\pm} h^{\pm} + \sum_{c = 1}^{\infty} \frac{S(m,\pm n;c)}{c} (\Kscr^{\pm} h^{\pm})\left(\frac{\sqrt{mn}}{c}\right),
\end{multline}
where
\begin{gather}
\label{eqn:Kloostermandefeq}
S(m,n;c) \coloneqq \sum_{d \in (\Z/c\Z)^{\times}} e\left(\frac{md + n\overline{d}}{c}\right),	\\
\label{eqn:NscrpmKscrpmdefeq}
\Nscr^{\pm} h^{\pm} \coloneqq \int_{-\infty}^{\infty} h^{\pm}(r) \, d_{\spec}r, \qquad (\Kscr^{\pm} h^{\pm})(x) \coloneqq \int_{-\infty}^{\infty} \JJ_r^{\pm}(x) h^{\pm}(r) \, d_{\spec}r,	\\
\label{eqn:JJrpmdefeq}
\JJ_r^{+}(x) \coloneqq \frac{\pi i}{\sinh \pi r} \left(J_{2ir}(4\pi x) - J_{-2ir}(4\pi x)\right), \qquad \JJ_r^{-}(x) \coloneqq 4 \cosh \pi r K_{2ir}(4\pi x),	\\
\label{eqn:dspecdefeq}
d_{\spec} r \coloneqq \frac{1}{2\pi^2} r \tanh \pi r \, dr.
\end{gather}
\end{theorem}

Here $J_{\alpha}(x)$ denotes the Bessel function of the first kind and $K_{\alpha}(x)$ denotes the modified Bessel function of the second kind.

\begin{theorem}[{Petersson formula \cite[Theorem 9.6]{Iwa02}}]
Let $h^{\hol} : 2\N \to \C$ be a sequence satisfying $h^{\hol}(k) \ll k^{-2 - \delta}$ for some $\delta > 0$. Then for $m,n \in \N$,
\begin{equation}
\label{eqn:Peterssonformula}
\sum_{f \in \BB_{\hol}} \frac{\lambda_f(m) \lambda_f(n)}{L(1,\ad f)} h^{\hol}(k_f) = \delta_{m,n} \Nscr^{\hol} h^{\hol} + \sum_{c = 1}^{\infty} \frac{S(m,n;c)}{c} (\Kscr^{\hol} h^{\hol})\left(\frac{\sqrt{mn}}{c}\right).
\end{equation}
Here
\begin{gather}
\label{eqn:Nscrholdefeq}
\Nscr^{\hol} h^{\hol} \coloneqq \sum_{\substack{k = 2 \\ k \equiv 0 \hspace{-.25cm} \pmod{2}}}^{\infty} \frac{k - 1}{2 \pi^2} h^{\hol}(k),	\\
\label{eqn:KscrholJJkholdefeq}
(\Kscr^{\hol} h^{\hol})(x) \coloneqq \sum_{\substack{k = 2 \\ k \equiv 0 \hspace{-.25cm} \pmod{2}}}^{\infty} \frac{k - 1}{2 \pi^2} \JJ_k^{\hol}(x) h^{\hol}(k), \qquad \JJ_k^{\hol}(x) \coloneqq 2\pi i^{-k} J_{k - 1}(4\pi x).
\end{gather}
\end{theorem}

We also use the Kuznetsov and Petersson formul\ae{} \emph{in reverse}, namely the spectral decomposition of sums of Kloosterman sums.

\begin{theorem}[{Kloosterman summation formula \cite[Theorem 16.5]{IK04}}]
Let $H \in C^3((0,\infty))$ be a function satisfying $x^j \frac{d^j}{dx^j} H(x) \ll \min\{x^{1/2 + \delta},x^{-1-\delta}\}$ for $j \in \{0,1,2,3\}$ and for some $\delta > 0$. Then for $m,n \in \N$,
\begin{multline}
\label{eqn:Kloostermanformula}
\sum_{f \in \BB_0} \epsilon_f^{\frac{1 \mp 1}{2}} \frac{\lambda_f(m) \lambda_f(n)}{L(1,\ad f)} (\Lscr^{\pm} H)(t_f) + \frac{1}{2\pi} \int_{-\infty}^{\infty} \frac{\lambda(m,t) \lambda(n,t)}{\zeta(1 + 2it) \zeta(1 - 2it)} (\Lscr^{\pm} H)(t) \, dt	\\
+ \delta_{\pm,+} \sum_{f \in \BB_{\hol}} \frac{\lambda_f(m) \lambda_f(n)}{L(1,\ad f)} (\Lscr^{\hol} H)(k_f)	\\
= \sum_{c = 1}^{\infty} \frac{S(m,\pm n;c)}{c} H\left(\frac{\sqrt{mn}}{c}\right),
\end{multline}
where
\begin{equation}
\label{eqn:Lscrdefeq}
(\Lscr^{\pm} H)(t) \coloneqq \int_{0}^{\infty} \JJ_t^{\pm}(x) H(x) \, \frac{dx}{x}, \qquad (\Lscr^{\hol} H)(k) \coloneqq \int_{0}^{\infty} \JJ_k^{\hol}(x) H(x) \, \frac{dx}{x}.
\end{equation}
\end{theorem}

\subsection{Mellin Transforms}

We record the following Mellin transforms of the functions $\JJ_r^{\pm}$ and $\JJ_k^{\hol}$ as in \eqref{eqn:JJrpmdefeq} and \eqref{eqn:KscrholJJkholdefeq}.

\begin{lemma}[{\cite[(A.7)]{BLM19}, \cite[(3.13)]{BlK19b}}]
We have that
\begin{align}
\notag
\widehat{\JJ_r^{+}}(s) & = \frac{\pi i (2\pi)^{-s}}{2 \sinh \pi r} \left(\frac{\Gamma\left(\frac{s}{2} + ir\right)}{\Gamma\left(1 - \frac{s}{2} + ir\right)} - \frac{\Gamma\left(\frac{s}{2} - ir\right)}{\Gamma\left(1 - \frac{s}{2} - ir\right)}\right)	\\
\label{eqn:JJr+Mellin}
& = (2\pi)^{-s} \Gamma\left(\frac{s}{2} + ir\right) \Gamma\left(\frac{s}{2} - ir\right) \cos \frac{\pi s}{2},	\\
\notag
\widehat{\JJ_r^{-}}(s) & = \frac{\pi i (2\pi)^{-s}}{2 \tanh \pi r \cos \frac{\pi s}{2}} \left(\frac{\Gamma\left(\frac{s}{2} + ir\right)}{\Gamma\left(1 - \frac{s}{2} + ir\right)} - \frac{\Gamma\left(\frac{s}{2} - ir\right)}{\Gamma\left(1 - \frac{s}{2} - ir\right)}\right)	\\
\label{eqn:JJr-Mellin}
& = (2\pi)^{-s} \Gamma\left(\frac{s}{2} + ir\right) \Gamma\left(\frac{s}{2} - ir\right) \cosh \pi r,	\\
\notag
\widehat{\JJ_k^{\hol}}(s) & = \pi i^{-k} (2\pi)^{-s} \frac{\Gamma\left(\frac{s + k - 1}{2}\right)}{\Gamma\left(\frac{1 - s + k}{2}\right)}	\\
\label{eqn:JJkholMellin}
& = (2\pi)^{-s} \Gamma\left(\frac{s + k - 1}{2}\right) \Gamma\left(\frac{s - k + 1}{2}\right) \cos \frac{\pi s}{2}.
\end{align}
\end{lemma}

We additionally state bounds for these Mellin transforms and their residues.

\begin{corollary}[{\cite[Corollary A.27]{HK20}}]
The functions $\widehat{\JJ_r^{\pm}}(s)$ extend meromorphically to $\C$ with simple poles at $s = 2(\pm ir - \ell)$ for $\ell \in \N_0$, where $\N_0$ denotes the nonnegative integers. For $s = \sigma + i\tau \in \C$ in bounded vertical strips at least a bounded distance away from $\{2(\pm ir - \ell) : \ell \in \N_0\}$ and for $r \in \R$,
\begin{align}
\label{eqn:JJr+Mellinbound}
\widehat{\JJ_r^{+}}(s) & \ll_{\sigma} \left(\left(1 + \left|\tau + 2r\right|\right) \left(1 + \left|\tau - 2r\right|\right)\right)^{\frac{\sigma - 1}{2}} \times \begin{dcases*}
e^{-\frac{\pi}{2}(2|r| - |\tau|)} & if $|\tau| \leq 2|r|$,	\\
1 & if $|\tau| \geq 2|r|$,
\end{dcases*}	\\
\label{eqn:JJr-Mellinbound}
\widehat{\JJ_r^{-}}(s) & \ll_{\sigma} \left(\left(1 + \left|\tau + 2r\right|\right) \left(1 + \left|\tau - 2r\right|\right)\right)^{\frac{\sigma - 1}{2}} \times \begin{dcases*}
1 & if $|\tau| \leq 2|r|$,	\\
e^{-\frac{\pi}{2}(|\tau| - 2|r|)} & if $|\tau| \geq 2|r|$.
\end{dcases*}
\end{align}
Moreover,
\begin{equation}
\label{eqn:JJrpmMellinResbound}
\Res_{s = 2(\pm ir - \ell)} \widehat{\JJ_r^{+}}(s) = (-1)^{\ell} \Res_{s = 2(\pm ir - \ell)} \widehat{\JJ_r^{-}}(s) \ll_{\ell} (1 + |r|)^{-\ell - \frac{1}{2}}.
\end{equation}

The function $\widehat{\JJ_k^{\hol}}(s)$ extends meromorphically to $\C$ with simple poles at $s = 1 - k - 2\ell$ for $\ell \in \N_0$. For $s = \sigma + i\tau \in \C$ in bounded vertical strips, at least a bounded distance away from $\{1 - k - 2\ell : \ell \in \N_0\}$,
\begin{equation}
\label{eqn:JJkholMellinbound}
\widehat{\JJ_k^{\hol}}(s) \ll_{\sigma} (k + |\tau|)^{\sigma - 1}.
\end{equation}
Moreover,
\begin{equation}
\label{eqn:JJkholMellinResbound}
\Res_{s = 1 - k - 2\ell} \widehat{\JJ_k^{\hol}}(s) = \frac{(2\pi i)^{k + 2\ell}}{\Gamma(k + \ell) \Gamma(\ell + 1)}.
\end{equation}
\end{corollary}

\subsection{Vorono\u{\i} Summation Formul\ae{}}

Along with the Kuznetsov and Petersson formul\ae{}, we also make use of the $\GL_3$ Vorono\u{\i} summation formula. This involves distinguished special functions. For these special functions, we have the following lemma, which is a straightforward consequence of the meromorphic continuation of the gamma function together with Stirling's formula.

\begin{lemma}
Let $\mu = (\mu_1,\mu_2,\mu_3) \in \C^3$ with $-1/2 < \Re(\mu_j) < 1/2$ and $\mu_1 + \mu_2 + \mu_3 = 0$, and for $s = \sigma + i\tau \in \C$, define
\begin{equation}
\label{eqn:Gscrmupm}
\Gscr_{\mu}^{\pm}(s) \coloneqq \frac{1}{2} \prod_{j = 1}^{3} G_0(s + \mu_j) \pm \frac{1}{2i} \prod_{j = 1}^{3} G_1(s + \mu_j),
\end{equation}
where
\[G_j(s) \coloneqq \frac{\Gamma_{\R}(s + j)}{\Gamma_{\R}(1 - s + j)} = 2(2\pi)^{-s} \Gamma(s) \times \begin{dcases*}
\cos \frac{\pi s}{2} & if $j = 0$,	\\
\sin \frac{\pi s}{2} & if $j = 1$,
\end{dcases*}\]
with $\Gamma_{\R}(s) \coloneqq \pi^{-s/2} \Gamma(s/2)$. Then $\Gscr_{\mu}^{\pm}(s)$ is meromorphic on $\C$ with simple poles at $s = -\mu_j - \ell$ for each $\ell \in \N_0$. Moreover, if $s$ is a bounded distance away from such a pole, we have that
\begin{equation}
\label{eqn:Gscrmubounds}
\Gscr_{\mu}^{\pm}(s) \ll_{\mu,\sigma} (1 + |\tau|)^{3\sigma - \frac{3}{2}}.
\end{equation}

Similarly, for $s \in \C$, define
\begin{equation}
\label{eqn:Gpm}
G^{\pm}(s) \coloneqq \frac{1}{2} G_0(s) \mp \frac{1}{2i} G_1(s) = (2\pi)^{-s} \Gamma(s) \exp\left(\pm \frac{i\pi s}{2}\right).
\end{equation}
Then $G^{\pm}(s)$ is meromorphic on $\C$ with simple poles at $s = -\ell$ with residue $(-1)^{\ell} i^{\pm \ell} (2\pi)^{\ell} /\ell!$ for each $\ell \in \N_0$. Moreover, if $s = \sigma + i\tau$ is a bounded distance away from such a pole, we have that
\begin{equation}
\label{eqn:Gpmbounds}
G^{\pm}(s) \ll_{\sigma} (1 + |\tau|)^{\sigma - \frac{1}{2}}.
\end{equation}
\end{lemma}

We use the $\GL_3$ Vorono\u{\i} summation formula for Hecke--Maa\ss{} cusps forms for $\SL_3(\Z)$ (i.e.\ Hecke--Maa\ss{} cusp forms on $\Zgp(\R) \SL_3(\Z) \backslash \GL_3(\R) / \Ogp(3)$).

\begin{lemma}[{Vorono\u{\i} Summation Formula \cite[Section 4]{BlK19b}}]
Given a Hecke--Maa\ss{} cusp form $F$ for $\SL_3(\Z)$ with Hecke eigenvalues $A_F(\ell,n)$, define the Vorono\u{\i} series
\begin{equation}
\label{eqn:PhiFdefeq}
\Phi_F(c,d,\ell;w) \coloneqq \sum_{n = 1}^{\infty} \frac{A_F(\ell,n)}{n^w} e\left(\frac{n\overline{d}}{c}\right),
\end{equation}
where $c,\ell \in \N$, $d \in (\Z/c\Z)^{\times}$, and $w = u + iv \in \C$. This converges absolutely for $u > 1$ and extends holomorphically to the entire complex plane. We have the functional equation
\begin{equation}
\label{eqn:GL3Voronoi}
\Phi_F(c,d,\ell;w) = \sum_{\pm} \Gscr_{\mu_F}^{\pm}(1 - w) \Xi_F(c,\pm d,\ell;-w),
\end{equation}
with $\Gscr_{\mu_F}$ as in \eqref{eqn:Gscrmupm} with $\mu = \mu_F$ equal to the spectral parameters of $F$, where
\begin{equation}
\label{eqn:XiFdefeq}
\Xi_F(c,\pm d,\ell;-w) \coloneqq c \sum_{n_1 \mid c\ell} \sum_{n_2 = 1}^{\infty} \frac{A_F(n_2,n_1)}{n_2 n_1} S\left(d\ell,\pm n_2; \frac{c\ell}{n_1}\right) \left(\frac{n_2 n_1^2}{c^3 \ell}\right)^w,
\end{equation}
which converges absolutely for $u < 0$. Moreover, we have the bounds
\begin{equation}
\label{eqn:PhiFPL}
\Phi_F(c,d,\ell;w) \ll_{F,\e} \begin{dcases*}
(c^3 \ell (1 + |\Im(w)|^3))^{\e} \max_{a \mid \ell} |A_F(a,1)| & if $\Re(w) > 1$,	\\
(c^3 \ell (1 + |\Im(w)|^3))^{\frac{1}{2}(1 - \Re(w)) + \e} \max_{a \mid \ell} |A_F(a,1)| & if $0 \leq \Re(w) \leq 1$,	\\
(c^3 \ell (1 + |\Im(w)|^3))^{\frac{1}{2}(1 - 2\Re(w)) + \e} \max_{a \mid \ell} |A_F(a,1)| & if $\Re(w) < 0$.
\end{dcases*}
\end{equation}
\end{lemma}

Here we have included the weak bounds \eqref{eqn:PhiFPL} for $\Phi_F$ in vertical strips, which follow from Stirling's formula together with the Phragm\'{e}n--Lindel\"{o}f convexity principle.

The $\GL_1$ analogue of the $\GL_3$ Vorono\u{\i} summation formula is simply the functional equation for the Hurwitz zeta function, which we record in the following form.

\begin{lemma}[{\cite[Section 2.2]{BHKM20}}]
For $c \in \N$, $d \in (\Z/c\Z)^{\times}$, and $w = u + iv \in \C$, let
\begin{equation}
\label{eqn:Phidefeq}
\Phi(c,d;w) \coloneqq \sum_{m = 1}^{\infty} \frac{e\left(\frac{dm}{c}\right)}{m^w}.
\end{equation}
This converges absolutely for $u > 1$ and extends meromorphically to all of $\C$ with a simple pole at $w = 1$ with residue $1$ if and only if $c = 1$. We have the functional equation
\[\Phi(c,d;w) = \sum_{\pm} G^{\mp}(1 - w) \Xi(c,\pm d;-w),\]
where
\begin{equation}
\label{eqn:Xidefeq}
\Xi(c,\pm d;-w) \coloneqq c^{-w} \sum_{b \in \Z/c\Z} e\left(\pm \frac{bd}{c}\right) \sum_{m = 1}^{\infty} \frac{e\left(\frac{bm}{c}\right)}{m^{1 - w}},
\end{equation}
which converges absolutely for $u < 0$. Moreover, for any $M > 0$, we have the bounds
\begin{equation}
\label{eqn:PhiPL}
\Phi(c,d;w) \ll_{u,\e} \begin{dcases*}
c^{\e} (1 + |v|)^{\e} & if $u > 1$,	\\
c^{1 - u + \e} (1 + |v|)^{\frac{1}{2}(1 - u) + \e} & if $0 \leq u \leq 1$,	\\
c^{1 - u + \e} (1 + |v|)^{\frac{1}{2} - u + \e} & if $u < 0$
\end{dcases*}
\end{equation}
for $c > 1$, while for $c = 1$ and $M \geq 1$, we have the bounds
\begin{equation}
\label{eqn:PhiPLpole}
\left(\frac{w - 1}{w + M}\right) \Phi(c,d;w) \ll_{M,u,\e} \begin{dcases*}
(1 + |v|)^{\e} & if $u > 1$,	\\
(1 + |v|)^{\frac{1}{2}(1 - u) + \e} & if $0 \leq u \leq 1$,	\\
(1 + |v|)^{\frac{1}{2} - u + \e} & if $-M < u < 0$.
\end{dcases*}
\end{equation}
\end{lemma}

Once more, the bounds \eqref{eqn:PhiPL} and \eqref{eqn:PhiPLpole} for $\Phi$ in vertical strips follow from Stirling's formula together with the Phragm\'{e}n--Lindel\"{o}f convexity principle.

\subsection{Applications of Vorono\u{\i} Summation Formul\ae{}}

During the course of the proof of $\GL_3 \times \GL_2 \leftrightsquigarrow \GL_4 \times \GL_1$ spectral reciprocity, a certain multiple sum of $\GL_3$ Fourier coefficients twisted by Kloosterman sums arises. The following lemma states that this sum is closely related to the inverse Mellin transform of $L(s,\widetilde{F})$.

\begin{lemma}
\label{lem:sumXiF}
Let $F$ be a Hecke--Maa\ss{} cusp form for $\SL_3(\Z)$ and let $\Xi_F$ be as in \eqref{eqn:XiFdefeq}. For $\ell \in \N$ and $w = u + iv \in \C$ with $u < -1/2$, we have that
\begin{equation}
\label{eqn:sumXiF}
\sum_{c \mid \ell} c^{2w - 1} \sum_{d \in (\Z/c\Z)^{\times}} e\left(\frac{d}{c}\right) \Xi_F\left(c,\pm d,\frac{\ell}{c};-w\right) = \frac{1}{2\pi i} \int_{\CC_0} L\left(1 - w + z,\widetilde{F}\right) G^{\mp}(z) \ell^{1 - w + z} \, dz,
\end{equation}
where $G^{\pm}$ is as in \eqref{eqn:Gpm} and $\CC_0$ is the contour consisting of the straight lines connecting the points $x_0 - i\infty$, $x_0 - i$, $\delta - i$, $\delta + i$, $x_0 + i$, and $x_0 + i\infty$, with $u < x_0 < -1/2$ and $\delta > 0$.
\end{lemma}

\begin{proof}
Since $u < 0$, we may replace $\Xi_F(c,\pm d, \ell/c;-w)$ by its absolutely convergent expression \eqref{eqn:XiFdefeq}. The left-hand side of \eqref{eqn:sumXiF} is therefore equal to
\begin{equation}
\label{eqn:sumXiFopen}
\ell^{-w} \sum_{n_1 \mid \ell} \sum_{n_2 = 1}^{\infty} \frac{A_F(n_2,n_1)}{n_1^{1 - 2w} n_2^{1 - w}} \sum_{c \mid \ell} \sum_{d \in (\Z/c\Z)^{\times}} e\left(\frac{d}{c}\right) S\left(\frac{d\ell}{c},\pm n_2; \frac{\ell}{n_1}\right)
\end{equation}
upon interchanging the order of summation. By opening up the Kloosterman sum, we find that
\begin{align*}
\sum_{d \in (\Z/c\Z)^{\times}} e\left(\frac{d}{c}\right) S\left(\frac{d\ell}{c},\pm n_2; \frac{\ell}{n_1}\right) & = \sum_{a \in (\Z/ \frac{\ell}{n_1}\Z)^{\times}} e\left(\pm \frac{n_1 n_2 \overline{a}}{\ell}\right) \sum_{d \in (\Z/c\Z)^{\times}} e\left(\frac{(1 + n_1 a)d}{c}\right)	\\
& = \frac{\varphi\left(\frac{\ell}{n_1}\right)}{\varphi(\ell)} \sum_{d \mid c} d \mu\left(\frac{c}{d}\right) \sum_{\substack{a \in (\Z/ \ell\Z)^{\times} \\ n_1 a \equiv -1 \hspace{-.25cm} \pmod{d}}} e\left(\pm \frac{n_1 n_2 \overline{a}}{\ell}\right),
\end{align*}
where we have inflated the sum over $a \in (\Z/\frac{\ell}{n_1}\Z)^{\times}$ to run over elements of $(\Z/\ell\Z)^{\times}$, at the cost of multiplying through by $\varphi(\frac{\ell}{n_1})/\varphi(\ell)$, and we have used the Ramanujan sum identity
\begin{equation}
\label{eqn:Ramanujan}
\sum_{d \in (\Z/c\Z)^{\times}} e\left(\frac{dn}{c}\right) = \sum_{d \mid (c,n)} d \mu\left(\frac{c}{d}\right).
\end{equation}
We insert this back into \eqref{eqn:sumXiFopen} and make the change of variables $c \mapsto cd$, so that $c \mid \frac{\ell}{d}$ and $d \mid \ell$. Since $\sum_{c \mid \frac{\ell}{d}} \mu(c)$ is $1$ if $d = \ell$ and $0$ otherwise, we deduce that \eqref{eqn:sumXiF} is equal to
\[\ell^{1 - w} \sum_{n_1 \mid \ell} \sum_{n_2 = 1}^{\infty} \frac{A_F(n_2,n_1)}{n_1^{1 - 2w} n_2^{1 - w}} \frac{\varphi\left(\frac{\ell}{n_1}\right)}{\varphi(\ell)} \sum_{\substack{a \in (\Z/\ell\Z)^{\times} \\ n_1 a \equiv -1 \hspace{-.25cm} \pmod{\ell}}} e\left(\pm \frac{n_1 n_2 \overline{a}}{\ell}\right) = \ell \sum_{n = 1}^{\infty} \frac{A_F(n,1)}{n^{1 - w}} e\left(\mp \frac{n}{\ell}\right),\]
since the congruence condition $n_1 a \equiv -1 \pmod{\ell}$ subject to the restriction $n_1 \mid \ell$ can only be met if $n_1 = 1$. We now invoke the \emph{analytic reciprocity} identity
\[e\left(\mp \frac{n}{\ell}\right) = \frac{1}{2\pi i} \int_{\CC_0} G^{\mp}(z) \left(\frac{n}{\ell}\right)^{-z} \, dz,\]
where $x_0 < -1/2$, so that this integral converges absolutely by \eqref{eqn:Gpmbounds}. Interchanging the order of integration and summation, which is valid so long as $x_0 > u$, we obtain the desired identity.
\end{proof}

Similarly, during the course of the proof of $\GL_4 \times \GL_2 \leftrightsquigarrow \GL_4 \times \GL_2$ spectral reciprocity, a certain double sum of Vorono\u{\i} series arises. The following lemma states that this sum is closely related to sums of Kloosterman sums.

\begin{lemma}
\label{lem:sumXiXiF}
Let $F$ be a Hecke--Maa\ss{} cusp form for $\SL_3(\Z)$, let $\Xi_F$ be as in \eqref{eqn:XiFdefeq}, and let $\Xi$ be as in \eqref{eqn:Xidefeq}. For $\ell \in \N$ and $w_1 = u_1 + iv_1, w_2 = u_2 + iv_2 \in \C$ with $u_1,u_2 < 0$, we have that
\begin{multline}
\label{eqn:sumXiXiF}
\sum_{c \mid \ell} c^{2w_2 - 1} \sum_{d \in (\Z/c\Z)^{\times}} \Xi\left(c,\pm_1 d;-w_1\right) \Xi_F\left(c,\pm_2 d,\frac{\ell}{c};-w_2\right)	\\
= \ell^{1 - w_1 - w_2} \sum_{n_1 \mid \ell} \sum_{m,n_2 = 1}^{\infty} \frac{A_F(n_2,n_1)}{m^{1 - w_1} n_2^{1 - w_2} n_1^{1 - 2w_2}} S\left(m, \mp_1 \pm_2 n_2; \frac{\ell}{n_1}\right).
\end{multline}
\end{lemma}

\begin{proof}
Since $u_1, u_2 < 0$, we may replace both Vorono\u{\i} series on the left-hand side of \eqref{eqn:sumXiXiF} with their absolutely convergent expressions and interchange the order of summation and integration, which gives
\begin{multline}
\label{eqn:sumXiXiF2}
\ell^{-w_2} \sum_{c \mid \ell} c^{-w_1} \sum_{m = 1}^{\infty} \frac{1}{m^{1 - w_1}} \sum_{n_1 \mid \ell} \sum_{n_2 = 1}^{\infty} \frac{A_F(n_2,n_1)}{n_2^{1 - w_2} n_1^{1 - 2w_2}}	\\
\times \sum_{b \in \Z/c\Z} e\left(\frac{bm}{c}\right) \sum_{d \in (\Z/c\Z)^{\times}} e\left(\pm_1 \frac{bd}{c}\right) S\left(\frac{d\ell}{c},\pm_2 n_2; \frac{\ell}{n_1}\right).
\end{multline}
Opening up the Kloosterman sum and using the Ramanujan sum identity \eqref{eqn:Ramanujan}, we find that
\begin{multline*}
\sum_{b \in \Z/c\Z} e\left(\frac{bm}{c}\right) \sum_{d \in (\Z/c\Z)^{\times}} e\left(\pm_1 \frac{bd}{c}\right) S\left(\frac{d\ell}{c},\pm_2 n_2; \frac{\ell}{n_1}\right)	\\
= \sum_{d \mid c} d \mu\left(\frac{c}{d}\right) \sum_{a \in (\Z / \frac{\ell}{n_1} \Z)^{\times}} e\left(\mp_1 \pm_2 \frac{n_1 n_2 \overline{a}}{\ell}\right) \sum_{\substack{b \in \Z/c\Z \\ b \equiv n_1 a \hspace{-.25cm} \pmod{d}}} e\left(\frac{bm}{c}\right)
\end{multline*}
upon making the change of variables $a \mapsto \mp_1 a$. Making the change of variables $b \mapsto n_1 a + bd$, where now $b \in \Z/\frac{c}{d}\Z$, we see that
\[\sum_{\substack{b \in \Z/c\Z \\ b \equiv n_1 a \hspace{-.25cm} \pmod{d}}} e\left(\frac{bm}{c}\right) = \begin{dcases*}
\frac{c}{d} e\left(\frac{mn_1 a}{c}\right) & if $\frac{c}{d} \mid m$,	\\
0 & otherwise.
\end{dcases*}\]
We insert this identity into \eqref{eqn:sumXiXiF2} and make the change of variables $c \mapsto cd$ and $m \mapsto cm$, so that $c \mid \frac{\ell}{d}$ and $d \mid \ell$. Since $\sum_{c \mid \frac{\ell}{d}} \mu(c)$ is $1$ if $d = \ell$ and $0$ otherwise, we obtain the desired identity.
\end{proof}

\section{\texorpdfstring{$\mathrm{GL}_3 \times \mathrm{GL}_2 \leftrightsquigarrow \mathrm{GL}_4 \times \mathrm{GL}_1$}{GL\9040\203 \80\327 GL\9040\202 \9041\224 GL\9040\204 \80\327 GL\9040\201} Spectral Reciprocity}

We show the following form of spectral reciprocity: a $\GL_2$ moment of $\GL_3 \times \GL_2$ Rankin--Selberg $L$-functions is equal to a main term plus a dual moment, which is a $\GL_1$ moment of $\GL_4$ $L$-functions that factorise as the product of $\GL_3$ and $\GL_1$ $L$-functions. The proof uses the Kuznetsov and Petersson formul\ae{} and the $\GL_3$ Vorono\u{\i} summation formula in the guise of \hyperref[lem:sumXiF]{Lemma \ref*{lem:sumXiF}}.

\begin{theorem}
\label{thm:3x2reciprocity}
Let $h^{\pm}(t)$ be functions that are even, holomorphic in the horizontal strip $|\Im(t)| \leq 1/2 + \delta$ for some $\delta > 0$, and satisfy $h^{\pm}(t) \ll (1 + |t|)^{-4}$, and let $h^{\hol} : 2\N \to \C$ be a sequence satisfying $h^{\hol}(k) \ll k^{-4}$. Suppose additionally that the functions
\begin{align}
\label{eqn:H+fromKscr}
H^{+}(x) & \coloneqq (\Kscr^{+} h^{+})(x) + (\Kscr^{\hol} h^{\hol})(x),	\\
\label{eqn:H-fromKscr}
H^{-}(x) & \coloneqq (\Kscr^{-} h^{-})(x),
\end{align}
where $\Kscr^{\pm}$ and $\Kscr^{\hol}$ are as in \eqref{eqn:NscrpmKscrpmdefeq} and \eqref{eqn:KscrholJJkholdefeq}, are such that their Mellin transforms $\widehat{H^{\pm}}(s) \coloneqq \int_{0}^{\infty} H^{\pm}(x) x^s \, \frac{dx}{x}$ are holomorphic in the strip $-4 < \Re(s) < 1$, in which they satisfy the bounds $\widehat{H^{\pm}}(s) \ll (1 + |\Im(s)|)^{\Re(s) - 4}$. Let $F$ be a self-dual Hecke--Maa\ss{} cusp form for $\SL_3(\Z)$. Then
\begin{multline}
\label{eqn:3x2identity}
\sum_{\pm} \sum_{f \in \BB_0} \frac{L\left(\frac{1}{2},F \otimes f\right)}{L(1,\ad f)} h^{\pm}(t_f) + \sum_{\pm} \frac{1}{2\pi} \int_{-\infty}^{\infty} \frac{L\left(\frac{1}{2} + it,F\right) L\left(\frac{1}{2} - it,F\right)}{\zeta(1 + 2it) \zeta(1 - 2it)} h^{\pm}(t) \, dt	\\
+ \sum_{f \in \BB_{\hol}} \frac{L\left(\frac{1}{2},F \otimes f\right)}{L(1,\ad f)} h^{\hol}(k_f)	\\
= L(1,F) \sum_{\pm} \int_{-\infty}^{\infty} h^{\pm}(r) \, d_{\spec}r + L(1,F) \sum_{\substack{k = 4 \\ k \equiv 0 \hspace{-.25cm} \pmod{4}}}^{\infty} \frac{k - 1}{\pi^2} h^{\hol}(k)	\\
+ \frac{1}{2\pi} \int_{-\infty}^{\infty} L\left(\frac{1}{2} + it,F\right) \zeta\left(\frac{1}{2} - it\right) \HH_{\mu_F}(t) \, dt,
\end{multline}
where for $0 < \sigma < 1$,
\begin{equation}
\label{eqn:HHmuFt}
\HH_{\mu_F}(t) \coloneqq \frac{1}{2\pi i} \int_{\sigma - i\infty}^{\sigma + i\infty} \sum_{\pm_1,\pm_2} \widehat{H^{\pm_1}}(s) \Gscr_{\mu_F}^{\pm_2}\left(\frac{1 - s}{2}\right) G^{\mp_1 \pm_2}\left(\frac{s}{2} + it\right) \, ds
\end{equation}
with $\Gscr_{\mu}^{\pm}$ as in \eqref{eqn:Gscrmupm} and $G^{\pm}$ as in \eqref{eqn:Gpm}.
\end{theorem}

\begin{remark}
The assumptions on the decay of $h^{\pm}(t)$, $h^{\hol}(k)$, and $\widehat{H^{\pm}}(s)$ are sufficient but certainly not necessary for the identity \eqref{eqn:3x2identity} to hold; with more work, one can impose weaker assumptions on $h^{\pm}(t)$ and $h^{\hol}(k)$.
\end{remark}

\hyperref[thm:3x2reciprocity]{Theorem \ref*{thm:3x2reciprocity}} is a cuspidal analogue of Motohashi's formula, as discussed in \hyperref[sect:GL3xGL2intro]{Section \ref*{sect:GL3xGL2intro}}; indeed, if $F$ is replaced by a minimal parabolic Eisenstein series, then the identity \eqref{eqn:3x2identity} is Motohashi's formula (with additional degenerate terms appearing on the right-hand side due to the non-cuspidality of $F$). Motohashi's formula has previously been generalised to allow for character twists \cite{BHKM20,Kan22,Pet15} as well as in the more general setting of $L$-functions over number fields \cite{Nel19b,Wu22}; \hyperref[thm:3x2reciprocity]{Theorem \ref*{thm:3x2reciprocity}} gives a new generalisation in a different direction.

The identity \eqref{eqn:3x2identity} has been independently proven by Kwan \cite[Theorem 1.1]{Kwa24} via different means, albeit with more stringent conditions imposed on the triple of test functions $(h^+,h^-,h^{\hol})$, which are insufficiently flexible for our desired applications\footnote{In particular, Kwan's result only allows for the possibility that $h^{\hol}(k) = h^{+}(t) - h^{-}(t) = 0$ and that $h^{+}(t) + h^{-}(t)$ is the product of $\cosh \pi t \prod_{\pm} \Gamma_{\R}(\frac{1}{2} \pm it) \prod_{\pm_1,\pm_2} \Gamma_{\R}(\frac{1}{2} \pm_1 it \pm_2 2it_g)$ and an even function that is holomorphic in a sufficiently wide horizontal strip in which it decays exponentially. On the other hand, Kwan's proof is valid more generally for arbitrarily Hecke-Maa\ss{} cusp forms on $\SL_3(\Z)$, not just self-dual forms. The proof that we give of the identity \eqref{eqn:3x2identity} also remains valid for non-self-dual forms (with an additional term appearing on the right-hand side of \eqref{eqn:3x2identity}), though the ensuing identity is no longer relevant for the applications that we have in mind.}. Approximate forms of the identity \eqref{eqn:3x2identity} (due to the usage approximate functional equations) go back to work of Li \cite[Theorem 1.1]{Li11}, who showed that with a particular choice of triple $(h^+,h^-,h^{\hol})$, one can prove subconvex bounds for $L(1/2,F \otimes f)$ and $L(1/2 + it,F)$. The state of the art in this regard is the pair of subconvex bounds \cite[Corollary 1.2]{LNQ23} (cf.\ \cite[Theorem 1.1]{GHLN24})
\begin{equation}
\label{eqn:LNQsubconvex}
L\left(\frac{1}{2},F \otimes f\right) \ll_{F,\e} t_f^{\frac{6}{5} + \e}, \qquad L\left(\frac{1}{2} + it,F\right) \ll_{F,\e} (1 + |t|)^{\frac{3}{5} + \e}.
\end{equation}

The existence of the identity \eqref{eqn:3x2identity} answers in the affirmative a speculation of Lin, Nunes, and Qi \cite[Section 1.4]{LNQ23}, for one can choose a triple of test functions $(h^{+},h^{-},h^{\hol})$ in such a way that $h^{-}(t)$ localises to the interval $[T - U,T + U]$; upon determining the support and size of the transform $\HH_{\mu_F}(t)$, one recovers an upper bound roughly of the form
\begin{multline*}
\sum_{\substack{f \in \BB_0 \\ T - U \leq t_f \leq T + U}} \frac{L\left(\frac{1}{2},F \otimes f\right)}{L(1,\ad f)} + \frac{1}{2\pi} \int\limits_{T - U \leq |t| \leq T + U} \frac{L\left(\frac{1}{2} + it,F\right) L\left(\frac{1}{2} - it,F\right)}{\zeta(1 + 2it) \zeta(1 - 2it)} \, dt	\\
\ll_F TU + U \int_{-\frac{T}{U}}^{\frac{T}{U}} \left|L\left(\frac{1}{2} + it,F\right) \zeta\left(\frac{1}{2} - it\right)\right| \, dt.
\end{multline*}
In conjunction with the Montgomery--Vaughan mean value theorem for Dirichlet polynomials \cite[Corollary 3]{MV74} (see \hyperref[lem:integralsecondmomentbounds]{Lemma \ref*{lem:integralsecondmomentbounds}} below), this can be used to give an alternate proof of the subconvex bounds \eqref{eqn:LNQsubconvex}.

Before proceeding to the proof of \hyperref[thm:3x2reciprocity]{Theorem \ref*{thm:3x2reciprocity}}, we must include the following weak bounds for the second moment of the Riemann zeta function and for the $L$-function of a Hecke--Maa\ss{} cusp form $F$ for $\SL_3(\Z)$. These bounds will be required in the proof of \hyperref[thm:3x2reciprocity]{Theorem \ref*{thm:3x2reciprocity}}.

\begin{lemma}
\label{lem:integralsecondmomentbounds}
We have the bounds
\begin{align}
\label{eqn:zeta2ndmoment}
\int_{U}^{2U} |\zeta(\sigma + it)|^2 \, dt & \ll_{\e} U^{1 + \e} \quad \text{for $\sigma \geq \tfrac{1}{2}$},	\\
\label{eqn:LsF2ndmoment}
\int_{U}^{2U} |L(\sigma + it,F)|^2 \, dt & \ll_{F,\e} \begin{dcases*}
U^{3(1 - \sigma) + \e} & if $\frac{1}{2} \leq \sigma \leq \frac{2}{3}$,	\\
U^{1 + \e} & if $\sigma \geq \frac{2}{3}$,
\end{dcases*}
\end{align}
where $F$ is a Hecke--Maa\ss{} cusp form for $\SL_3(\Z)$.
\end{lemma}

Under the assumption of the generalised Lindel\"{o}f hypothesis, the bound \eqref{eqn:zeta2ndmoment} is essentially optimal but \eqref{eqn:LsF2ndmoment} falls shy of the conjecturally optimal upper bound $O_{F,\e}(U^{1 + \e})$ when $1/2 \leq \sigma < 2/3$.

\begin{proof}
This follows by using the approximate functional equation \cite[Theorem 5.3]{IK04} to write $\zeta(\sigma + it)$ and $L(\sigma + it,F)$ in terms of Dirichlet polynomials and then invoking the Montgomery--Vaughan mean value theorem for Dirichlet polynomials \cite[Corollary 3]{MV74}.
\end{proof}

\begin{proof}[Proof of {\hyperref[thm:3x2reciprocity]{Theorem \ref*{thm:3x2reciprocity}}}]
Let $w = u + iv$ be a complex variable. Given $f \in \BB_0$ or $f \in \BB_{\hol}$ with Hecke eigenvalues $\lambda_f(n)$, the $\GL_3 \times \GL_2$ Rankin--Selberg $L$-function $L(w,F \otimes f)$ has the Dirichlet series expansion
\[L(w,F \otimes f) = \sum_{\ell,n = 1}^{\infty} \frac{A_F(\ell,n) \lambda_f(n)}{\ell^{2w} n^w}\]
for $u > 1$. Similarly, let $E(z,1/2 + it)$ be the real analytic Eisenstein series on $\Gamma \backslash \Hb$ with Hecke eigenvalues $\lambda(n,t) \coloneqq \sum_{ab = n} a^{it} b^{-it}$; then for $u > 1$, we have the identity
\[L(w + it,F) L(w - it,F) = \sum_{\ell,n = 1}^{\infty} \frac{A_F(\ell,n) \lambda(n,t)}{\ell^{2w} n^w}.\]

With this in mind, we initially assume that $5/4 < u \leq 4/3$ and multiply the Kuznetsov and Petersson formul\ae{}, \eqref{eqn:Kuznetsovformula} and \eqref{eqn:Peterssonformula}, with $m = 1$ by $A_F(\ell,n) \ell^{-2w} n^{-w}$, then sum over $\ell,n \in \N$. Adding the Petersson formula to the sum of the same sign and opposite sign Kuznetsov formul\ae{}, we obtain the identity
\begin{multline}
\label{eqn:preVoronoi}
\sum_{\pm} \sum_{f \in \BB_0} \epsilon_f^{\frac{1 \mp 1}{2}} \frac{L(w,F \otimes f)}{L(1,\ad f)} h^{\pm}(t_f) + \sum_{\pm} \frac{1}{2\pi} \int_{-\infty}^{\infty} \frac{L(w + it,F) L(w - it,F)}{\zeta(1 + 2it) \zeta(1 - 2it)} h^{\pm}(t) \, dt	\\
+ \sum_{f \in \BB_{\hol}} \frac{L(w,F \otimes f)}{L(1,\ad f)} h^{\hol}(k_f)	\\
= L(2w,\widetilde{F}) \Nscr^{+} h^{+} + L(2w,\widetilde{F}) \Nscr^{\hol} h^{\hol}	\\
+ \sum_{c,\ell = 1}^{\infty} \frac{1}{c\ell^{2w}} \frac{1}{2\pi i} \int_{\sigma_0 - i\infty}^{\sigma_0 + i\infty} \sum_{\pm} \widehat{H^{\pm}}(s) c^s \sum_{n = 1}^{\infty} \frac{A_F(\ell,n)}{n^{\frac{s}{2} + w}} S(1,\pm n;c) \, ds.
\end{multline}
Here $\Nscr^+$ and $\Nscr^{\hol}$ are as in \eqref{eqn:NscrpmKscrpmdefeq} and \eqref{eqn:Nscrholdefeq}, while $S(m,n;c)$ denotes the Kloosterman sum, as in \eqref{eqn:Kloostermandefeq}. This identity is valid for $2 - 2u < \sigma_0 < -1/2$, which is a nonempty region provided that $u > 5/4$. We have used the Mellin inversion formula to write
\[H^{\pm}(x) = \frac{1}{2\pi i} \int_{\sigma_0 - i\infty}^{\sigma_0 + i\infty} \widehat{H^{\pm}}(s) x^{-s} \, ds\]
for $-4 < \sigma_0 < 1$, since in this range we have the bounds
\begin{equation}
\label{eqn:Ksdecay}
\widehat{H^{\pm}}(\sigma + i\tau) \ll_{\sigma} (1 + |\tau|)^{\sigma - 4}.
\end{equation}
By the Weil bound for Kloosterman sums, the sum over $c$ converges absolutely since $\sigma_0 < -1/2$, while the sum over $n$ converges absolutely since $\sigma_0 > 2 - 2u$; the sum over $\ell$ converges since $u > 1$.

In anticipation of future simplifications, we write $\ell' = c\ell$, relabel $\ell'$ as $\ell$, and open up the Kloosterman sum, so that the last term on the right-hand side of \eqref{eqn:preVoronoi} is
\begin{equation}
\label{eqn:integrand}
\sum_{\ell = 1}^{\infty} \frac{1}{\ell^{2w}} \frac{1}{2\pi i} \int_{\sigma_0 - i\infty}^{\sigma_0 + i\infty} \sum_{\pm} \widehat{H^{\pm}}(s) \sum_{c \mid \ell} c^{s + 2w - 1} \sum_{d \in (\Z/c\Z)^{\times}} e\left(\frac{d}{c}\right) \Phi_F\left(c,\pm d,\frac{\ell}{c};\frac{s}{2} + w\right) \, ds,
\end{equation}
where the Vorono\u{\i} series $\Phi_F$ is as in \eqref{eqn:PhiFdefeq}.

The left-hand side of \eqref{eqn:preVoronoi} extends holomorphically to $u \geq 1/2$, since the convexity bound for $L(w,F \otimes f)$ and $L(w + it,F)$ together with the assumptions $h^{\pm}(r) \ll (1 + |r|)^{-4}$ and $h^{\hol}(k) \ll k^{-4}$ ensure that each term on the left-hand side is absolutely convergent for all $u \geq 1/2$. The holomorphic extension to $w = 1/2$ is precisely the left-hand side of \eqref{eqn:3x2identity}, since if $f \in \BB_0$, the root number of $L(w,F \otimes f)$ is $\epsilon_f$, and hence $L(1/2,F \otimes f) = 0$ when $\epsilon_f = -1$.

We shall show that the right-hand side of \eqref{eqn:preVoronoi} extends holomorphically to $w = 1/2$ and is equal to the right-hand side of \eqref{eqn:3x2identity}. To begin, we shift the contour of integration of \eqref{eqn:integrand} to $\Re(s) = \sigma_1$ with $-4 < \sigma_1 < -1 - 2u$, which is a nonempty region since $u < 3/2$; due to the bounds \eqref{eqn:PhiFPL} for $\Phi_F$ and \eqref{eqn:Ksdecay} for the Mellin transforms of $H^{\pm}$, the ensuing integral is absolutely convergent. We then use the Vorono\u{\i} summation formula \eqref{eqn:GL3Voronoi}. Via the identity \eqref{eqn:sumXiF}, we deduce that for $\Re(s) = \sigma_1$, the integrand in \eqref{eqn:integrand} is equal to
\[\sum_{\pm_1, \pm_2} \widehat{H^{\pm_1}}(s) \Gscr_{\mu_F}^{\pm_2}\left(1 - \frac{s}{2} - w\right) \frac{1}{2\pi i} \int_{\CC_0} L\left(1 - \frac{s}{2} - w + z,\widetilde{F}\right) G^{\mp_1 \pm_2}(z) \ell^{1 - \frac{s}{2} - w + z} \, dz\]
with $\CC_0$ the contour defined in \hyperref[lem:sumXiF]{Lemma \ref*{lem:sumXiF}} such that $\sigma_1/2 + u < x_0 < -1/2$ and $0 < \delta < \sigma_1/2 + 3u - 2$. With this replacing the integrand of \eqref{eqn:integrand} and with the contour of integration shifted to $\Re(s) = \sigma_1$, the resulting expression is absolutely convergent, and so we may interchange the order of summation and integration. Thus we see that \eqref{eqn:integrand} is equal to
\begin{multline*}
\frac{1}{2\pi i} \int_{\sigma_1 - i\infty}^{\sigma_1 + i\infty} \sum_{\pm_1,\pm_2} \widehat{H^{\pm_1}}(s) \Gscr_{\mu_F}^{\pm_2}\left(1 - \frac{s}{2} - w\right)	\\
\times \frac{1}{2\pi i} \int_{\CC_0} L\left(1 - \frac{s}{2} - w + z,\widetilde{F}\right) \zeta\left(\frac{s}{2} + 3w - z - 1\right) G^{\mp_1 \pm_2}(z) \, dz \, ds.
\end{multline*}

We are ensured absolute convergence of this double integral by the bounds \eqref{eqn:Ksdecay} for the Mellin transform of $H^{\pm}$, \eqref{eqn:Gscrmubounds} for $\Gscr_{\mu_F}^{\pm}(s)$, and \eqref{eqn:Gpmbounds} for $G^{\pm}(s)$. Thus we may make the change of variables $z \mapsto s/2 + 3w - z - 3/2$ and interchange the order of integration, yielding
\[\frac{1}{2\pi i} \int_{x_1 - i\infty}^{x_1 + i\infty} L\left(2w - z - \frac{1}{2},\widetilde{F}\right) \zeta\left(\frac{1}{2} + z\right) \HH_{\mu_F}(w,z) \, dz,\]
where
\begin{equation}
\label{eqn:HHmuFwz}
\HH_{\mu_F}(w,z) \coloneqq \frac{1}{2\pi i} \int_{\CC_1} \sum_{\pm_1,\pm_2} \widehat{H^{\pm_1}}(s) \Gscr_{\mu_F}^{\pm_2}\left(1 - \frac{s}{2} - w\right) G^{\mp_1 \pm_2}\left(\frac{s}{2} + 3w - z - \frac{3}{2}\right) \, ds.
\end{equation}
Here $x_1 = - x_0 + \sigma_1/2 + 3u - 3/2$, so that $1 < x_1 < 2u - 3/2$, and $\CC_1$ is the contour consisting of the straight lines connecting the points $\sigma_1 - i\infty$, $2z - 6w + 2x_0 + 3 - 2i$, $2z - 6w + 2\delta + 3 - 2i$, $2z - 6w + 2\delta + 3 + 2i$, $2z - 6w + 2x_0 + 3 + 2i$, and $\sigma_1 + i\infty$. Finally, we may straighten the inner contour of integration from $\CC_1$ to the vertical line $\Re(s) = \sigma_2$ with $2x_1 - 6u + 3 < \sigma_2 < 2 - 2u$.

We now begin the process of analytically continuing this expression to $w = 1/2$. Suppose that $w$ lies in a compact subset $K$ of the closed vertical strip $1/2 \leq \Re(w) \leq 4/3$. Then by Stirling's formula and \eqref{eqn:Ksdecay}, the integrand in \eqref{eqn:HHmuFwz} is meromorphic as a function of $s \in \C$ with simple poles at $s = 2z - 6w + 3 - 2\ell$ for $\ell \in \N_0$ with residues of size $O_{F,K}((1 + |y|)^{- x - 4 + \ell})$ for $z = x + iy$, while for $s = \sigma + i\tau$ a bounded distance away from such a pole, the integrand is
\[O_{F,\sigma,K}\left((1 + |\tau|)^{-\frac{1}{2}(\sigma + 6u + 5)} (1 + |\tau - 2y|)^{\frac{1}{2}(\sigma + 6u - 2x - 3)}\right).\]
Thus by shifting the contour of integration of the inner integral to the left to $\Re(s) = \sigma_3$ with $\sigma_3 = 2x_1 - 6u + 3 + \alpha$ for $6u - 2x_1 - 7 < \alpha < -8/3$, which picks up residues at $s = 2z - 6w + 3 - 2\ell$ for $\ell \in \N_0$ from the poles of $G^{\mp_1 \pm_2}(s/2 + 3w - z - 3/2)$, and breaking up the integral into the three ranges $|\tau| \leq |y|$, $|y| \leq |\tau| \leq 3|y|$, and $|\tau| \geq 3|y|$, we find that $\HH_{\mu_F}(w,z)$ is holomorphic as a function of $z \in \C$ and satisfies the bound
\begin{equation}
\label{eqn:innerintbound}
\HH_{\mu_F}(w,z) \ll_{F,K,\alpha} (1 + |y|)^{\frac{\alpha}{2}} + (1 + |y|)^{-x - \frac{\alpha}{2} - 4}.
\end{equation}

Next, we observe that by the Cauchy--Schwarz inequality and the bounds \eqref{eqn:zeta2ndmoment} and \eqref{eqn:LsF2ndmoment},
\begin{multline}
\label{eqn:MVTbound}
\int_{U}^{2U} \sum_{\pm} \left|L\left(2w - x \mp iy - \frac{1}{2},\widetilde{F}\right) \zeta\left(\frac{1}{2} + x \pm iy\right)\right| \, dy	\\
\ll_{F,K,\e} \begin{dcases*}
U^{1 + \e} & if $0 \leq x \leq 2u - \frac{7}{6}$,	\\
U^{\frac{11}{4} - 3u + \frac{3x}{2} + \e} & if $\max\left\{2u - \frac{7}{6},0\right\} \leq x \leq 2u - 1$.
\end{dcases*}
\end{multline}

Thus for $w \in K$, we may shift the outer contour to $\Re(z) = x_2$ with $x_2 = 2u - 1$, since the bounds \eqref{eqn:innerintbound} and \eqref{eqn:MVTbound} ensure that the resulting integral converges absolutely. For $u < 3/4$, this introduces an additional term
\[L\left(2w - 1,\widetilde{F}\right) \HH_{\mu_F}\left(w,\frac{1}{2}\right)\]
arising from the residue at $z = 1/2$ of the outer integral, since $\zeta(1/2 + z)$ has a simple pole at $z = 1/2$ with residue $1$. Shifting the contour of integration in \eqref{eqn:HHmuFwz} from $\Re(s) = \sigma_2$ with $1 - 2u < \sigma_2 < 2 - 2u$ to $\Re(s) = \sigma_4$ with $-4 < \sigma_4 < 2 - 6u$, which picks up residues at $s = 4 - 6w$ and $s = 2 - 6w$ from the poles of $G^{\mp_1 \pm_2}(s/2 + 3w - 2)$, we see that this additional term can be written as
\begin{multline*}
2 L(2 - 2w,F) \sum_{\pm} \widehat{H^{\pm}}(4 - 6w) - L\left(2w - 1,\widetilde{F}\right) \sum_{\pm_1,\pm_2} \frac{i^{\pm_1 \mp_2 1}}{\pi} \widehat{H^{\pm_1}}(2 - 6w) \Gscr_{\mu_F}^{\pm_2}(2w)	\\
+ L\left(2w - 1,\widetilde{F}\right) \frac{1}{2\pi i} \int_{\sigma_4 - i\infty}^{\sigma_4 + i\infty} \sum_{\pm_1,\pm_2} \widehat{H^{\pm_1}}(s) \Gscr_{\mu_F}^{\pm_2}\left(1 - \frac{s}{2} - w\right) G^{\mp_1 \pm_2}\left(\frac{s}{2} + 3w - 2\right) \, ds
\end{multline*}
via the functional equation $L(2 - 2w,F) = \sum_{\pm} \Gscr_{\mu_F}^{\pm}(2w - 1) L(2w - 1,\widetilde{F})$. All three of these terms extend holomorphically to $w = 1/2$; furthermore, the holomorphic extensions to $w = 1/2$ of the second and third terms vanish since $L(0,\widetilde{F}) = 0$ due to the self-duality of $F$, while the holomorphic extension to $w = 1/2$ of the first term is equal to
\[2 L(1,F) \sum_{\pm} \widehat{H^{\pm}}(1) = L(1,F) \int_{-\infty}^{\infty} h^{-}(r) \, d_{\spec}r + L(1,F) \sum_{\substack{k = 2 \\ k \equiv 0 \hspace{-.25cm} \pmod{2}}}^{\infty} \frac{k - 1}{2\pi^2} i^{-k} h^{\hol}(k).\]
Here we have used the fact that
\begin{align}
\label{eqn:widehatH+1}
\widehat{H^+}(1) & = \int_{-\infty}^{\infty} \widehat{\JJ_r^+}(1) h^+(r) \, d_{\spec}r + \sum_{\substack{k = 2 \\ k \equiv 0 \hspace{-.25cm} \pmod{2}}}^{\infty} \frac{k - 1}{2 \pi^2} \widehat{\JJ_k^{\hol}}(1) h^{\hol}(k),	\\
\label{eqn:widehatH-1}
\widehat{H^-}(1) & = \int_{-\infty}^{\infty} \widehat{\JJ_r^-}(1) h^-(r) \, d_{\spec}r
\end{align}
by \eqref{eqn:NscrpmKscrpmdefeq}, \eqref{eqn:KscrholJJkholdefeq}, \eqref{eqn:H+fromKscr}, and \eqref{eqn:H-fromKscr}, together with the fact that
\[\widehat{\JJ_r^+}(1) = 0, \qquad \widehat{\JJ_k^{\hol}}(1) = \frac{i^{-k}}{2}, \qquad \widehat{\JJ_r^-}(1) = \frac{1}{2}\]
from \eqref{eqn:JJr+Mellin}, \eqref{eqn:JJr-Mellin}, and \eqref{eqn:JJkholMellin}. Finally, the main term
\[\frac{1}{2\pi i} \int_{x_2 - i\infty}^{x_2 + i\infty} L\left(2w - z - \frac{1}{2},\widetilde{F}\right) \zeta\left(\frac{1}{2} + z\right) \HH_{\mu_F}(w,z) \, dz\]
extends holomorphically to $w = 1/2$, where it becomes
\[\frac{1}{2\pi} \int_{-\infty}^{\infty} L\left(\frac{1}{2} + it,\widetilde{F}\right) \zeta\left(\frac{1}{2} - it\right) \HH_{\mu_F}(t) \, dt\]
with $\HH_{\mu_F}(t)$ as in \eqref{eqn:HHmuFt} upon writing $z = -it$.
\end{proof}

\section{\texorpdfstring{$\mathrm{GL}_4 \times \mathrm{GL}_2 \leftrightsquigarrow \mathrm{GL}_4 \times \mathrm{GL}_2$}{GL\9040\204 \80\327 GL\9040\202 \9041\224 GL\9040\204 \80\327 GL\9040\202} Spectral Reciprocity}

We show the following form of spectral reciprocity: a $\GL_2$ moment of $\GL_4 \times \GL_2$ Rankin--Selberg $L$-functions is equal to a main term plus a dual moment, which is a $\GL_2$ moment of $\GL_4 \times \GL_2$ $L$-functions. The proof uses the Kuznetsov and Petersson formul\ae{} and the $\GL_3$ Vorono\u{\i} summation formula in the guise of \hyperref[lem:sumXiXiF]{Lemma \ref*{lem:sumXiXiF}}. The archetypal version of this form of spectral reciprocity is a reciprocity formula for the fourth moment of $L(1/2,f)$ due to Kuznetsov \cite{Kuz89,Kuz99}, though the initial proof was incomplete in parts and was subsequently completed by Motohashi \cite{Mot03}. With the goal of proving \hyperref[prop:fourranges]{Proposition \ref*{prop:fourranges}}, we prove a new form of $\GL_4 \times \GL_2 \leftrightsquigarrow \GL_4 \times \GL_2$ spectral reciprocity. In place of $L(1/2,f)^4$, our identity instead involves $L(1/2,f) L(1/2,F \otimes f)$, where $F$ is a self-dual Hecke--Maa\ss{} cusp form for $\SL_3(\Z)$.

\begin{theorem}
\label{thm:4x2reciprocity}
Let $h^{\pm}(t)$ be functions that are even, holomorphic in the horizontal strip $|\Im(t)| \leq 1/2 + \delta$ for some $\delta > 0$, and satisfy $h^{\pm}(t) \ll (1 + |t|)^{-5}$, and let $h^{\hol} : 2\N \to \C$ be a sequence satisfying $h^{\hol}(k) \ll k^{-5}$. Suppose additionally that the functions $H^{\pm}$ given by \eqref{eqn:H+fromKscr} and \eqref{eqn:H-fromKscr} are such that their Mellin transforms $\widehat{H^{\pm}}(s) \coloneqq \int_{0}^{\infty} H^{\pm}(x) x^s \, \frac{dx}{x}$ are holomorphic in the strip $-5 < \Re(s) < 1$, in which they satisfy the bounds $\widehat{H^{\pm}}(s) \ll (1 + |\Im(s)|)^{\Re(s) - 5}$. Let $F$ be a self-dual Hecke--Maa\ss{} cusp form for $\SL_3(\Z)$. Then
\begin{multline}
\label{eqn:4x2identity}
\sum_{\pm} \sum_{f \in \BB_0} \frac{L\left(\frac{1}{2},f\right) L\left(\frac{1}{2},F \otimes f\right)}{L(1,\ad f)} h^{\pm}(t_f)	\\
+ \sum_{\pm} \frac{1}{2\pi} \int_{-\infty}^{\infty} \frac{\zeta\left(\frac{1}{2} + it\right) \zeta\left(\frac{1}{2} - it\right) L\left(\frac{1}{2} + it,F\right) L\left(\frac{1}{2} - it,F\right)}{\zeta(1 + 2it) \zeta(1 - 2it)} h^{\pm}(t) \, dt	\\
+ \sum_{f \in \BB_{\hol}} \frac{L\left(\frac{1}{2},f\right) L\left(\frac{1}{2},F \otimes f\right)}{L(1,\ad f)} h^{\hol}(k_f)	\\
= \frac{L(1,F)^2}{\zeta(2)} \sum_{\pm} \int_{-\infty}^{\infty} h^{\pm}(r) \, d_{\spec}r + \frac{L(1,F)^2}{\zeta(2)} \sum_{\substack{k = 4 \\ k \equiv 0 \hspace{-.25cm} \pmod{4}}}^{\infty} \frac{k - 1}{\pi^2} h^{\hol}(k)	\\
+ \sum_{\pm} \sum_{f \in \BB_0} \frac{L\left(\frac{1}{2},f\right) L\left(\frac{1}{2},F \otimes f\right)}{L(1,\ad f)} \widetilde{h}^{\pm}(t_f)	\\
+ \sum_{\pm} \frac{1}{2\pi} \int_{-\infty}^{\infty} \frac{\zeta\left(\frac{1}{2} + it\right) \zeta\left(\frac{1}{2} - it\right) L\left(\frac{1}{2} + it,F\right) L\left(\frac{1}{2} - it,F\right)}{\zeta(1 + 2it) \zeta(1 - 2it)} \widetilde{h}^{\pm}(t) \, dt	\\
+ \sum_{f \in \BB_{\hol}} \frac{L\left(\frac{1}{2},f\right) L\left(\frac{1}{2},F \otimes f\right)}{L(1,\ad f)} \widetilde{h}^{\hol}(k_f),
\end{multline}
where for $0 < \sigma < 1$,
\begin{align}
\label{eqn:tildehpmdefeq}
\widetilde{h}^{\pm}(t) & = \frac{1}{2\pi i} \int_{\sigma - i\infty}^{\sigma + i\infty} \sum_{\pm_1,\pm_2} \widehat{H^{\pm_1}}(s) \widehat{\JJ_t^{\pm}}(s) G^{\pm_2}\left(\frac{1 - s}{2}\right) \Gscr_{\mu_F}^{\pm \pm_1 \pm_2}\left(\frac{1 - s}{2}\right) \, ds,	\\
\label{eqn:tildehholdefeq}
\widetilde{h}^{\hol}(k) & = \frac{1}{2\pi i} \int_{\sigma - i\infty}^{\sigma + i\infty} \sum_{\pm_1,\pm_2} \widehat{H^{\pm_1}}(s) \widehat{\JJ_k^{\hol}}(s) G^{\pm_2}\left(\frac{1 - s}{2}\right) \Gscr_{\mu_F}^{\pm_1 \pm_2}\left(\frac{1 - s}{2}\right) \, ds,
\end{align}
with $\widehat{\JJ_t^{\pm}}$ and $\widehat{\JJ_k^{\hol}}$ as in \eqref{eqn:JJr+Mellin}, \eqref{eqn:JJr-Mellin}, and \eqref{eqn:JJkholMellin}, $\Gscr_{\mu}^{\pm}$ as in \eqref{eqn:Gscrmupm}, and $G^{\pm}$ as in \eqref{eqn:Gpm}.
\end{theorem}

\begin{remark}
The assumptions on the decay of $h^{\pm}(t)$, $h^{\hol}(k)$, and $\widehat{H^{\pm}}(s)$ are sufficient but certainly not necessary for the identity \eqref{eqn:4x2identity} to hold; with more work, one can impose weaker assumptions on $h^{\pm}(t)$ and $h^{\hol}(k)$ for this identity to remain valid.
\end{remark}

\hyperref[thm:4x2reciprocity]{Theorem \ref*{thm:4x2reciprocity}} is a cuspidal analogue of Kuznetsov's formula for the fourth moment of $L(1/2,f)$, as discussed in \hyperref[sect:GL4xGL2intro]{Section \ref*{sect:GL4xGL2intro}}; indeed, if $F$ is replaced by a minimal parabolic Eisenstein series, then the identity \eqref{eqn:4x2identity} is Kuznetsov's formula (with additional degenerate terms appearing on the right-hand side due to the non-cuspidality of $F$). The authors have previously proven an analogue of the identity \eqref{eqn:4x2identity} with $F$ replaced by a maximal parabolic Eisenstein series induced from a dihedral Hecke--Maa\ss{} cusp form \cite[Proposition 7.1]{HK20}, with applications towards $L^4$-norm asymptotic formul\ae{} for dihedral Maa\ss{} cusp forms. More generally, Blomer, Li, and Miller have proven a \emph{completely} cuspidal version of \eqref{eqn:4x2identity} for the first moment of $L(1/2,G \otimes f)$, where $G$ is a Hecke--Maa\ss{} cusp form for $\SL_4(\Z)$ \cite[Theorem 1]{BLM19}.

We briefly mention that there additionally exist \emph{level}-aspect versions of $\GL_4 \times \GL_2 \leftrightsquigarrow \GL_4 \times \GL_2$ spectral reciprocity, which have striking applications towards subconvexity; see, in particular, \cite{AK18,BlK19a,BlK19b,Nun23,Zac19,Zac21}. These level-aspect versions have also been generalised to higher rank spectral reciprocity formul\ae{} \cite{JN21,Mia21}.

\begin{proof}[Proof of {\hyperref[thm:4x2reciprocity]{Theorem \ref*{thm:4x2reciprocity}}}]
Let $w_1 = u_1 + iv_1$ and $w_2 = u_2 + iv_2$ be complex variables. We initially assume that $5/4 < u_1 < u_2 \leq 4/3$ and multiply the Kuznetsov and Petersson formul\ae{}, \eqref{eqn:Kuznetsovformula} and \eqref{eqn:Peterssonformula}, by $A_F(\ell,n) m^{-w_1} \ell^{-2w_2} n^{-w_2}$, then sum over $\ell,m,n \in \N$. Adding the Petersson formula to the sum of the same sign and opposite sign Kuznetsov formul\ae{}, we obtain the identity
\begin{multline}
\label{eqn:preVoronoi2}
\sum_{\pm} \sum_{f \in \BB_0} \epsilon_f^{\frac{1 \mp 1}{2}} \frac{L(w_1,f) L(w_2,F \otimes f)}{L(1,\ad f)} h^{\pm}(t_f)	\\
+ \sum_{\pm} \frac{1}{2\pi} \int_{-\infty}^{\infty} \frac{\zeta(w_1 + it) \zeta(w_1 - it) L(w_2 + it,F) L(w_2 - it,F)}{\zeta(1 + 2it) \zeta(1 - 2it)} h^{\pm}(t) \, dt	\\
+ \sum_{f \in \BB_{\hol}} \frac{L(w_1,f) L(w_2,F \otimes f)}{L(1,\ad f)} h^{\hol}(k_f)	\\
= \frac{L(2w_2,\widetilde{F}) L(w_1 + w_2,F)}{\zeta(w_1 + 3w_2)} \Nscr^{+} h^{+} + \frac{L(2w_2,\widetilde{F}) L(w_1 + w_2,F)}{\zeta(w_1 + 3w_2)} \Nscr^{\hol} h^{\hol}	\\
+ \sum_{c,\ell = 1}^{\infty} \frac{1}{c\ell^{2w_2}} \frac{1}{2\pi i} \int_{\sigma_0 - i\infty}^{\sigma_0 + i\infty} \sum_{\pm} \widehat{H^{\pm}}(s) c^s \sum_{m,n = 1}^{\infty} \frac{A_F(\ell,n)}{m^{\frac{s}{2} + w_1} n^{\frac{s}{2} + w_2}} S(m,\pm n;c) \, ds.
\end{multline}
For the diagonal terms, we have used the Hecke relations \cite[Theorem 6.4.11]{Gol06}
\begin{equation}
\label{eqn:GL3Heckerelations}
A_F(\ell,n) = \sum_{d \mid (\ell,n)} \mu(d) A_F\left(\frac{\ell}{d},1\right) A_F\left(1,\frac{n}{d}\right)
\end{equation}
and made the change of variables $\ell \mapsto d\ell$ and $n \mapsto dn$ in order to see that
\[\sum_{\ell,n = 1}^{\infty} \frac{A_F(\ell,n)}{\ell^{2w_2} n^{w_1 + w_2}} = \frac{L(2w_2,\widetilde{F}) L(w_1 + w_2,F)}{\zeta(w_1 + 3w_2)}.\]
The identity \eqref{eqn:preVoronoi2} is valid for $2 - 2u_1 < \sigma_0 < -1/2$, which is a nonempty region provided that $u_1 > 5/4$. Here we have used the Mellin inversion formula to write
\[H^{\pm}(x) = \frac{1}{2\pi i} \int_{\sigma_0 - i\infty}^{\sigma_0 + i\infty} \widehat{H^{\pm}}(s) x^{-s} \, ds\]
for $-5 < \sigma_0 < 1$, since in this range we have the bounds
\begin{equation}
\label{eqn:Ksdecay2}
\widehat{H^{\pm}}(\sigma + i\tau) \ll_{\sigma} (1 + |\tau|)^{\sigma - 5}.
\end{equation}
By the Weil bound for Kloosterman sums, the sum over $c$ converges absolutely since $\sigma_0 < -1/2$, while the sums over $m,n$ converge absolutely since $\sigma_0 > 2 - 2u_1 > 2 - 2u_2$; the sum over $\ell$ converges since $u_2 > 1$.

In anticipation of future simplifications, we write $\ell' = c\ell$, relabel $\ell'$ as $\ell$, and open up the Kloosterman sum, so that the last term on the right-hand side of \eqref{eqn:preVoronoi2} is
\[\sum_{\ell = 1}^{\infty} \frac{1}{\ell^{2w_2}} \frac{1}{2\pi i} \int_{\sigma_0 - i\infty}^{\sigma_0 + i\infty} \sum_{\pm} \widehat{H^{\pm}}(s) \sum_{c \mid \ell} c^{s + 2w_2 - 1} \sum_{d \in (\Z/c\Z)^{\times}} \Phi\left(c,d;\frac{s}{2} + w_1\right) \Phi_F\left(c,\pm d,\frac{\ell}{c};\frac{s}{2} + w_2\right) \, ds,\]
where the Vorono\u{\i} series $\Phi$ and $\Phi_F$ are as in \eqref{eqn:Phidefeq} and \eqref{eqn:PhiFdefeq}.

The left-hand side of \eqref{eqn:preVoronoi2} extends holomorphically to $u_1,u_2 \geq 1/2$, since the convexity bounds for $L(w_1,f)$, $L(w_2,F \otimes f)$, $\zeta(w_1 + it)$, and $L(w_2 + it,F)$ together with the assumptions $h^{\pm}(r) \ll (1 + |r|)^{-5}$ and $h^{\hol}(k) \ll k^{-5}$ ensure that the left-hand side converges for all $u_1,u_2 \geq 1/2$. The holomorphic extension to $w_1 = w_2 = 1/2$ is precisely the left-hand side of the desired identity \eqref{eqn:4x2identity}, since if $f \in \BB_0$, the root number of $L(w,F \otimes f)$ is $\epsilon_f$, and hence $L(1/2,F \otimes f) = 0$ when $\epsilon_f = -1$. Note that for $\Re(w_1) < 1$, additionally polar divisors arise via shifting the contour in the integration over $t \in \R$ in the second term of \eqref{eqn:preVoronoi2}, since the integrand has poles at $t = \pm i(1 - w_1)$. The holomorphic extension of these polar divisors vanishes when $w_1 = w_2 = 1/2$, however, since $L(0,F) = 0$ as $F$ is self-dual.

We shall show that the right-hand side of \eqref{eqn:preVoronoi2} extends holomorphically to $w_1 = w_2 = 1/2$ and is equal to the right-hand side of the desired identity. To begin, we shift the contour of integration of the third term on the right-hand side of \eqref{eqn:preVoronoi2} to $\Re(s) = \sigma_1$ with $5/2 - u_1 - 3u_2 < \sigma_1 < -2u_2$, which is a nonempty region since $u_2 > u_1 > 5/4$; due to the bounds \eqref{eqn:PhiFPL}, \eqref{eqn:PhiPL}, and \eqref{eqn:PhiPLpole} the ensuing integral is absolutely convergent. The only pole that we encounter along the way is at $s = 2(1 - w_1)$ when $c = 1$. For $u_2 > u_1$, the resulting residue is
\begin{multline*}
\sum_{\ell = 1}^{\infty} \frac{1}{\ell^{2w_2}} 2 \sum_{\pm} \widehat{H^{\pm}}(2 - 2w_1) \Phi_F\left(1,\pm 1,\ell; 1 - w_1 + w_2\right)	\\
= \frac{2 L(2w_2,\widetilde{F}) L(1 - w_1 + w_2,F)}{\zeta(1 - w_1 + 3w_2)} \sum_{\pm} \widehat{H^{\pm}}(2 - 2w_1).
\end{multline*}
This extends holomorphically to $w_1 = w_2 = 1/2$, where it is equal to
\[\frac{L(1,\widetilde{F}) L(1,F)}{\zeta(2)} \int_{-\infty}^{\infty} h^{-}(r) \, d_{\spec}r + \frac{L(1,\widetilde{F}) L(1,F)}{\zeta(2)} \sum_{\substack{k = 2 \\ k \equiv 0 \hspace{-.25cm} \pmod{2}}}^{\infty} \frac{k - 1}{2\pi^2} i^{-k} h^{\hol}(k),\]
where we have used \eqref{eqn:widehatH+1} and \eqref{eqn:widehatH-1} for the values of $\widehat{H^{\pm}}(1)$. Combined with the first two terms on the right-hand side of \eqref{eqn:preVoronoi2} evaluated at $w_1 = w_2 = 1/2$, this yields the first two terms on the right-hand side of \eqref{eqn:4x2identity}, 

Now we wish to re-express the remaining Kloosterman term where $\sigma_0$ has been replaced by $\sigma_1$, with $5/2 - u_1 - 3u_2 < \sigma_1 < -2u_2$. We apply the Vorono\u{\i} summation formul\ae{} to both Vorono\u{\i} series, yielding
\begin{multline*}
\sum_{\ell = 1}^{\infty} \frac{1}{\ell^{2w_2}} \frac{1}{2\pi i} \int_{\sigma_1 - i\infty}^{\sigma_1 + i\infty} \sum_{\pm_1} \widehat{H^{\pm_1}}(s) \sum_{\pm_2} G^{\mp_2}\left(1 - \frac{s}{2} - w_1\right) \sum_{\pm_3} \Gscr_{\mu_F}^{\pm_3}\left(1 - \frac{s}{2} - w_2\right)	\\
\times \sum_{c \mid \ell} c^{s + 2w_2 - 1} \sum_{d \in (\Z/c\Z)^{\times}} \Xi\left(c,\pm_2 d;-\frac{s}{2} - w_1\right) \Xi_F\left(c,\pm_1 \pm_3 d,\frac{\ell}{c};-\frac{s}{2} - w_2\right) \, ds.
\end{multline*}
Inserting the identity \eqref{eqn:sumXiXiF}, interchanging the order of summation and integration, and making the change of variables $\ell \mapsto \ell n_1$, we find that this is equal to
\[\sum_{\pm} \sum_{m,n_1,n_2 = 1}^{\infty} \frac{A_F(n_2,n_1)}{m^{\frac{3w_2 - w_1}{2}} n_2^{\frac{w_1 + w_2}{2}} n_1^{w_1 + w_2}} \sum_{\ell = 1}^{\infty} \frac{S(m,\pm n_2;\ell)}{\ell} \widetilde{H}_{w_1,w_2}^{\pm}\left(\frac{\sqrt{mn_2}}{\ell}\right),\]
where
\[\widetilde{H}_{w_1,w_2}^{\pm}(x) \coloneqq \frac{1}{2\pi i} \int_{\sigma_1 - i\infty}^{\sigma_1 + i\infty} \sum_{\pm_1,\pm_2} \widehat{H^{\pm_1}}(s) G^{\pm_2}\left(1 - \frac{s}{2} - w_1\right) \Gscr_{\mu_F}^{\pm \pm_1 \pm_2}\left(1 - \frac{s}{2} - w_2\right) x^{s + w_1 + 3w_2 - 2} \, ds.\]
Here the integral converges since $\sigma_1 > -2 - u_1 - 3u_2$ via the bounds \eqref{eqn:Ksdecay2}, \eqref{eqn:Gpmbounds}, and \eqref{eqn:Gscrmubounds}, the sum over $m$ converges since $\sigma_1 < -2u_1$, the sum over $n_1$ converges since $u_1 + u_2 > 1$, the sum over $n_2$ converges since $\sigma_1 < -2u_2$, and the sum over $\ell$ converges since $\sigma_1 > 5/2 - u_1 - 3u_2$ via the Weil bound for Kloosterman sums.

Now we apply the Kloosterman summation formula, \eqref{eqn:Kloostermanformula}. In order to do so, we require that there exists some $\delta > 0$ such that for each $j \in \{0,1,2,3\}$,
\[x^j \frac{d^j}{dx^j} \widetilde{H}_{w_1,w_2}^{\pm}(x) \ll \begin{dcases*}
x^{\frac{1}{2} + \delta} & for $x \leq 1$,	\\
x^{-1 - \delta} & for $x \geq 1$.
\end{dcases*}\]
This can readily be seen by differentiating under the integral sign and shifting the contour of integration to $\Re(s) = \sigma_2$ with $\sigma_2 = 5/2 + \delta - u_1 - 3u_2$ for $x \leq 1$ and shifting the contour of integration to $\Re(s) = \sigma_3$ with $\sigma_3 = 1 - \delta - u_1 - 3u_2$ for $x \geq 1$.

As in \cite[Proofs of Theorems 2.3 and 2.4]{Mot97}, these conditions on $\widetilde{H}_{w_1,w_2}^{\pm}$ imply that
\begin{align*}
\widetilde{h}_{w_1,w_2}^{\pm}(t) & \coloneqq (\Lscr^{\pm} \widetilde{H}_{w_1,w_2}^{\pm})(t) \ll (1 + |t|)^{-\frac{5}{2} + \delta},	\\
\widetilde{h}_{w_1,w_2}^{\hol}(k) & \coloneqq (\Lscr^{\hol} \widetilde{H}_{w_1,w_2}^{+})(k) \ll k^{-\frac{5}{2} + \delta}.
\end{align*}
Via the Weyl law, this ensures the absolute convergence of the ensuing expression, and so we may interchange the order of summation in order to arrive at
\begin{multline}
\label{eqn:postVoronoi}
\sum_{\pm} \sum_{f \in \BB_0} \epsilon_f^{\frac{1 \mp 1}{2}} \frac{L\left(\frac{3w_2 - w_1}{2}, f\right) L\left(\frac{w_1 + w_2}{2}, \widetilde{F} \otimes f\right)}{L(1, \ad f)} \widetilde{h}_{w_1,w_2}^{\pm}(t_f)	\\
+ \sum_{\pm} \frac{1}{2\pi} \int_{-\infty}^{\infty} \frac{\zeta\left(\frac{3w_2 - w_1}{2} + it\right) \zeta\left(\frac{3w_2 - w_1}{2} - it\right) L\left(\frac{w_1 + w_2}{2} + it, \widetilde{F}\right) L\left(\frac{w_1 + w_2}{2} - it, \widetilde{F}\right)}{\zeta(1 + 2it) \zeta(1 - 2it)} \widetilde{h}_{w_1,w_2}^{\pm}(t) \, dt	\\
+ \sum_{f \in \BB_{\hol}} \frac{L\left(\frac{3w_2 - w_1}{2}, f\right) L\left(\frac{w_1 + w_2}{2}, \widetilde{F} \otimes f\right)}{L(1, \ad f)} \widetilde{h}_{w_1,w_2}^{\hol}(k_f).
\end{multline}

It remains to holomorphically extend this expression to $w_1 = w_2 = 1/2$, which gives us the last three terms on the right-hand side of \eqref{eqn:4x2identity}. Note that for $3\Re(w_2) - \Re(w_1) < 2$, additionally polar divisors arise via shifting the contour in the integration over $t \in \R$ in the second term of \eqref{eqn:postVoronoi}, since the integrand has poles at $t = \pm i(1 - 3(w_2 - w_1)/2)$. The holomorphic extension of these polar divisors vanishes when $w_1 = w_2 = 1/2$, however, since $L(0,\widetilde{F}) = 0$ as $F$ is self-dual. We are left with showing that the functions $\widetilde{h}_{w_1,w_2}^{\pm}(t)$ and $\widetilde{h}_{w_1,w_2}^{\hol}(k)$ extend holomorphically to $w_1 = w_2 = 1/2$ and have sufficiently rapid decay to ensure the absolute convergence of the three terms in \eqref{eqn:postVoronoi}.

We first observe that by Parseval's identity for the Mellin transform, we have that
\begin{align}
\label{eqn:tildehpmMellin}
\widetilde{h}_{w_1,w_2}^{\pm}(t) & = \frac{1}{2\pi i} \int_{\sigma - i\infty}^{\sigma + i\infty} \sum_{\pm_1,\pm_2} \widehat{H^{\pm_1}}(s) \widehat{\JJ_t^{\pm}}\left(s + w_1 + 3w_2 - 2\right)	\\
\notag
& \hspace{2.5cm} \times G^{\pm_2}\left(1 - \frac{s}{2} - w_1\right) \Gscr_{\mu_F}^{\pm \pm_1 \pm_2}\left(1 - \frac{s}{2} - w_2\right) \, ds,	\\
\notag
\widetilde{h}_{w_1,w_2}^{\hol}(k) & = \frac{1}{2\pi i} \int_{\sigma - i\infty}^{\sigma + i\infty} \sum_{\pm_1,\pm_2} \widehat{H^{\pm_1}}(s) \widehat{\JJ_k^{\hol}}\left(s + w_1 + 3w_2 - 2\right)	\\
\notag
& \hspace{2.5cm} \times G^{\pm_2}\left(1 - \frac{s}{2} - w_1\right) \Gscr_{\mu_F}^{\pm_1 \pm_2}\left(1 - \frac{s}{2} - w_2\right) \, ds
\end{align}
for $2 - u_1 - 3u_2 < \sigma < 2 - 2u_2$. As a function of $s = \sigma + i\tau$, the functions
\[\sum_{\pm_2} G^{\pm_2}\left(1 - \frac{s}{2} - w_1\right) \Gscr_{\mu_F}^{\pm_1 \pm_2}\left(1 - \frac{s}{2} - w_2\right)\]
extend meromorphically to the entire complex plane and are holomorphic in the left-half plane $\sigma < 2 - 2u_2$. By \eqref{eqn:Gpmbounds} and \eqref{eqn:Gscrmubounds}, we have the bounds
\begin{equation}
\label{eqn:Gpm2bounds}
\sum_{\pm_2} G^{\pm_2}\left(1 - \frac{s}{2} - w_1\right) \Gscr_{\mu_F}^{\pm_1 \pm_2}\left(1 - \frac{s}{2} - w_2\right) \ll_{F,K,\sigma}\left((1 + |\tau|)^{2 - 2\sigma - u_1 - 3u_2}\right)
\end{equation}
for $w_1,w_2$ in a compact subset $K$ of the vertical strip $1/2 \leq \Re(w_1),\Re(w_2) \leq 4/3$. Recalling the bounds \eqref{eqn:JJr+Mellinbound}, \eqref{eqn:JJr-Mellinbound}, and \eqref{eqn:JJkholMellinbound} for the Mellin transforms of $\JJ_t^{\pm}$ and $\JJ_k^{\hol}$, we see that the integrands are integrable along the vertical line $\Re(s) = \sigma$ provided that $\sigma > -5$. In particular, as functions of the complex variables $w_1 = u_1 + iv_1$ and $w_2 = u_2 + iv_2$, these integrals extend holomorphically to $w_1 = w_2 = 1/2$.

By the convexity bound and the Weyl law, in order to ensure the absolute convergence of each of the three terms in \eqref{eqn:postVoronoi}, it suffices to show that there exists some $\delta > 0$ such that $\widetilde{h}_{w_1,w_2}^{\pm}(t) \ll_{F,K} (1 + |t|)^{u_1 + 3u_2 - 6 - \delta}$ and that $\widetilde{h}_{w_1,w_2}^{\hol}(k) \ll_{F,K} k^{u_1 + 3u_2 - 6 - \delta}$. For the former, we shift the contour of integration in \eqref{eqn:tildehpmMellin} to $\Re(s) = \sigma_4$ with $\max\{-5,-3u_1 - 9u_2 + 3\} < \sigma_4 < -u_1 - 3u_2 + 1$. Due to the poles of $\widehat{\JJ_t^{\pm}}(s + w_1 + 3w_2 - 2)$, this picks up residues at $s = -w_1 - 3w_2 + 2 + 2(\pm it - \ell)$ of size $O_{F,K}((1 + |t|)^{\ell - 11/2})$ for $\ell \in \N_0$ by \eqref{eqn:JJrpmMellinResbound}, \eqref{eqn:Ksdecay2}, and \eqref{eqn:Gpm2bounds}. We then break up the ensuing integral into the three ranges $|\tau| \leq |t|$, $|t| \leq |\tau| \leq 3|t|$, and $|\tau| \geq 3|t|$. By \eqref{eqn:JJr+Mellinbound}, \eqref{eqn:JJr-Mellinbound}, \eqref{eqn:Ksdecay2}, and \eqref{eqn:Gpm2bounds}, the former and latter contributions are $O_{F,K}((1 + |t|)^{-5})$, while the middle range is $O_{F,K}((1 + |t|)^{-(\sigma_4 + u_1 + 3u_2 + 9)/2})$. Our assumption on $\sigma_4$ then ensures that this decays sufficiently rapidly. The same method (using the bounds \eqref{eqn:JJkholMellinbound} in place of \eqref{eqn:JJr+Mellinbound} or \eqref{eqn:JJr-Mellinbound}) yields the desired bounds for $\widetilde{h}_{w_1,w_2}^{\hol}(k)$.
\end{proof}

\section{Test Functions and Transforms for the Short Initial Range}

Our treatment of the short initial range requires the usage of $\GL_3 \times \GL_2 \leftrightsquigarrow \GL_4 \times \GL_1$ and $\GL_4 \times \GL_2 \leftrightsquigarrow \GL_4 \times \GL_2$ spectral reciprocity in the guises of \hyperref[thm:3x2reciprocity]{Theorems \ref*{thm:3x2reciprocity}} and \ref{thm:4x2reciprocity}. In order to apply these two forms of spectral reciprocity, we must choose triples of test functions $(h^{+},h^{-},h^{\hol})$ that localise to dyadic intervals $[T,2T]$. We subsequently bound the associated transforms $\HH_{\mu_F}$ as in \eqref{eqn:HHmuFt} and $(\widetilde{h}^{+},\widetilde{h}^{-},\widetilde{h}^{\hol})$ as in \eqref{eqn:tildehpmdefeq} and \eqref{eqn:tildehholdefeq}. Extra care must be undertaken in producing these bounds due to the hybrid nature of the problem at hand: we must obtain bounds that are uniform in both the dyadic parameter $T$ and the spectral parameter $t_g$.

\subsection{Test Functions}

We define two triples of test functions $(h^{+},h^{-},h^{\hol})$:
\begin{align}
\label{eqn:firsttriple}
h^{+}(t) &= 0, & h^{-}(t) & = e^{-\frac{t^2}{T^2}} \prod_{j = 1}^{M} \left(\frac{t^2 + \left(j - \frac{1}{2}\right)^2}{T^2}\right)^2, & h^{\hol}(k) & = 0,	\\
\label{eqn:secondtriple}
h^{+}(t) & = (\Lscr^{+} H^{+})(t), & h^{-}(t) & = 0, & h^{\hol}(k) & = (\Lscr^{\hol} H^{+})(k).
\end{align}
Here $\Lscr^{+}$ and $\Lscr^{\hol}$ are the transforms given by \eqref{eqn:Lscrdefeq}, while $H^{+}$ is chosen to be the function
\begin{equation}
\label{eqn:H+defeq}
H^{+}(x) \coloneqq \sinh^{M - 1}\left(\frac{1}{T}\right) (4\pi x)^M e^{-4\pi x \sinh \left(\frac{1}{T}\right)},
\end{equation}
which depends on auxiliary parameters $M \in \N$ and $T > 0$; for both \eqref{eqn:firsttriple} and \eqref{eqn:secondtriple}, $M$ is a fixed positive integer and $T > M$. These are chosen to localise to dyadic regions: \eqref{eqn:firsttriple} is evidently constructed such that $h^{-}(t)$ localises to the intervals $[-2T,-T] \cup [T,2T]$, while we shall presently show that \eqref{eqn:secondtriple} is constructed such that $i^k h^{\hol}(k)$ localises to the interval $[T,2T]$.

There are of course plenty of other choices of triples of test functions that localise to dyadic intervals. In order for these test functions to be admissible for \hyperref[thm:3x2reciprocity]{Theorems \ref*{thm:3x2reciprocity}} and \ref{thm:4x2reciprocity}, however, it is necessary that the Mellin transforms of the functions $H^{\pm}(x)$ given by \eqref{eqn:H+fromKscr} and \eqref{eqn:H-fromKscr} are holomorphic in a sufficiently wide vertical strip in which they decay sufficiently rapidly. This feature is by no means automatic and is crucial behind our choices of test functions.

\begin{lemma}
\label{lem:H+}
Let $H^+$ be as in \eqref{eqn:H+defeq} with $M \in \N$ and $T > M$.
\begin{enumerate}[leftmargin=*,label=\textup{(\arabic*)}]
\item\label{lem:H+Mellinbound} For $s = \sigma + i\tau$, the Mellin transform $\widehat{H^{+}}(s)$ of $H^{+}$ is holomorphic for $\sigma > -M$, in which it satisfies the bound
\[\widehat{H^{+}}(s) \ll_{\sigma,M} T^{1 + \sigma} (1 + |\tau|)^{\sigma + M - \frac{1}{2}} e^{-\frac{\pi}{2}|\tau|}.\]
\item\label{lem:H+Lhol} The transform $(\Lscr^{\hol} H^{+})(k)$ is such that for $k \in 2\N$,
\begin{enumerate}[label=\textup{(\alph*)}]
\item\label{lem:H+Lholbound} $(\Lscr^{\hol} H^{+})(k) \ll_M (k/T)^{M - 1} e^{-k/T}$,
\item\label{lem:H+Lholpos} $i^k (\Lscr^{\hol} H^{+})(k) > 0$ for $k > M$,
\item\label{lem:H+Lholasymp} $i^k (\Lscr^{\hol} H^{+})(k) \asymp_M 1$ for $k \asymp T$ with $k > M$.
\end{enumerate}
\item\label{lem:H+L+bound} The transform $(\Lscr^{+} H^{+})(r)$ is such that for $r \in \R \cup i(-\frac{1}{2},\frac{1}{2})$,
\[(\Lscr^{+} H^{+})(r) \ll \left(\frac{1 + |r|}{T}\right)^{M - 1} e^{-\pi|r|}.\]
\end{enumerate}
\end{lemma}

\begin{proof}\hspace{1em}
\begin{enumerate}[leftmargin=*,label=\textup{(\arabic*)}]
\item By making the change of variables $x \mapsto (4\pi \sinh(1/T))^{-1} x$, we see that
\[\widehat{H^{+}}(s) \coloneqq \int_{0}^{\infty} H^{+}(x) x^s \, \frac{dx}{x} = 4\pi \left(4\pi\sinh\left(\frac{1}{T}\right)\right)^{-s - 1} \Gamma(s + M),\]
which is holomorphic for $\Re(s) > -M$. The desired bound for $\widehat{H^{+}}(s)$ then follows from Stirling's formula.
\item By \cite[6.621.1 and 6.621.4]{GR15}, we have that
\begin{align*}
(\Lscr^{\hol} H^{+})(k) 
& = 2\pi i^{-k} (-1)^{M - 1} \sinh^{M - 1}\left(\frac{1}{T}\right) \left. \left(\sech y \frac{d}{dy}\right)^{M - 1}\right|_{y = \frac{1}{T}} \left(e^{-(k - 1)y} \sech y\right)	\\
& = 2\pi i^{-k} \tanh^{M - 1}\left(\frac{1}{T}\right) \sech\left(\frac{1}{T}\right) \Gamma(k - 1 + M) \mathsf{P}_{M - 1}^{1 - k}\left(\tanh\left(\frac{1}{T}\right)\right).
\end{align*}
where $\mathsf{P}_{\alpha}^{\beta}(x)$ denotes the Ferrers function of the first kind. By \cite[8.714.2]{GR15}, $i^k (\Lscr^{\hol} H^{+})(k)$ is equal to
\[\tanh^{M - 1}\left(\frac{1}{T}\right) \sech^k\left(\frac{1}{T}\right) \frac{\pi \Gamma(2k - 1)}{2^{k - 3} \Gamma(k)\Gamma(k - M)} \int_{0}^{\infty} \frac{t^{k + M - 1}}{\left(1 + 2t \tanh \left(\frac{1}{T}\right) + t^2\right)^{k - \frac{1}{2}}} \, \frac{dt}{t}\]
for $k > M$, and in particular is positive in this range. Finally, by \cite[8.704]{GR15},
\begin{multline*}
i^k (\Lscr^{\hol} H^{+})(k) = 2\pi \sinh^{M - 1}\left(\frac{1}{T}\right) e^{-k/T}	\\
\times \sum_{j = 0}^{M - 1} (-1)^j \binom{M - 1}{j} \frac{(M - 1 + j)! (k - 2 + M)!}{(M - 1)! (k - 1 + j)!} 2^{-j} e^{-(j - 1)/T} \cosh^{M - j - 2}\left(\frac{1}{T}\right).
\end{multline*}
Since $(k - 2 + M)! / (k - 1 + j)! \asymp_M k^{M - 1 - j}$, the above quantity is $O_M((k/T)^{M - 1} e^{-k/T})$ for all $k \in 2\N$ and is $\asymp_M 1$ for $k \asymp T$.
\item Similarly, by \cite[6.621.1 and 6.621.4]{GR15},
\begin{align*}
(\Lscr^{+} H^{+})(r) & = \frac{2\pi}{\sinh \pi r} (-1)^{M - 1} \sinh^{M - 1}\left(\frac{1}{T}\right) \left. \left(\sech y \frac{d}{dy}\right)^{M - 1}\right|_{y = \frac{1}{T}} \left(\sech y \sin (2ry)\right)	\\
& = \frac{\pi i}{\sinh \pi r} \tanh^{M - 1}\left(\frac{1}{T}\right) \sech\left(\frac{1}{T}\right) \sum_{\pm} \pm \Gamma(M \pm 2ir) \mathsf{P}_{M - 1}^{\mp 2ir}\left(\tanh\left(\frac{1}{T}\right)\right).
\end{align*}
By \cite[8.714.2]{GR15},
\begin{multline*}
\sum_{\pm} \pm \Gamma(M \pm 2ir) \mathsf{P}_{M - 1}^{\mp 2ir}\left(\tanh \left(\frac{1}{T}\right)\right)	\\
= \sum_{j = 0}^{M - 1} (-1)^j \binom{M - 1}{j} \frac{(M - 1 + j)!}{(M - 1)!} 2^{-j} e^{-j/T} \sech^j \left(\frac{1}{T}\right) \sum_{\pm} \pm e^{\mp 2i r/T} \frac{\Gamma(M \pm 2ir)}{\Gamma(1 + j \pm 2ir)}.
\end{multline*}
The desired bound then holds from the fact that
\[\sum_{\pm} \pm e^{\mp 2i r/T} \frac{\Gamma(M \pm 2ir)}{\Gamma(1 + j \pm 2ir)} \ll_M (1 + |r|)^{M - 1 - j}.\qedhere\]
\end{enumerate}
\end{proof}

\subsection{\texorpdfstring{$\mathrm{GL}_4 \times \mathrm{GL}_1$}{GL\9040\204 \80\327 GL\9040\201} Transforms}

Next, we determine the behaviour of $\HH_{\mu_F}(t)$ as in \eqref{eqn:HHmuFt} with $(h^{+},h^{-},h^{\hol})$ either of the triples of test functions \eqref{eqn:firsttriple} or \eqref{eqn:secondtriple}.

\begin{lemma}
\label{lem:HHmuFbounds}
Let $g$ be a Hecke--Maa\ss{} cusp form on $\Gamma \backslash \Hb$ with spectral parameter $t_g$. Let $(h^{+},h^{-},h^{\hol})$ be either of the triples of test functions \eqref{eqn:firsttriple} or \eqref{eqn:secondtriple} with $M > 50$ and $T \geq 1$. Then for $F = \ad g$ and $\HH_{\mu_F}(t)$ as in \eqref{eqn:HHmuFt}, we have that
\begin{equation}
\label{eqn:HHmuFbounds}
\HH_{\mu_F}(t) \ll_M \begin{dcases*}
\frac{T^2}{t_g} & for $|t| \leq \frac{t_g^2}{T^2} + 1$,	\\
\frac{T}{|t|^{1/2}} \left(\frac{T^2 |t|}{t_g^2}\right)^{-\frac{M}{2}} & for $|t| \geq \frac{t_g^2}{T^2} + 1$.
\end{dcases*}
\end{equation}
\end{lemma}

Combining the bounds in \hyperref[lem:HHmuFbounds]{Lemma \ref*{lem:HHmuFbounds}} with the $\GL_3 \times \GL_2 \leftrightsquigarrow \GL_4 \times \GL_1$ spectral reciprocity formula obtained in \hyperref[thm:3x2reciprocity]{Theorem \ref*{thm:3x2reciprocity}}, we may deduce an identity roughly of the form
\begin{multline*}
\sum_{\substack{f \in \BB_0 \\ T \leq t_f \leq 2T}} \frac{L\left(\frac{1}{2},\ad g \otimes f\right)}{L(1,\ad f)} + \frac{1}{2\pi} \int\limits_{T \leq |t| \leq 2T} \left|\frac{L\left(\frac{1}{2} + it,\ad g\right)}{\zeta(1 + 2it)}\right|^2 \, dt + \sum_{\substack{f \in \BB_{\hol} \\ T \leq k_f \leq 2T}} \frac{L\left(\frac{1}{2},\ad g \otimes f\right)}{L(1,\ad f)}	\\
\approx T^2 + \frac{T^2}{t_g} \int_{-\frac{t_g^2}{T^2} - 1}^{\frac{t_g^2}{T^2} + 1} L\left(\frac{1}{2} + it,\ad g\right) \zeta\left(\frac{1}{2} - it\right) \, dt.
\end{multline*}
In \hyperref[prop:upperbounds]{Proposition \ref*{prop:upperbounds}}, we use this identity to produce upper bounds for each of the terms on the left-hand side.

\begin{proof}[Proof of {\hyperref[lem:HHmuFbounds]{Lemma \ref*{lem:HHmuFbounds}}}]
For the triple of test functions \eqref{eqn:firsttriple}, we have that for $s = \sigma + i\tau$ with $-M/2 < \sigma < M/2$,
\[\widehat{H^{-}}(s) \ll_{\sigma,M} T^{1 + \sigma} (1 + |\tau|)^{-M}\]
by \cite[Lemma 4]{BLM19}, while $\widehat{H^{+}}(s) = 0$, where $H^+$ and $H^-$ are as in \eqref{eqn:H+fromKscr} and \eqref{eqn:H-fromKscr}. Similarly, for the triple of test functions \eqref{eqn:secondtriple}, the Sears--Titchmarsh inversion formula \cite[Appendix B.5]{Iwa02} implies that $(\Kscr^{+} h^{+})(x) + (\Kscr^{\hol} h^{\hol})(x) = H^{+}(x)$, so that
\[\widehat{H^{+}}(s) \ll_{\sigma,M} T^{1 + \sigma} (1 + |\tau|)^{\sigma + M - \frac{1}{2}} e^{-\frac{\pi}{2}|\tau|}\]
for $\sigma > -M$ by \hyperref[lem:H+Mellinbound]{Lemma \ref*{lem:H+} \ref*{lem:H+Mellinbound}}, while $\widehat{H^{-}}(s) = 0$.

Provided that $\sigma$ is bounded and $s$ is a bounded distance away from the poles at $s = 1 + 2\ell$, $s = 1 + 2\ell \pm 4it_g$, and $s = -2(it + \ell)$ with $\ell \in \N_0$, we have by the definitions \eqref{eqn:Gscrmupm} and \eqref{eqn:Gpm} and Stirling's formula that the integrand in \eqref{eqn:HHmuFt} satisfies the bounds
\begin{multline*}
\sum_{\pm_1,\pm_2} \widehat{H^{\pm_1}}(s) \Gscr_{\mu_F}^{\pm_2}\left(\frac{1 - s}{2}\right) G^{\mp_1 \pm_2}\left(\frac{s}{2} + it\right)	\\
\ll_{\sigma,M} T^{1 + \sigma} (1 + |\tau|)^{-M} ((1 + |\tau + 4t_g|) (1 + |\tau|) (1 + |\tau - 4t_g|))^{-\frac{\sigma}{2}} (1 + |\tau + 2t|)^{\frac{\sigma - 1}{2}}.
\end{multline*}

Since
\[\Res_{s = 1} \Gscr_{\mu_F}^{\pm}\left(\frac{1 - s}{2}\right) \ll \frac{1}{t_g}, \qquad \Res_{s = 1 \pm_3 4it_g} \Gscr_{\mu_F}^{\pm}\left(\frac{1 - s}{2}\right) \ll \frac{1}{t_g},\]
we have that
\begin{align*}
\Res_{s = 1} \sum_{\pm_1, \pm_2} \widehat{H^{\pm_1}}(s) \Gscr_{\mu_F}^{\pm_2}\left(\frac{1 - s}{2}\right) G^{\mp_1 \pm_2}\left(\frac{s}{2} + it\right) & \ll_M \frac{T^2}{t_g},	\\
\Res_{s = 1 \pm_3 4it_g} \sum_{\pm_1,\pm_2} \widehat{H^{\pm_1}}(s) \Gscr_{\mu_F}^{\pm_2}\left(\frac{1 - s}{2}\right) G^{\mp_1 \pm_2}\left(\frac{s}{2} + it\right) & \ll_M \frac{T^2}{t_g^{M + 1}}.
\end{align*}
To bound $\HH_{\mu_F}(t)$ when $|t| \leq t_g^2/T^2$, we shift the contour of integration in \eqref{eqn:HHmuFt} to $\Re(s) = 1$, avoiding the poles at $s = 1$ and $s = 1 \pm 4it_g$ via deforming the contour to instead consist of small semicircles to the left of this line at each of these three poles. Since the integrand decays rapidly due to the decay of the Mellin transform of $H^{\pm_1}(x)$, the main contribution arises from the pole at $s = 1$ and from the portion of the integral for which $\tau$ is essentially bounded. In this way, we find that for $|t| \leq t_g^2/T^2$,
\[\HH_{\mu_F}(t) \ll_M \frac{T^2}{t_g}.\]

Since
\[\Res_{s = -2(it + \ell)} G^{\pm}\left(\frac{s}{2} + it\right) \ll_{\ell} 1,\]
we have that for each nonnegative integer $\ell < M$,
\begin{multline*}
\Res_{s = -2(it + \ell)} \sum_{\pm_1,\pm_2} \widehat{H^{\pm_1}}(s) \Gscr_{\mu_F}^{\pm_1}\left(\frac{1 - s}{2}\right) G^{\mp_1 \pm_2}\left(\frac{s}{2} + it\right)	\\
\ll_M T^{1 - 2\ell} (1 + |t|)^{-M} ((1 + |2t_g + t|) (1 + |t|) (1 + |2t_g - t|))^{\ell}.
\end{multline*}
To bound $\HH_{\mu_F}(t)$ when $|t| \geq t_g^2/T^2$, we shift the contour of integration in \eqref{eqn:HHmuFt} to $\Re(s) = -M/4$, deforming the contour if necessary by a small semicircle if a pole lies on this line. For the resulting integral on the line $\Re(s) = -M/4$, the integrand is negligibly small unless $\tau$ is essentially bounded, and hence for $|t| \geq t_g^2/T^2$,
\[\HH_{\mu_F}(t) \ll_M \frac{T}{|t|^{1/2}} \left(\frac{T^2 |t|}{t_g^2}\right)^{-\frac{M}{2}}.\qedhere\]
\end{proof}

\subsection{\texorpdfstring{$\mathrm{GL}_4 \times \mathrm{GL}_2$}{GL\9040\204 \80\327 GL\9040\202} Transforms}

Similarly, we determine the behaviour of $(\widetilde{h}^{+},\widetilde{h}^{-},\widetilde{h}^{\hol})$ as in \eqref{eqn:tildehpmdefeq} and \eqref{eqn:tildehholdefeq} with $(h^{+},h^{-},h^{\hol})$ either of the triples of test functions \eqref{eqn:firsttriple} or \eqref{eqn:secondtriple}.

\begin{lemma}
\label{lem:tildehpmbounds}
Let $g$ be a Hecke--Maa\ss{} cusp form on $\Gamma \backslash \Hb$ with spectral parameter $t_g$. Let $(h^{+},h^{-},h^{\hol})$ be either of the triples of test functions \eqref{eqn:firsttriple} or \eqref{eqn:secondtriple} for a fixed positive integer $M > 50$ and $T > M$. Then for $F = \ad g$ and $(\widetilde{h}^{+},\widetilde{h}^{-},\widetilde{h}^{\hol})$ as in \eqref{eqn:tildehpmdefeq} and \eqref{eqn:tildehholdefeq}, we have that
\begin{align}
\label{eqn:tildeh+bound}
\widetilde{h}^{+}(t) & \ll_M \frac{T^2 \log (t_g + T)}{t_g} (1 + |t|)^{-\frac{1}{2}(M + 1)},	\\
\label{eqn:tildeh-bound}
\widetilde{h}^{-}(t) & \ll_M \begin{dcases*}
\frac{T^2 \log (t_g + T)}{t_g} & for $|t| \leq \frac{t_g}{T} + 1$,	\\
\frac{T}{|t|} \left(\frac{T |t|}{t_g}\right)^{-\frac{M}{4}} & for $|t| \geq \frac{t_g}{T} + 1$,	\\
\end{dcases*}	\\
\label{eqn:tildehholbound}
\widetilde{h}^{\hol}(k) & \ll_M \begin{dcases*}
\frac{T^2 \log (t_g + T)}{t_g} & for $k \leq \frac{t_g}{T} + 1$,	\\
\frac{T}{k} \left(\frac{T k}{t_g}\right)^{-\frac{M}{4}} & for $k \geq \frac{t_g}{T} + 1$.
\end{dcases*}
\end{align}
\end{lemma}

Combining the bounds in \hyperref[lem:tildehpmbounds]{Lemma \ref*{lem:tildehpmbounds}} with the $\GL_4 \times \GL_2 \leftrightsquigarrow \GL_4 \times \GL_2$ spectral reciprocity formula obtained in \hyperref[thm:4x2reciprocity]{Theorem \ref*{thm:4x2reciprocity}}, we may deduce an identity roughly of the form
\begin{multline*}
\sum_{\substack{f \in \BB_0 \\ T \leq t_f \leq 2T}} \frac{L\left(\frac{1}{2},f\right) L\left(\frac{1}{2},\ad g \otimes f\right)}{L(1,\ad f)} + \frac{1}{2\pi} \int\limits_{T \leq |t| \leq 2T} \left|\frac{\zeta\left(\frac{1}{2} + it\right) L\left(\frac{1}{2} + it,\ad g\right)}{\zeta(1 + 2it)}\right|^2 \, dt	\\
+ \sum_{\substack{f \in \BB_{\hol} \\ T \leq k_f \leq 2T}} \frac{L\left(\frac{1}{2},f\right) L\left(\frac{1}{2},\ad g \otimes f\right)}{L(1,\ad f)}	\\
\approx T^2 + \frac{T^2 \log t_g}{t_g} \sum_{\substack{f \in \BB_0 \\ t_f \leq \frac{t_g}{T}}} \frac{L\left(\frac{1}{2},f\right) L\left(\frac{1}{2},\ad g \otimes f\right)}{L(1,\ad f)} + \frac{T^2 \log t_g}{t_g} \frac{1}{2\pi} \int\limits_{|t| \leq \frac{t_g}{T}} \left|\frac{\zeta\left(\frac{1}{2} + it\right) L\left(\frac{1}{2} + it,\ad g\right)}{\zeta(1 + 2it)}\right|^2 \, dt	\\
+ \frac{T^2 \log t_g}{t_g} \sum_{\substack{f \in \BB_{\hol} \\ k_f \leq \frac{t_g}{T}}} \frac{L\left(\frac{1}{2},f\right) L\left(\frac{1}{2},\ad g \otimes f\right)}{L(1,\ad f)}.
\end{multline*}
In \hyperref[prop:momentsboundsinitial]{Proposition \ref*{prop:momentsboundsinitial}}, we use this identity to provide upper bounds for each of the terms on the left-hand side.

We first state bounds for the integrands in \eqref{eqn:tildehpmdefeq} and \eqref{eqn:tildehholdefeq}.

\begin{lemma}
\label{lem:mellinbounds}
Provided that $s$ is a bounded distance away from the poles at $s = 1 + 2\ell$, $s = 1 + 2\ell \pm 4it_g$, and $s = -2(\pm it + \ell)$ with $\ell \in \N_0$, we have that for $t \in \R$,
\begin{multline}
\label{eqn:tildeintegrandpm}
\sum_{\pm_2} \widehat{\JJ_t^{\pm}}(s) G^{\mp_2}\left(\frac{1 - s}{2}\right) \Gscr_{\mu_F}^{\pm \pm_1 \pm_2}\left(\frac{1 - s}{2}\right)	\\
\ll_{\sigma} (1 + |\tau|)^{-\sigma} ((1 + |\tau + 4t_g|)(1 + |\tau - 4t_g|))^{-\frac{\sigma}{2}} ((1 + |\tau + 2t|)(1 + |\tau - 2t|))^{\frac{1}{2} (\sigma - 1)} e^{-\frac{\pi}{2} \Omega^{\pm,\pm_1}(\tau,t,t_g)},
\end{multline}
and for $k \in 2\N$,
\begin{multline}
\label{eqn:tildeintegrandhol}
\sum_{\pm_2} \widehat{\JJ_k^{\hol}}(s) G^{\mp_2}\left(\frac{1 - s}{2}\right) \Gscr_{\mu_F}^{\pm \pm_1 \pm_2}\left(\frac{1 - s}{2}\right)	\\
\ll_{\sigma} (1 + |\tau|)^{-\sigma} ((1 + |\tau + 4t_g|)(1 + |\tau - 4t_g|))^{-\frac{\sigma}{2}} (k + |\tau|)^{\sigma - 1} e^{-\frac{\pi}{2} \Omega^{\hol,\pm_1}(\tau,t_g)},
\end{multline}
where
\begin{align*}
\Omega^{+,+}(\tau,t,t_g) & \coloneqq \begin{dcases*}
2|t| & if $|\tau| \leq 2\min\{t_g,|t|\}$,	\\
|\tau| & if $2|t| \leq |\tau| \leq 2t_g$,	\\
2(2t_g + |t| - |\tau|) & if $2t_g \leq |\tau| \leq 2\min\{2t_g,|t|\}$,	\\
4t_g - |\tau| & if $2\max\{t_g,|t|\} \leq |\tau| \leq 4t_g$,	\\
2|t| - |\tau| & if $4t_g \leq |\tau| \leq 2|t|$,	\\
0 & if $|\tau| \geq 2\max\{2t_g,|t|\}$,
\end{dcases*}	\\
\Omega^{-,+}(\tau,t,t_g) & \coloneqq \begin{dcases*}
0 & if $|\tau| \leq 2\min\{2t_g,|t|\}$,	\\
|\tau| - 2|t| & if $2|t| \leq |\tau| \leq 4t_g$,	\\
|\tau| - 4t_g & if $4t_g \leq |\tau| \leq 2|t|$,	\\
2(|\tau| - |t| - 2t_g) & if $|\tau| \geq 2\max\{2t_g,|t|\}$,
\end{dcases*}	\\
\Omega^{\hol,+}(\tau,t_g) & \coloneqq \begin{dcases*}
|\tau| & if $|\tau| \leq 2t_g$,	\\
4t_g - |\tau| & if $2t_g \leq |\tau| \leq 4t_g$,	\\
0 & if $|\tau| \geq 4t_g$,
\end{dcases*}
\end{align*}
while
\begin{align*}
\Omega^{+,-}(\tau,t,t_g) & \coloneqq \begin{dcases*}
2|t| - |\tau| & if $|\tau| \leq 2\min\{2t_g,|t|\}$,	\\
0 & if $2|t| \leq |\tau| \leq 4t_g$,	\\
2(|t| - 2t_g) & if $4t_g \leq |\tau| \leq 2|t|$,	\\
|\tau| - 4t_g & if $|\tau| \geq 2\max\{2t_g,|t|\}$,
\end{dcases*}	\\
\Omega^{-,-}(\tau,t,t_g) & \coloneqq \begin{dcases*}
|\tau| & if $|\tau| \leq 2\min\{t_g,|t|\}$,	\\
2(|\tau| - |t|) & if $2|t| \leq |\tau| \leq 2t_g$,	\\
4t_g - |\tau| & if $2t_g \leq |\tau| \leq 2\min\{2t_g,|t|\}$,	\\
2(2t_g - |t|) & if $2\max\{t_g,|t|\} \leq |\tau| \leq 4t_g$,	\\
0 & if $4t_g \leq |\tau| \leq 2|t|$,	\\
|\tau| - 2|t| & if $|\tau| \geq 2\max\{2t_g,|t|\}$,
\end{dcases*}	\\
\Omega^{\hol,-}(\tau,t_g) & \coloneqq \begin{dcases*}
0 & if $|\tau| \leq 4t_g$,	\\
|\tau| - 4t_g & if $|\tau| \geq 4t_g$.
\end{dcases*}
\end{align*}
\end{lemma}

\begin{proof}
This follows from the definitions \eqref{eqn:Gscrmupm} and \eqref{eqn:Gpm} of $\Gscr_{\mu_F}^{\pm}(s)$ and $G^{\pm}(s)$, the bounds \eqref{eqn:JJr+Mellinbound}, \eqref{eqn:JJr-Mellinbound}, and \eqref{eqn:JJkholMellinbound} for the Mellin transforms of $\JJ_t^{\pm}$ and $\JJ_k^{\hol}$, and Stirling's formula.
\end{proof}

\begin{proof}[Proof of {\hyperref[lem:tildehpmbounds]{Lemma \ref*{lem:tildehpmbounds}}}]
For the triple of test functions \eqref{eqn:firsttriple}, we have that for $s = \sigma + i\tau$ with $-M/2 < \sigma < M/2$,
\begin{equation}
\label{eqn:H-Mellinbound}
\widehat{H^{-}}(s) \ll_{\sigma,M} T^{1 + \sigma} (1 + |\tau|)^{-M},
\end{equation}
while $\widehat{H^{+}}(s) = 0$. Similarly, for the triple of test functions \eqref{eqn:secondtriple}, we have that for $\sigma > -M$,
\begin{equation}
\label{eqn:H+Mellinbound}
\widehat{H^{+}}(s) \ll_{\sigma,M} T^{1 + \sigma} (1 + |\tau|)^{\sigma + M - \frac{1}{2}} e^{-\frac{\pi}{2}|\tau|},
\end{equation}
while $\widehat{H^{-}}(s) = 0$.

Since $\sum_{\pm_2} G^{\mp_2}(\frac{1 - s}{2}) \Gscr_{\mu_F}^{\pm \pm_2}(\frac{1 - s}{2})$ has a double pole at $s = 1$ and simple poles at $s = 1 \pm 4it_g$, we have that for $|t| \leq t_g/T + 1$ and $k \leq t_g/T + 1$,
\begin{align*}
\Res_{s = 1} \sum_{\pm_2} \widehat{H^{\pm_1}}(s) \widehat{\JJ_t^{\pm}}(s) G^{\mp_2}\left(\frac{1 - s}{2}\right) \Gscr_{\mu_F}^{\pm \pm_1 \pm_2}\left(\frac{1 - s}{2}\right) & \ll_M \frac{T^2 \log (t_g + T)}{t_g} e^{-\frac{\pi}{2} \Omega^{\pm,\pm_1}(0,t,t_g)},	\\
\Res_{s = 1} \sum_{\pm_2} \widehat{H^{\pm_1}}(s) \widehat{\JJ_k^{\hol}}(s) G^{\mp_2}\left(\frac{1 - s}{2}\right) \Gscr_{\mu_F}^{\pm \pm_1 \pm_2}\left(\frac{1 - s}{2}\right) & \ll_M \frac{T^2 \log (t_g + T)}{t_g} e^{-\frac{\pi}{2} \Omega^{\hol,\pm_1}(0,t_g)},	\\
\Res_{s = 1 \pm 4it_g} \sum_{\pm_2} \widehat{\JJ_t^{\pm}}(s) G^{\mp_2}\left(\frac{1 - s}{2}\right) \Gscr_{\mu_F}^{\pm \pm_1 \pm_2}\left(\frac{1 - s}{2}\right) & \ll_M \frac{T^2}{t_g^{M + 1}} e^{-\frac{\pi}{2} \Omega^{\pm,\pm_1}(\pm 4t_g,t,t_g)},	\\
\Res_{s = 1 \pm 4it_g} \sum_{\pm_2} \widehat{\JJ_k^{\hol}}(s) G^{\mp_2}\left(\frac{1 - s}{2}\right) \Gscr_{\mu_F}^{\pm \pm_1 \pm_2}\left(\frac{1 - s}{2}\right) & \ll_M \frac{T^2}{t_g^{M + 1}} e^{-\frac{\pi}{2} \Omega^{\hol,\pm_1}(\pm 4t_g,t_g)}.
\end{align*}
Here we have used Stirling's formula, the bounds \eqref{eqn:JJr+Mellinbound}, \eqref{eqn:JJr-Mellinbound}, and \eqref{eqn:JJkholMellinbound} for the Mellin transforms of $\JJ_t^{\pm}$ and $\JJ_k^{\hol}$, and the bounds \eqref{eqn:H-Mellinbound} and \eqref{eqn:H+Mellinbound} for the Mellin transforms of $H^{\pm}$. To bound $\widetilde{h}^{-}(t)$ when $|t| \leq t_g/T + 1$ and $\widetilde{h}^{\hol}(k)$ when $k \leq t_g/T + 1$, we shift the contour of integration to $\Re(s) = 1$, avoiding the poles at $s = 1$ and $s = 1 \pm 4it_g$ via deforming the contour to instead consist of small semicircles to the left of this line at each of these three poles. Since the integrand decays rapidly due to the decay of the Mellin transform of $H^{\pm_1}(x)$, the main contribution arises from the pole at $s = 1$ and from the portion of the integral for which $\tau$ is essentially bounded. Together with the bounds \eqref{eqn:tildeintegrandpm} and \eqref{eqn:tildeintegrandhol}, we find that for $|t| \leq t_g/T + 1$ and $k \leq t_g/T + 1$,
\[\widetilde{h}^{-}(t) \ll_M \frac{T^2 \log (t_g + T)}{t_g}, \qquad \widetilde{h}^{\hol}(k) \ll_M \frac{T^2 \log (t_g + T)}{t_g}.\]
Similarly, we find that for all $t \in \R$,
\[\widetilde{h}^{+}(t) \ll_M \frac{T^2}{t_g} (1 + |t|)^{-\frac{1}{2}(M + 1)}.\]

Next, we have that for each nonnegative integer $\ell < M/4$,
\begin{multline*}
\Res_{s = -2(\pm it + \ell)} \sum_{\pm_1,\pm_2} \widehat{H^{\pm_1}}(s) \widehat{\JJ_t^{-}}(s) \Gscr_{\mu_F}^{\pm_1}\left(\frac{1 - s}{2}\right) G^{\mp_1 \pm_2}\left(\frac{s}{2} + it\right)	\\
\ll_M T^{1 - 2\ell} (1 + |t|)^{-M + \ell - \frac{1}{2}} ((1 + |t + 2t_g|)(1 + |t - 2t_g|))^{\ell}
\end{multline*}
via \eqref{eqn:JJrpmMellinResbound}, while for each nonnegative integer $\ell < M/2 + 1 - k$,
\begin{multline*}
\Res_{s = 1 - k - 2\ell} \sum_{\pm_1,\pm_2} \widehat{H^{\pm_1}}(s) \widehat{\JJ_k^{\hol}}(s) \Gscr_{\mu_F}^{\pm_1}\left(\frac{1 - s}{2}\right) G^{\mp_1 \pm_2}\left(\frac{s}{2} + it\right)	\\
\ll_M T^{1 - 2\ell} (1 + |t|)^{-M + \ell - \frac{1}{2}} ((1 + |t + 2t_g|)(1 + |t - 2t_g|))^{\ell}
\end{multline*}
via \eqref{eqn:JJkholMellinResbound}. To bound $\widetilde{h}^{-}(t)$ when $|t| \geq t_g/T + 1$ and $\widetilde{h}^{\hol}(k)$ when $k \geq t_g/T + 1$, we shift the contour of integration to $\Re(s) = -M/4$, deforming the contour if necessary by a small semicircle if a pole lies on this line. For the resulting integral on the line $\Re(s) = -M/4$, the integrand is negligibly small unless $\tau$ is essentially bounded, and hence for $|t| \geq t_g/T + 1$ and $k \geq t_g/T + 1$,
\[\widetilde{h}^{-}(t) \ll_M \frac{T}{|t|} \left(\frac{T |t|}{t_g}\right)^{-\frac{M}{4}}, \qquad \widetilde{h}^{\hol}(k) \ll_M \frac{T}{k} \left(\frac{T k}{t_g}\right)^{-\frac{M}{4}}.\qedhere\]
\end{proof}

\section{Bounds for Mixed Moments of \texorpdfstring{$L$}{L}-Functions in the Short Initial Range}

\subsection{Bounds via the Spectral Large Sieve}

The simplest approach to bounding the mixed moment of $L$-functions \eqref{eqn:initial} is to perform a dyadic subdivision, apply the Cauchy--Schwarz inequality, and bounding the ensuing second moments of $L$-functions via the spectral large sieve. Below, we state the bounds that one achieves via this approach.

\begin{proposition}
\label{prop:largesievebounds1}\hspace{1em}
\begin{enumerate}[leftmargin=*,label=\textup{(\arabic*)}]
\item\label{item:largesievebounds1} Let $g$ be a Hecke--Maa\ss{} cusp form on $\Gamma \backslash \Hb$ with spectral parameter $t_g$. For $T \geq 1$, we have the bounds
\begin{equation}
\label{eqn:largesievebounds1}
\begin{drcases*}
\sum_{\substack{f \in \BB_0 \\ T \leq t_f \leq 2T}} \frac{L\left(\frac{1}{2},\ad g \otimes f\right)^2}{L(1,\ad f)} & \\
\frac{1}{2\pi} \int\limits_{T \leq |t| \leq 2T} \left|\frac{L\left(\frac{1}{2} + it,\ad g\right)^2}{\zeta(1 + 2it)}\right|^2 \, dt & \\
\sum_{\substack{f \in \BB_{\hol} \\ T \leq k_f \leq 2T}} \frac{L\left(\frac{1}{2},\ad g \otimes f\right)^2}{L(1,\ad f)} & 
\end{drcases*} \ll_{\e} \begin{dcases*}
t_g^{2 + \e} T & if $T \leq 2t_g$,	\\
T^{3 + \e} & if $T \geq 2t_g$.
\end{dcases*}
\end{equation}
\item\label{item:largesievebounds2} For $T \geq 1$ and $1 \leq U \leq T$, we have the bounds
\begin{equation}
\label{eqn:largesievebounds2}
\begin{drcases*}
\sum_{\substack{f \in \BB_0 \\ T - U \leq t_f \leq T + U}} \frac{L\left(\frac{1}{2},f\right)^2}{L(1,\ad f)} & \\
\frac{1}{2\pi} \int\limits_{T - U \leq |t| \leq T + U} \left|\frac{\zeta\left(\frac{1}{2} + it\right)^2}{\zeta(1 + 2it)}\right|^2 \, dt & \\
\sum_{\substack{f \in \BB_{\hol} \\ T - U \leq k_f \leq T + U}} \frac{L\left(\frac{1}{2},f\right)^2}{L(1,\ad f)} & 
\end{drcases*} \ll_{\e} T^{1 + \e} U.
\end{equation}
\end{enumerate}
\end{proposition}

\begin{proof}
These are all consequences of the approximate functional equation \cite[Theorem 5.3]{IK04} and the spectral large sieve. We give details for \eqref{eqn:largesievebounds1} for the first term on the left-hand side; the other cases are similar. For $f \in \BB_0$, we have that
\[L(s,\ad g \otimes f) = \sum_{m,n = 1}^{\infty} \frac{A_F(m,n) \lambda_f(n)}{m^{2s} n^s}\]
for $\Re(s) > 1$, where $F = \ad g$, and the conductor of $L(1/2,\ad g \otimes f)$ is $O(t_f^2 \max\{t_g^4,t_f^4\})$. Thus by writing $L(1/2,\ad g \otimes f)$ as a Dirichlet polynomial via the approximate functional equation \cite[Theorem 5.3]{IK04} and applying the spectral large sieve \cite[Theorem 3.3]{Mot97}, we find that the first term on the left-hand side is
\[\ll_{\e} t_g^{\e} T^{\e} \sup_{M^2 N \leq t_g^{\e} T^{1 + \e} \max\{t_g^2,T^2\}} (T^2 + N) \sum_{N \leq n \leq 2N} \left|\sum_{M \leq m \leq 2M} \frac{A_F(m,n)}{m\sqrt{n}}\right|^2.\]
By the Cauchy--Schwarz inequality, this in turn is
\[\ll_{\e} t_g^{\e} T^{\e} \sup_{M^2 N \leq t_g^{\e} T^{1 + \e} \max\{t_g^2,T^2\}} (T^2 + N) \sum_{M \leq \ell \leq 2M} \frac{1}{\ell} \sum_{M \leq m \leq 2M} \sum_{N \leq n \leq 2N} \frac{|A_F(m,n)|^2}{mn}.\]
The sum over $\ell$ is $O_{\e}(t_g^{\e} T^{\e})$. By Rankin's trick and the fact that $A_F(m,n) = \overline{A_F(n,m)}$, the double sum over $m$ and $n$ above is
\[O_{\e} \left(t_g^{\e} T^{\e} \min\{M,N\}\sum_{m,n = 1}^{\infty} \frac{|A_F(m,n)|^2}{(m^2 n)^{1 + \e}}\right).\]
The new double sum over $m,n \in \N$ is equal to
\[\frac{L(1 + \e,\ad g \otimes \ad g)}{\zeta(3 + 3\e)} = \frac{L(1 + \e,\sym^4 g) L(1 + \e,\ad g) \zeta(1 + \e)}{\zeta(3 + 3\e)}.\]
By \cite[Theorem 2]{Li10}, this is $O_{\e}(t_g^{\e})$. Thus the left-hand side of \eqref{eqn:largesievebounds1} is
\[\ll_{\e} t_g^{\e} T^{\e} \sup_{M^2 N \leq t_g^{\e} T^{1 + \e} \max\{t_g^2,T^2\}} (T^2 + N) \min\{M,N\} \ll_{\e} \begin{dcases*}
t_g^{2 + \e} T & if $T \leq 2t_g$,	\\
T^{3 + \e} & if $T \geq 2t_g$.
\end{dcases*}\qedhere\]
\end{proof}

We record the following refinement of the second moment bound \eqref{eqn:LsF2ndmoment} when $\sigma = 1/2$. We shall shortly use this to bound the last term on the right-hand side of \eqref{eqn:3x2identity}.

\begin{lemma}
\label{lem:L1/2F2ndmoment}
Let $g$ be a Hecke--Maa\ss{} cusp form on $\Gamma \backslash \Hb$ with spectral parameter $t_g$. Then for $U \geq 1$, we have that
\begin{equation}
\label{eqn:L1/2F2ndmoment}
\int_{U}^{2U} \left|L\left(\frac{1}{2} + it,\ad g\right)\right|^2 \, dt \ll_{\e} \begin{dcases*}
t_g^{1 + \e} U^{\frac{1}{2}} & if $U \leq t_g$,	\\
U^{\frac{3}{2} + \e} & if $U \geq t_g$.
\end{dcases*}
\end{equation}
\end{lemma}

\begin{proof}
This follows by using the approximate functional equation \cite[Theorem 5.3]{IK04} to write $L(1/2 + it, \ad g)$ in terms of a Dirichlet polynomial and then invoking the Montgomery--Vaughan mean value theorem for Dirichlet polynomials \cite[Corollary 3]{MV74}, noting that the analytic conductor of $L(1/2 + it,\ad g)$ is $O(t_g^2 (1 + |t|))$ if $|t| \leq t_g$ and is $O(|t|^3)$ if $|t| \geq t_g$.
\end{proof}

\subsection{Bounds via Spectral Reciprocity}

We now show how to obtain improved bounds for the mixed moment of $L$-functions \eqref{eqn:initial} via spectral reciprocity. Our first step is to use \hyperref[thm:3x2reciprocity]{Theorem \ref*{thm:3x2reciprocity}} to prove bounds for the first moment of $L(1/2,\ad g \otimes f)$.

\begin{proposition}
\label{prop:upperbounds}
Let $g$ be a Hecke--Maa\ss{} cusp form on $\Gamma \backslash \Hb$ with spectral parameter $t_g$. Then for $T \geq 1$, we have that
\begin{equation}
\label{eqn:upperbounds}
\begin{drcases*}
\sum_{\substack{f \in \BB_0 \\ T \leq t_f \leq 2T}} \frac{L\left(\frac{1}{2},\ad g \otimes f\right)}{L(1,\ad f)} &	\\
\frac{1}{2\pi} \int\limits_{T \leq |t| \leq 2T} \left|\frac{L\left(\frac{1}{2} + it,\ad g\right)}{\zeta(1 + 2it)}\right|^2 \, dt &	\\
\sum_{\substack{f \in \BB_{\hol} \\ T \leq k_f \leq 2T}} \frac{L\left(\frac{1}{2},\ad g \otimes f\right)}{L(1,\ad f)} &	
\end{drcases*}
\ll_{\e} \begin{dcases*}
t_g^{1 + \e} T^{\frac{3}{2}} & if $T \leq t_g^{\frac{1}{4}}$,	\\
t_g^{\frac{3}{2} + \e} T^{-\frac{1}{2}} & if $t_g^{\frac{1}{4}} \leq T \leq t_g^{\frac{1}{2}}$,	\\
t_g^{1 + \e} T^{\frac{1}{2}} & if $t_g^{\frac{1}{2}} \leq T \leq t_g^{\frac{2}{3}}$,	\\
t_g^{\e} T^2 & if $T \geq t_g^{\frac{2}{3}}$.
\end{dcases*}
\end{equation}
\end{proposition}

The summands and integrands on the left-hand side of \eqref{eqn:upperbounds} are nonnegative due to the fact that $L(1/2,\ad g \otimes f) \geq 0$ by \cite[Theorem 1.1]{Lap03}.

\begin{proof}
For $T \leq t_g^{1/4}$, this follows by the Cauchy--Schwarz inequality together with the bounds \eqref{eqn:largesievebounds1}. For $T \geq t_g^{1/4}$, we use \hyperref[thm:3x2reciprocity]{Theorem \ref*{thm:3x2reciprocity}} with $F = \ad g$, so that $\mu_F = (2it_g,0,-2it_g)$.

We first take the triple of test functions $(h^{+},h^{-},h^{\hol})$ given by \eqref{eqn:firsttriple}. With this choice of test functions, the left-hand side of \eqref{eqn:3x2identity} provides an upper bound for the first and second terms on the left-hand side of \eqref{eqn:upperbounds} by positivity, as $h^{+}(t) = 0$ and $h^{\hol}(k) = 0$, while $h^{-}(t) \geq 0$ for all $t \in \R$ and $h^{-}(t) \asymp_M 1$ for $t \in [-2T,-T] \cup [T,2T]$. The first term on the right-hand side of \eqref{eqn:3x2identity} is $O_{\e}(t_g^{\e} T^2)$ via the bounds $L(1,\ad g) \ll_{\e} t_g^{\e}$ and $\int_{-\infty}^{\infty} h^{\pm}(r) \, d_{\spec}r \ll T^2$, which is clear from the definitions \eqref{eqn:firsttriple} of $h^{\pm}$ and \eqref{eqn:dspecdefeq} of $d_{\spec}r$. The second term is equal to zero since $h^{\hol}(k) = 0$. Finally, for the third term, we use the bounds \eqref{eqn:HHmuFbounds} for $\HH_{\mu_F}$ and then apply the Cauchy--Schwarz inequality. The desired bounds then follow from the bounds \eqref{eqn:L1/2F2ndmoment} and \eqref{eqn:zeta2ndmoment} for the second moments of $L(1/2 + it,\ad g)$ and $\zeta(1/2 + it)$.

We next take the triple of test functions $(h^{+},h^{-},h^{\hol})$ given by \eqref{eqn:secondtriple}. Here it is no longer the case that the left-hand side of \eqref{eqn:3x2identity} consists of only nonnegative terms. Nonetheless, the first two terms on the left-hand side of \eqref{eqn:3x2identity} as well as the contribution from the terms in the third term for which $k_f \leq M$ are $O_{M,\e}(t_g^{1 + \e} T^{1 - M})$ by \hyperref[lem:H+Lholbound]{Lemma \ref*{lem:H+} \ref*{lem:H+Lhol} \ref*{lem:H+Lholbound}} and \ref{lem:H+L+bound}, the Cauchy--Schwarz inequality, and the bounds \eqref{eqn:largesievebounds1}. The contribution from the terms in the third term on the left-hand side of \eqref{eqn:3x2identity} for which $k_f > M$ provides an upper bound for the third term on the left-hand side of \eqref{eqn:upperbounds} by positivity via \hyperref[lem:H+Lholpos]{Lemma \ref*{lem:H+} \ref*{lem:H+Lhol} \ref*{lem:H+Lholpos}} and \ref{lem:H+Lholasymp}, noting that the root number of $L(s,\ad g \otimes f)$ is $i^{k_f}$, and hence $L(1/2,\ad g \otimes f) = 0$ when $k_f \equiv 2 \pmod{4}$. Finally, the right-hand side of \eqref{eqn:3x2identity} is bounded in the same way as for the triple of test functions given by \eqref{eqn:firsttriple}.
\end{proof}

In our treatment of the mixed moment of $L$-functions \eqref{eqn:initial}, we shall apply H\"{o}lder's inequality to separate the $L$-functions involved. Underlying this step is the key requirement that we have strong bounds for high moments of $L(1/2,f)$ and $\zeta(1/2 + it)$. While we could merely employ the \emph{individual} Weyl-strength subconvex bounds $L(1/2,f) \ll_{\e} t_f^{1/3 + \e}$ and $\zeta(1/2 + it) \ll_{\e} (1 + |t|)^{1/6 + \e}$, stronger bounds hold \emph{on average}; indeed, we are best served by using bounds for the twelfth moment (though fifth moment bounds would also be advantageous, as discussed in \hyperref[sect:fifthmoment]{Section \ref*{sect:fifthmoment}}).

\begin{proposition}
\label{prop:Jutilabounds}
For $T \geq 1$, we have that
\begin{equation}
\label{eqn:Jutilabounds}
\begin{drcases*}
\sum_{\substack{f \in \BB_0 \\ T \leq t_f \leq 2T}} \frac{L\left(\frac{1}{2},f\right)^{12}}{L(1,\ad f)} & \\
\frac{1}{2\pi} \int\limits_{T \leq |t| \leq 2T} \left|\frac{\zeta\left(\frac{1}{2} + it\right)^{12}}{\zeta(1 + 2it)}\right|^2 \, dt & \\
\sum_{\substack{f \in \BB_{\hol} \\ T \leq k_f \leq 2T}} \frac{L\left(\frac{1}{2},f\right)^{12}}{L(1,\ad f)} & 
\end{drcases*} \ll_{\e} T^{4 + \e}.
\end{equation}
\end{proposition}

\begin{proof}
For the second term on the left-hand side, this bound follows from Heath-Brown's bound for the twelfth moment of the Riemann zeta function \cite[Theorem 1]{H-B78} (see also \cite[Theorem 4]{Iwa80}, \cite[Theorem 4.7]{Jut87}, and more generally \cite[Chapter 8]{Ivi03}) together with the Weyl-strength subconvex bound $\zeta(1/2 + it) \ll_{\e} (1 + |t|)^{1/6 + \e}$ and the classical lower bound $|\zeta(1 + 2it)| \gg 1/\log (2 + |t|)$; alternatively, we can simply appeal to \cite[Theorem 8.3]{Ivi03}. For the first term on the left-hand side, this bound is a result of Jutila \cite[Theorem 2]{Jut04b}. The authors extended Jutila's result to cover the same result for the third term on the left-hand side, namely for holomorphic cusp forms, in \cite[Theorem 1.1]{HK24}.
\end{proof}

\begin{remark}
Notably, the proofs \cite[Theorem 2]{Jut04b} and \cite[Theorem 1.1]{HK24} of the cuspidal cases of \eqref{eqn:Jutilabounds} use $\GL_4 \times \GL_2 \leftrightsquigarrow \GL_4 \times \GL_2$ spectral reciprocity in a crucial way, as discussed in \cite[Section 3]{HK24}.
\end{remark}

We can combine the bounds attained so far to prove new bounds via $\GL_4 \times \GL_2 \leftrightsquigarrow \GL_4 \times \GL_2$ spectral reciprocity. In this way, we can show the following bounds for mixed moments of $L$-functions in the short initial range.

\begin{proposition}
\label{prop:momentsboundsinitial}
Let $g$ be a Hecke--Maa\ss{} cusp form on $\Gamma \backslash \Hb$ with spectral parameter $t_g$. Then for $T \geq 1$, we have that
\begin{multline}
\label{eqn:momentsboundsinitial}
\begin{drcases*}
\sum_{\substack{f \in \BB_0 \\ T \leq t_f \leq 2T}} \frac{L\left(\frac{1}{2},f\right) L\left(\frac{1}{2},\ad g \otimes f\right)}{L(1,\ad f)} &	\\
\frac{1}{2\pi} \int\limits_{T \leq |t| \leq 2T} \left|\frac{\zeta\left(\frac{1}{2} + it\right) L\left(\frac{1}{2} + it,\ad g\right)}{\zeta(1 + 2it)}\right|^2 \, dt &	\\
\sum_{\substack{f \in \BB_{\hol} \\ T \leq k_f \leq 2T}} \frac{L\left(\frac{1}{2},f\right) L\left(\frac{1}{2},\ad g \otimes f\right)}{L(1,\ad f)} &	
\end{drcases*}
\ll_{\e} \begin{dcases*}
t_g^{1 + \e} T^{\frac{3}{2}} & if $T \leq t_g^{\frac{3}{19}}$,	\\
t_g^{\frac{5}{4} + \e} T^{-\frac{1}{12}} & if $t_g^{\frac{3}{19}} \leq T \leq t_g^{\frac{1}{3}}$,	\\
t_g^{\frac{5}{6} + \e} T^{\frac{7}{6}} & if $t_g^{\frac{1}{3}} \leq T \leq t_g^{\frac{1}{2}}$,	\\
t_g^{1 + \e} T^{\frac{5}{6}} & if $t_g^{\frac{1}{2}} \leq T \leq t_g^{\frac{2}{3}}$,	\\
t_g^{\frac{1}{6} + \e} T^{\frac{25}{12}} & if $t_g^{\frac{2}{3}} \leq T \leq t_g^{\frac{16}{19}}$,	\\
t_g^{\frac{3}{2} + \e} T^{\frac{1}{2}} & if $t_g^{\frac{16}{19}} \leq T \leq 2t_g$,	\\
T^{2 + \e} & if $T \geq 2t_g$.
\end{dcases*}
\end{multline}
\end{proposition}

\begin{proof}
Our first approach is to apply the Cauchy--Schwarz inequality and use the bounds \eqref{eqn:largesievebounds1} and \eqref{eqn:largesievebounds2} from \hyperref[prop:largesievebounds1]{Proposition \ref*{prop:largesievebounds1}} arising from the spectral large sieve; this shows that the left-hand side of \eqref{eqn:momentsboundsinitial} is
\begin{equation}
\label{eqn:momentsboundsinitiallargesieve}
\ll_{\e} \begin{dcases*}
t_g^{1 + \e} T^{\frac{3}{2}} & if $T \leq 2t_g$,	\\
T^{\frac{5}{2} + \e} & if $T \geq 2t_g$.
\end{dcases*}
\end{equation}

Our second approach is to write
\[\frac{L\left(\frac{1}{2},f\right) L\left(\frac{1}{2},\ad g \otimes f\right)}{L(1,\ad f)} = \left(\frac{L\left(\frac{1}{2},f\right)^{12}}{L(1,\ad f)}\right)^{\frac{1}{12}} \left(\frac{L\left(\frac{1}{2},\ad g \otimes f\right)^2}{L(1,\ad f)}\right)^{\frac{1}{12}} \left(\frac{L\left(\frac{1}{2},\ad g \otimes f\right)}{L(1,\ad f)}\right)^{\frac{5}{6}}\]
(which crucially uses the nonnegativity of these $L$-functions), apply H\"{o}lder's inequality with exponents $(1/12,1/12,5/6)$, and use the bounds from \hyperref[prop:largesievebounds1]{Propositions \ref*{prop:largesievebounds1}}, \ref{prop:upperbounds}, and \ref{prop:Jutilabounds}; that is, as well as bounds arising from the spectral large sieve, we use twelfth moment bounds and bounds from $\GL_3 \times \GL_2 \leftrightsquigarrow \GL_4 \times \GL_1$ spectral reciprocity. This shows that the left-hand side of \eqref{eqn:momentsboundsinitial} is
\begin{equation}
\label{eqn:momentsboundGL3xGL2}
\ll_{\e} \begin{dcases*}
t_g^{1 + \e} T^{\frac{5}{3}} & if $T \leq t_g^{\frac{1}{4}}$,	\\
t_g^{\frac{17}{12} + \e} & if $t_g^{\frac{1}{4}} \leq T \leq t_g^{\frac{1}{2}}$,	\\
t_g^{1 + \e} T^{\frac{5}{6}} & if $t_g^{\frac{1}{2}} \leq T \leq t_g^{\frac{2}{3}}$,	\\
t_g^{\frac{1}{6} + \e} T^{\frac{25}{12}} & if $t_g^{\frac{2}{3}} \leq T \leq 2t_g$,	\\
T^{\frac{9}{4} + \e} & if $T \geq 2t_g$.
\end{dcases*}
\end{equation}
These bounds improve upon the earlier bounds in the range $T \geq t_g^{5/18}$.

Our final approach is to use \hyperref[thm:4x2reciprocity]{Theorem \ref*{thm:4x2reciprocity}}, namely $\GL_4 \times \GL_2 \leftrightsquigarrow \GL_4 \times \GL_2$ spectral reciprocity. We first take the triple of test functions $(h^{+},h^{-},h^{\hol})$ given by \eqref{eqn:firsttriple}. With this choice of test functions, the left-hand side of \eqref{eqn:4x2identity} provides an upper bound for the first and second terms on the left-hand side of \eqref{eqn:momentsboundsinitial} by positivity, as $h^{+}(t) = 0$ and $h^{\hol}(k) = 0$, while $h^{-}(t) \geq 0$ for all $t \in \R$ and $h^{-}(t) \asymp_M 1$ for $t \in [-2T,-T] \cup [T,2T]$. The first term on the right-hand side of \eqref{eqn:4x2identity} is $O_{\e}(t_g^{\e} T^2)$ as $L(1,\ad g) \ll_{\e} t_g^{\e}$. The second term is equal to zero since $h^{\hol}(k) = 0$. Finally, for the third, fourth, and fifth terms, we divide the terms into dyadic ranges and use the bounds \eqref{eqn:tildeh+bound}, \eqref{eqn:tildeh-bound}, and \eqref{eqn:tildehholbound} for the transforms $\widetilde{h}^{+}(t)$, $\widetilde{h}^{-}(t)$, and $\widetilde{h}^{\hol}(t)$. We then apply the pre-existing bounds for the left-hand side of \eqref{eqn:momentsboundsinitial} from \eqref{eqn:momentsboundsinitiallargesieve} and \eqref{eqn:momentsboundGL3xGL2}. By this method, we deduce slightly improved bounds for the first and second terms on the left-hand side of \eqref{eqn:momentsboundsinitial} in certain ranges, namely that the first and second terms are
\[\ll_{\e} \begin{dcases*}
t_g^{\frac{5}{4} + \e} T^{-\frac{1}{12}} & if $t_g^{\frac{3}{19}} \leq T \leq t_g^{\frac{1}{3}}$,	\\
t_g^{\frac{5}{6} + \e} T^{\frac{7}{6}} & if $t_g^{\frac{1}{3}} \leq T \leq t_g^{\frac{1}{2}}$,	\\
t_g^{\frac{3}{2} + \e} T^{\frac{1}{2}} & if $t_g^{\frac{16}{19}} \leq T \leq 2t_g$,	\\
T^{2 + \e} & if $T \geq 2t_g$.
\end{dcases*}\]

To deduce analogous improved bounds for the third term on the left-hand side of \eqref{eqn:momentsboundsinitial}, we use \hyperref[thm:4x2reciprocity]{Theorem \ref*{thm:4x2reciprocity}} with the triple of test functions $(h^{+},h^{-},h^{\hol})$ given by \eqref{eqn:secondtriple}. Here it is no longer the case that the left-hand side of \eqref{eqn:4x2identity} consists of only nonnegative terms. Nonetheless, the first two terms on the left-hand side of \eqref{eqn:4x2identity} as well as the contribution from the terms in the third term for which $k_f \leq M$ are $O_{M,\e}(t_g^{1 + \e} T^{1 - M})$ by \hyperref[lem:H+Lholbound]{Lemma \ref*{lem:H+} \ref*{lem:H+Lhol} \ref*{lem:H+Lholbound}} and \ref{lem:H+L+bound}, the Cauchy--Schwarz inequality, and the bounds \eqref{eqn:largesievebounds1} and \eqref{eqn:largesievebounds2} from \hyperref[prop:largesievebounds1]{Proposition \ref*{prop:largesievebounds1}} arising from the spectral large sieve. The contribution from the terms in the third term on the left-hand side of \eqref{eqn:4x2identity} for which $k_f > M$ provides an upper bound for the third term on the left-hand side of \eqref{eqn:momentsboundsinitial} by positivity via \hyperref[lem:H+Lholpos]{Lemma \ref*{lem:H+} \ref*{lem:H+Lhol} \ref*{lem:H+Lholpos}} and \ref{lem:H+Lholasymp}, noting once more that the root number of $L(s,\ad g \otimes f)$ is $i^{k_f}$, and hence $L(1/2,F \otimes f) = 0$ when $k_f \equiv 2 \pmod{4}$. Finally, the right-hand side of \eqref{eqn:4x2identity} is bounded in the same way as for the triple of test functions given by \eqref{eqn:firsttriple}.
\end{proof}

\begin{remark}
For \emph{fixed} $g$, the bound $O_{t_g,\e}(T^{2 + \e})$ for the first term on the left-hand side of \eqref{eqn:momentsboundsinitial} was previously known by work of Li \cite[Theorem 1.1]{Li09}.
\end{remark}

\subsection{Proof of \texorpdfstring{\hyperref[item:initial]{Proposition \ref*{prop:fourranges} \ref*{item:initial}}}{Proposition \ref{prop:fourranges} \ref{item:initial}}}

We now prove \hyperref[item:initial]{Proposition \ref*{prop:fourranges} \ref*{item:initial}}, namely the bound \eqref{eqn:initial} for the short initial range, via \hyperref[prop:momentsboundsinitial]{Proposition \ref*{prop:momentsboundsinitial}}.

\begin{proof}[Proof of {\hyperref[item:initial]{Proposition \ref*{prop:fourranges} (1)}}]
By the lower bound $L(1,\ad g) \gg_{\e} t_g^{-\e}$ and the asymptotic formula \eqref{eqn:Htasymp} for $H(t)$, it suffices to show that
\[\begin{drcases*}
\sum_{\substack{f \in \BB_0 \\ t_f \leq t_g^{1 - \alpha}}} \frac{1}{t_f} \frac{L\left(\frac{1}{2},f\right) L\left(\frac{1}{2},\ad g \otimes f\right)}{L(1,\ad f)} &	\\
\frac{1}{2\pi} \int\limits_{|t| \leq t_g^{1 - \alpha}} \frac{1}{1 + |t|} \left|\frac{\zeta\left(\frac{1}{2} + it\right) L\left(\frac{1}{2} + it,\ad g\right)}{\zeta(1 + 2it)}\right|^2 \, dt & \end{drcases*}
\ll_{\e} t_g^{\frac{41}{38} + \e}.\]
We dyadically decompose both the sum over $f$ and the integral over $t$, so that we are left with proving the bounds
\begin{equation}
\label{eqn:shortinitialtobeproved}
\begin{drcases*}
\sum_{\substack{f \in \BB_0 \\ T \leq t_f \leq 2T}} \frac{L\left(\frac{1}{2},f\right) L\left(\frac{1}{2},\ad g \otimes f\right)}{L(1,\ad f)} &	\\
\frac{1}{2\pi} \int\limits_{T \leq |t| \leq 2T} \left|\frac{\zeta\left(\frac{1}{2} + it\right) L\left(\frac{1}{2} + it,\ad g\right)}{\zeta(1 + 2it)}\right|^2 \, dt & \end{drcases*}
\ll_{\e} t_g^{\frac{41}{38} + \e} T
\end{equation}
for $T \leq t_g^{1 - \alpha}$. These desired bounds are a consequence of \hyperref[prop:momentsboundsinitial]{Proposition \ref*{prop:momentsboundsinitial}}, which gives these bounds when either $T \asymp t_g^{3/19}$ or $T \asymp t_g^{16/19}$ and gives stronger bounds otherwise.
\end{proof}

\begin{remark}
\label{rem:shortinitialoptimal}
Obtaining the essentially optimal bound $\|g\|_4 \ll_{\e} t_g^{\e}$ for the $L^4$-norm would require us to improve the bounds for the left-hand side of \eqref{eqn:shortinitialtobeproved} to $O_{\e}(t_g^{1 + \e} T)$ for $T \leq t_g^{1 - \alpha}$. \hyperref[prop:momentsboundsinitial]{Proposition \ref*{prop:momentsboundsinitial}} gives these bounds apart from the ranges $T \leq t_g^{3/13}$ and $t_g^{10/13} \leq T \leq t_g^{1 - \alpha}$.
\end{remark}

\begin{remark}
\label{rem:Eisensteindihedral}
Improvements on the bounds \eqref{eqn:shortinitialtobeproved} are known when $g$ is either an Eisenstein series or a dihedral Hecke--Maa\ss{} cusp form, which is crucial to the \emph{unconditional} $L^4$-norm asymptotic formul\ae{} that have been proven in these settings \cite{DK20,HK20}. The key difference behind these improvements is the factorisation of the $\GL_3 \times \GL_2$ Rankin--Selberg $L$-function into the product of $L$-functions of lower degree in these settings: in particular, if $g(z) = E(z,1/2 + it_g)$ is an Eisenstein series, then $L(1/2,\ad g \otimes f) = L(1/2,f) |L(1/2 + 2it_g,f)|^2$. This allows for additionally flexibility in applying H\"{o}lder's inequality in bounding \eqref{eqn:shortinitialtobeproved}; for instance, in place of $\GL_3 \times \GL_2 \leftrightsquigarrow \GL_4 \times \GL_1$ spectral reciprocity, one can do better by instead inputting the individual Weyl--strength subconvex bounds $L(1/2,f) \ll_{\e} t_f^{1/3 + \e}$ and $\zeta(1/2 + it) \ll_{\e} (1 + |t|)^{1/6 + \e}$ together with the second moment bounds
\begin{equation}
\label{eqn:EisensteinJutila}
\sum_{\substack{f \in \BB_0 \\ T \leq t_f \leq 2T}} \frac{\left|L\left(\frac{1}{2} + 2it_g,f\right)\right|^2}{L(1,\ad f)} + \frac{1}{2\pi} \int\limits_{T \leq |t| \leq 2T} \left|\frac{\zeta\left(\frac{1}{2} + it + 2it_g\right)^2}{\zeta(1 + 2it)}\right|^2 \, dt \ll_{\e} \begin{dcases*}
t_g^{\frac{2}{3} + \e} & if $1 \leq T \leq t_g^{\frac{1}{3}}$,	\\
T^{2 + \e} & if $t_g^{\frac{1}{3}} \leq T \leq t_g$,
\end{dcases*}
\end{equation}
due to Jutila \cite[Theorem]{Jut04a}. This gives an effective treatment of the portion $T \leq t_g^{1/2}$ of the short initial range; coupled with $\GL_4 \times \GL_2 \leftrightsquigarrow \GL_4 \times \GL_2$ spectral reciprocity for the remaining portion, this shows that the short initial range in the Eisenstein setting is negligibly small.
\end{remark}

\section{Bounds for Mixed Moments of \texorpdfstring{$L$}{L}-Functions in the Bulk Range}

\subsection{Proof of \texorpdfstring{\hyperref[item:bulk]{Proposition \ref*{prop:fourranges} \ref*{item:bulk}}}{Proposition \ref{prop:fourranges} \ref{item:bulk}}}

The proof of \hyperref[item:bulk]{Proposition \ref*{prop:fourranges} \ref*{item:bulk}}, namely the bound \eqref{eqn:bulk} for the bulk range, follows by modifying earlier work of Buttcane and the second author \cite{BuK17}, where the \emph{asymptotic formula} $2 + o(1)$ was proven for the mixed moment of $L$-functions in the bulk range \eqref{eqn:bulk} under the assumption of the generalised Lindel\"{o}f hypothesis (GLH). We explain the minor modifications required to weaken this to an unconditional upper bound.

\begin{proof}[Proof of {\hyperref[item:bulk]{Proposition \ref*{prop:fourranges} (2)}}]
The bound \eqref{eqn:bulk} for the bulk range follows by modifying the main result of \cite{BuK17}. There are several places in \cite{BuK17} where we must remove the assumption of the GLH, which we list below. For the sake of consistency, we use the notation in \cite{BuK17}; in particular, $g$ is replaced by $f$ and $t_g$ is replaced by $T$.
\begin{enumerate}[leftmargin=*,label=\textup{(\arabic*)}]
\item It is stated in \cite[Lemma 2.1]{BuK17} and proven in \cite[Section 4]{BuK17} that under GLH,
\[\int_{-\infty}^{\infty} \left|\left\langle f^2, E\left(\cdot, \frac{1}{2} + it\right)\right\rangle\right|^2 \, dt \ll_{\e} T^{-1 + \e}.\]
We do not need to bound this separately.
\item In \cite[p.~1494]{BuK17}, it is stated that under GLH,
\[T^{-1/2 + \beta/2 + \alpha + \epsilon} \sum_{T^{1 - \alpha} < t_j < T^{1 + \e}} \left|H(t_j) W(t_j) \right| \left|\sum_{m,r \geq 1} \frac{\lambda_j(m) A_f(r,m)}{rm^{1/2}} V_2(r^2 m,t_j)\right| \ll T^{-1/2 + \beta/2 + \alpha/2 + \epsilon}.\]
To bound this unconditionally, we use the fact that $V_2(r^2 m,t_j)$ is negligibly small unless $r^2 m \leq T^{2 + \e} (1 + |2T - t_j|)$, in which case it is $O(1)$. We then apply the Cauchy--Schwarz inequality and the spectral large sieve, as in the proof of \hyperref[item:largesievebounds1]{Proposition \ref*{prop:largesievebounds1} \ref*{item:largesievebounds1}}, to unconditionally obtain the weaker bound $O(T^{\beta/2 + \alpha + \epsilon})$.
\item In \cite[p.~1495]{BuK17}, it is shown under GLH that the error term in \cite[(6.1)]{BuK17} is $O(T^{-1/2})$. This same bound holds unconditionally by using the fact that for $\Re(s_1) = \Re(s_2) = \epsilon$, $L(1 + 2s_2,\sym^2 f) L(1 + s_1 + s_2,\sym^2 f) \ll T^{\epsilon}$ by \cite[Theorem 1.1]{Li09}.
\item In \cite[p.~1496]{BuK17}, it is shown under GLH that in \cite[(6.2)]{BuK17} there is an error term of size $O(T^{-(1 - \beta)/10 + \epsilon})$ from shifting the line of integration to $\Re(s_1) = -\frac{1}{10}$. Instead using the convexity bound for $L(1 + s_1 + s_2,\sym^2 f)$ with $\Re(s_1) = -\frac{1}{10}$ and $\Re(s_2) = \epsilon$, we get the unconditional error term $O(T^{\beta/10 + \epsilon})$. Similarly, after shifting the line of integration to $\Re(s_1) = -\frac{1}{10}$, we get the unconditional error term $O(T^{\epsilon})$ instead of $O(T^{-1/20})$.
\item In \cite[Section 7]{BuK17}, it is shown that the continuous spectrum has a negligible contribution under GLH. However we do not need to bound this separately, since it is present in \eqref{eqn:bulk}.
\item In \cite[p.~1500]{BuK17}, it is stated that the quantity in \cite[(9.7)]{BuK17} is trivially $O(T^{\epsilon})$ unconditionally, which suffices for our purposes; we do not need a power-saving for this, which is given under GLH. Here $\epsilon$ is dependent on $\alpha$ and is arbitrarily small for arbitrarily small $\alpha$.
\item In \cite[Section 9.4]{BuK17}, we simply use the unconditional trivial bound $O(T^{\epsilon})$ for \cite[(9.35)]{BuK17} instead of the power-saving $O(T^{-\delta})$ proven under GLH. Again, $\epsilon$ is dependent on $\alpha$ and is arbitrarily small for arbitrarily small $\alpha$.
\qedhere
\end{enumerate}
\end{proof}

\begin{remark}
The presence of the bound $O_{\e}(t_g^{c(\alpha) + \e})$ on the right-hand side of \eqref{eqn:bulk} with $\lim_{\alpha \to 0} c(\alpha) = 0$ is due to the fact that the $\epsilon$-convention is used in \cite[Section 9]{BuK17} (see, in particular, \cite[p.~1499]{BuK17}). Indeed, from this usage of the $\epsilon$-convention, our unconditional modification of \cite{BuK17} shows that for all $\beta > 0$, there exists some $\alpha > 0$ such that
\begin{multline*}
\sum_{\substack{f \in \BB_0 \\ t_g^{1 - \alpha} \leq t_f \leq 2t_g - t_g^{1 - \alpha}}} \frac{L\left(\frac{1}{2},f\right) L\left(\frac{1}{2},\ad g \otimes f\right)}{L(1,\ad f) L(1,\ad g)^2} H(t_f)	\\
+ \frac{1}{2\pi} \int\limits_{t_g^{1 - \alpha} \leq |t| \leq 2t_g - t_g^{1 - \alpha}} \left|\frac{\zeta\left(\frac{1}{2} + it\right) L\left(\frac{1}{2} + it,\ad g\right)}{\zeta(1 + 2it) L(1,\ad g)}\right|^2 H(t) \, dt \ll_{\e} t_g^{\beta + \e}.
\end{multline*}
With more care, we could make the dependence of $c(\alpha)$ on $\alpha$ in \eqref{eqn:bulk} explicit by a more precise treatment of \cite[Section 9]{BuK17}\footnote{For instance, with a little work, it can be shown that $c(\alpha) \leq 9\alpha/2$ is admissible. With an overhaul of the methods in \cite{BuK17}, we expect that it should be possible to take $c(\alpha) \leq \alpha/2$.}. This, however, is not necessary for our purposes, since our bounds \eqref{eqn:initial} and \eqref{eqn:transition} for the short initial and short transition ranges are worse than our bound \eqref{eqn:bulk} for the bulk range.
\end{remark}

\section{Test Functions and Transforms for the Short Transition Range}

Our treatment of the short transition range once more requires the usage of $\GL_4 \times \GL_2 \leftrightsquigarrow \GL_4 \times \GL_2$ spectral reciprocity in the guise of \hyperref[thm:4x2reciprocity]{Theorem \ref*{thm:4x2reciprocity}}. We choose a triple of test functions $(h^{+},h^{-},h^{\hol})$ that localises to short intervals $[T,T + U]$ with $T \sim 2t_g$ and subsequently bound the associated transforms $(\widetilde{h}^{+},\widetilde{h}^{-},\widetilde{h}^{\hol})$ as in \eqref{eqn:tildehpmdefeq} and \eqref{eqn:tildehholdefeq}. Just as with our treatment of the short initial range, the hybrid nature of the problem at hand requires us to obtain bounds that are uniform in both the dyadic parameter $T$ and the spectral parameter $t_g$.

\subsection{Test Functions}

We define the following triple of test functions $(h^{+},h^{-},h^{\hol})$:
\begin{align}
\label{eqn:fourthtriple}
h^{+}(t) &= 0, & h^{-}(t) & = \sum_{\pm} e^{-\frac{(t \pm T)^2}{U^2}} \prod_{j = 1}^{M} \left(\frac{t^2 + \left(j - \frac{1}{2}\right)^2}{T^2}\right)^2, & h^{\hol}(k) & = 0.
\end{align}
Here $M \in \N$ is a fixed positive integer and $T,U$ are auxiliary parameters such that $2t_g - t_g^{1 - \alpha} \leq T \leq 2t_g$ and $U = t_g - \frac{T}{2} + 1$. In this section, we assume that $U>T^{1/3+\delta}$ for some $\delta>0$ which we may choose as small as we like. With this choice of triple of test functions, we have that $H^{+}(x) \coloneqq (\Kscr^{+} h^{+})(x) + (\Kscr^{\hol} h^{\hol})(x) = 0$, while we have the following bounds for $H^{-}(x) \coloneqq (\Kscr^{-} h^{-})(x)$ and its derivatives.

\begin{lemma}
\label{lem:k-hbounds}
For $j \in \N_0$ with $j \leq M$, we have that
\begin{equation}
\label{eqn:H-derivbound}
x^j \frac{d^j}{dx^j} H^{-}(x) \ll_M \begin{dcases*}
U \min\left\{\left(\frac{x}{T}\right)^{M/2}, \left(\frac{x}{T}\right)^{-M/2}, T^{-M/10}\right\}   & if $|2\pi x - T| > U \log T$,	\\
T \left(\frac{T}{U}\right)^j \left(1 + \frac{|2\pi x - T|}{U}\right)^{4M} e^{-\left(\frac{2\pi x - T}{U}\right)^2} & if $|2\pi x - T| \leq U \log T$.
\end{dcases*}
\end{equation}
\end{lemma}
\begin{proof}
This can be seen from the statement and proof of \cite[Lemma 12.2]{HK20}, which in turn follows ideas from the proof of \cite[Lemma 4]{BLM19}, after correcting a small error. The function $H^{-}(x)$ from \cite{HK20} corresponds to the function $H_T(4\pi x)$ from \cite{BLM19}, but unfortunately the factor $4\pi$ was not accounted for in \cite[Lemma 12.2]{HK20}. After making this correction, we see that the required bound for $|2\pi x - T| \leq U \log T$ is the same as the corresponding bound in the statement of \cite[Lemma 4]{BLM19}. For $|2\pi x - T| > U \log T$ and $4\pi x \le 1$ or $4 \pi x > T^{13/12}$, the required bound is again available in the statement of \cite[Lemma 4]{BLM19}. For $|2\pi x - T| > U \log T$ and $1< 4\pi x < T^{13/12}$, one must look in the proof of \cite[Lemma 4]{BLM19}, specifically at the bound given therein for $x^\beta \frac{d^{\alpha+\gamma}}{d x^{\alpha+\gamma}} x^n h_{\mathrm{spec}}(\frac{x}{2})$, after which it is remarked that this decays faster than any power of $T$.
\end{proof}

Using \hyperref[lem:k-hbounds]{Lemma \ref*{lem:k-hbounds}} and integration by parts, we deduce the following bounds for the Mellin transform of $H^-$.

\begin{lemma}[{\cite[Corollary 12.10]{HK20}}]
\label{lem:mellink-h}
For $s = \sigma + i\tau$, $\widehat{H^-}(s)$ is holomorphic in the strip $-M/2 < \Re(s) < M/2$ and satisfies the bounds
\begin{equation}
\label{eqn:mellink-hline}
\widehat{H^-}(s) \ll_{\sigma,M} \begin{dcases*}
U T^{\sigma} & for $|\tau| \leq \frac{T}{U}$,	\\
U T^{\sigma} \left(\frac{|\tau| U}{T}\right)^{-M/2} & for $|\tau| \geq \frac{T}{U}$.
\end{dcases*}
\end{equation}
\end{lemma}

\subsection{\texorpdfstring{$\mathrm{GL}_4 \times \mathrm{GL}_2$}{GL\9040\204 \80\327 GL\9040\202} Transforms}
\label{sect:GL4xGL2transforms}

We now determine the behaviour of $(\widetilde{h}^{+},\widetilde{h}^{-},\widetilde{h}^{\hol})$ as in \eqref{eqn:tildehpmdefeq} and \eqref{eqn:tildehholdefeq} with $(h^{+},h^{-},h^{\hol})$ the triple of test functions \eqref{eqn:fourthtriple}.

\begin{lemma}
\label{lem:tildehpmboundstransition}
Let $g$ be a Hecke--Maa\ss{} cusp form on $\Gamma \backslash \Hb$ with spectral parameter $t_g$. Fix $\delta > 0$ and $0 < \e < \delta/3$, and let $(h^{+},h^{-},h^{\hol})$ be the triple of test functions \eqref{eqn:fourthtriple} with $M \geq 1000$ a sufficiently large positive integer dependent on $\delta$ and $\e$, $2t_g - t_g^{1 - \delta} \leq T \leq 2t_g - t_g^{1/3 + \delta}$, and $U = t_g - \frac{T}{2} + 1$. Then for $F = \ad g$ and $(\widetilde{h}^{+},\widetilde{h}^{-},\widetilde{h}^{\hol})$ as in \eqref{eqn:tildehpmdefeq} and \eqref{eqn:tildehholdefeq}, we have that
\begin{align}
\label{eqn:tildeh+boundtransition}
\widetilde{h}^{+}(t) & \ll_{\e} \begin{dcases*}
U^{1 + \e} & for $|t| \leq \left(\frac{T}{U}\right)^{\frac{1}{2} + \e}$,	\\
(|t| T)^{-100} & for $|t| \geq \left(\frac{T}{U}\right)^{\frac{1}{2} + \e}$,
\end{dcases*}	\\
\label{eqn:tildeh-boundtransition}
\widetilde{h}^{-}(t) & \ll_{\e} \begin{dcases*}
U^{1 + \e} & for $|t| \leq T^{\e}$,	\\
(|t| T)^{-100} & for $|t| \geq T^{\e}$,
\end{dcases*}	\\
\label{eqn:tildehholboundtransition}
\widetilde{h}^{\hol}(k) & \ll_{\e} \begin{dcases*}
U^{1 + \e} & for $k \leq \left(\frac{T}{U}\right)^{\frac{1}{2} + \e}$,	\\
(kT)^{-100} & for $k \geq \left(\frac{T}{U}\right)^{\frac{1}{2} + \e}$.
\end{dcases*}
\end{align}
\end{lemma}

The proof of \hyperref[lem:tildehpmboundstransition]{Lemma \ref*{lem:tildehpmboundstransition}} is rather involved. We first prove the bounds \eqref{eqn:tildeh-boundtransition} as well as weakened forms of the bounds \eqref{eqn:tildeh+boundtransition} and \eqref{eqn:tildehholboundtransition}; we then refine these latter two bounds.

Throughout, we use the $\e$-convention: $\e$ denotes an arbitrarily small constant whose value may change from occurrence to occurrence. The same goes for the auxiliary constant $M$, which is an arbitrarily large constant, and an auxiliary constant $A$, which is also an arbitrarily large constant. We shall additionally always impose the conditions given in \hyperref[lem:tildehpmboundstransition]{Lemma \ref*{lem:tildehpmboundstransition}} on the choice of triple of test functions $(h^{+},h^{-},h^{\hol})$ and the size of the parameters $T$ and $U$.

\subsubsection{Weak Bounds}

We first prove the bounds \eqref{eqn:tildeh-boundtransition} for $\widetilde{h}^{-}(t)$.

\begin{lemma}
\label{lem:-small}
Let $t \in \R$. We have that
\[\widetilde{h}^{-}(t) \ll_{M,\e} \begin{dcases*}
U^{1 + \e} & for $|t| \leq T^{\e}$,	\\
(|t|T)^{-100} & for $|t| \geq T^{\e}$.
\end{dcases*}\]
\end{lemma}

\begin{proof}
From \eqref{eqn:tildehpmdefeq}, we have that for $0 < \sigma < 1$,
\begin{equation}
\label{eqn:l-h-bound}
\widetilde{h}^{-}(t) = \frac{1}{2\pi i} \int_{\sigma - i\infty}^{\sigma + i\infty} \widehat{H^{-}}(s) \widehat{\JJ_t^{-}}(s) \sum_{\pm} G^{\pm}\left(\frac{1 - s}{2}\right) \Gscr_{\mu_F}^{\pm}\left(\frac{1 - s}{2}\right) \, ds.
\end{equation}

First we show that the required bound for $\widetilde{h}^{-}(t)$ holds when $|t| >t_g^{1000}$ by shifting the line of integration to $\sigma = -\left \lfloor{\frac{M}{1000}}\right \rfloor+\frac{1}{2}$. The residues of the poles crossed are $O_M((|t|T)^{-100})$ provided that $M$ is sufficiently large since
\begin{multline*}
\Res_{s = 2(\pm it - \ell)} \widehat{H^{-}}(s) \widehat{\JJ_t^{-}}(s) \sum_{\pm} G^{\pm}\left(\frac{1 - s}{2}\right) \Gscr_{\mu_F}^{\pm}\left(\frac{1 - s}{2}\right)	\\
\ll_{M} UT^{-2\ell} \left(\frac{T}{U}\right)^{\frac{M}{2}} (1 + |t|)^{-\frac{M}{2} + \ell - \frac{1}{2}} ((1 + |t + 4t_g|)(1 + |t - 4t_g|))^{\ell}.
\end{multline*}
On the new line of integration $\sigma = -\left \lfloor{\frac{M}{1000}}\right \rfloor + \frac{1}{2}$, we apply \hyperref[lem:mellink-h]{Lemma \ref*{lem:mellink-h}} and \hyperref[lem:mellinbounds]{Lemma \ref*{lem:mellinbounds}} and then integrate trivially to see that the integral is $O((|t|T)^{-100})$ for large enough $M$. This is because the $|\tau|^{-M/2}$ term in \eqref{eqn:mellink-hline} ensures convergence and the term $((1 + |\tau + 2t|)(1 + |\tau - 2t|))^{\sigma}$ in \eqref{eqn:tildeintegrandpm} gives the saving, since one of $|\tau + 2t|$ and $|\tau - 2t|$ is $\gg |t|$.

Now assume that $T^{\e} \leq |t| \leq t_g^{1000}$. The contribution of the range $|\tau| > T^{1 + \e}/U$ is $O_M((|t|T)^{-100})$ provided that $M$ is sufficiently large via \hyperref[lem:mellink-h]{Lemma \ref*{lem:mellink-h}}. The contribution of the range $T^{\e}/2 \leq |\tau| \leq T^{1 + \e}/U$ is $O_M(|(t|T)^{-100})$ since $\Omega^{-,-}(\tau, t, t_g) \geq |\tau|$ in this range, so that the integrand in \eqref{eqn:l-h-bound} decays exponentially. For the range $|\tau| < T^{\e}/2$, we deform the $|\tau| < T^{\e}/2$ segment to a contour going horizontally from $-iT^{\e}/2$ to $-B - iT^{\e}/2$, vertically to $-B + iT^{\e}/2$, and horizontally to $iT^{\e}/2$, for a large constant $0 < B < M/4$. Since the poles of $\widehat{\JJ^\pm_t}(s)$ are at $s = 2(-\ell \pm it)$ for $\ell \in \N_0$, we do not cross any poles as we deform the contour in this way. The horizontal parts of the contour contribute $O_M((|t|T)^{-100})$ by the argument above. The vertical part at $\Re(s) = -B$ contributes $O(U T^{-B} t_g^B |t|^{-B - 1})$ by \hyperref[lem:mellinbounds]{Lemma \ref*{lem:mellinbounds}}. Taking $B$ sufficiently large ensures that this is $O_M((|t|T)^{-100})$ since $|t| > T^{\e}$. 

Finally, we assume that $|t| \leq T^{\e}$. The contribution of the range $|\tau| \geq T^{\e}/2$ is treated by the same method as above, while we estimate trivially the contribution of $|\tau| < T^{\e}/2$ segment by taking $\sigma=\e$. We get that this contribution is $O(U^{1 + \e})$.
\end{proof}

We next prove the bounds \eqref{eqn:tildeh+boundtransition} for $\widetilde{h}^{+}(t)$ apart from the range $T^{\e} \leq |t| \leq T^{1 + \e}/U$.

\begin{lemma}
\label{lem:+crude}
Let $t \in \R$. We have that
\[\widetilde{h}^{+}(t) \ll \begin{dcases*}
U^{1 + \e} & for $|t| \leq T^{\e}$,	\\
T^{1 + \e} & for $T^{\e} \leq |t| \leq \frac{T^{1 + \e}}{U}$,	\\
(|t| T)^{-100} & for $|t| \geq \frac{T^{1 + \e}}{U}$.
\end{dcases*}\]
\end{lemma}

\begin{proof}
From \eqref{eqn:tildehpmdefeq}, we have that for $0 < \sigma < 1$,
\begin{equation}
\label{eqn:l+h+bound}
\widetilde{h}^{+}(t) = \frac{1}{2\pi i}\int_{\sigma - i\infty}^{\sigma + i\infty} \widehat{H^{-}}(s) \widehat{\JJ_t^{+}}(s) \sum_{\pm} G^{\mp}\left(\frac{1 - s}{2}\right) \Gscr_{\mu_F}^{\pm}\left(\frac{1 - s}{2}\right) \, ds.
\end{equation}

First suppose that $|t| > T^{\e}$. The contribution of the range $|\tau|> |t|^\e T^{1 + \e}/U$ is $O((|t|T)^{-100})$ by \hyperref[lem:mellink-h]{Lemma \ref*{lem:mellink-h}} together with the bounds \hyperref[lem:mellinbounds]{Lemma \ref*{lem:mellinbounds}}. Now consider the range $|\tau|\leq |t|^\e T^{1 + \e}/U$. The contribution of the sub-range $ |\tau| < |t|$ is $O((|t|T)^{-100})$ since then $\Omega^{+,+}(\tau, t, t_g) \geq |t|$, so that the integrand in \eqref{eqn:l+h+bound} decays exponentially. Thus in the case $|t| > T^{\e}$, we have shown that $\widetilde{h}^{+}(t) = O((|t|T)^{-100})$ unless $|t| \leq |\tau| \leq |t|^\e T^{1 + \e} / U$, which implies that $|t| \leq T^{1 + \e} / U$ and $|\tau| \leq T^{1 + \e} / U$, for redefined values of $\e$. We can crudely bound the integral \eqref{eqn:l+h+bound} over the range $|\tau|\leq T^{1 + \e}/U$ by taking $\sigma=\e$ and using \hyperref[lem:mellinbounds]{Lemma \ref*{lem:mellinbounds}}, yielding the bound $O(T^{1 + \e})$.

Now suppose that $|t| \leq T^{\e}$. The contribution of the range $|\tau|> T^{1 + \e}/U$ is $O(T^{-100})$ as shown above. We can bound the integral \eqref{eqn:l+h+bound} over the range $|\tau|\leq T^{1 + \e}/U$ by taking $\sigma=\e$ and using \hyperref[lem:mellinbounds]{Lemma \ref*{lem:mellinbounds}} to get
\[\widetilde{h}^{+}(t) \ll \int_{|\tau| \leq \frac{T^{1 + \e}}{U}} U T^{\e} (1 + |\tau|)^{-1} \, d\tau + O(T^{-A}) \ll U^{1 + \e}.\qedhere\]
\end{proof}

Similarly, we prove the bounds \eqref{eqn:tildehholboundtransition} for $\widetilde{h}^{\hol}(k)$ apart from the range $T^{\e} \leq k \leq T^{1 + \e}/U$.

\begin{lemma}
\label{lem:holcrude}
Let $k \in 2\N$. We have that
\[\widetilde{h}^{\hol}(k) \ll \begin{dcases*}
U^{1 + \e} & for $k \leq T^{\e}$,	\\
T^{1 + \e} & for $T^{\e} \leq k \leq \frac{T^{1 + \e}}{U}$,	\\
(kT)^{-100} & for $k > \frac{T^{1 + \e}}{U}$.
\end{dcases*}\]
\end{lemma}

\begin{proof} 
The proof follows similar ideas, so we do not give full details. By \eqref{eqn:tildehholdefeq},
\[\widetilde{h}^{\hol}(k) = \frac{1}{2\pi i}\int_{\sigma - i\infty}^{\sigma + i\infty} \widehat{H^{-}}(s) \widehat{\JJ_k^{\hol}}(s) \sum_{\pm} G^{\mp}\left(\frac{1 - s}{2}\right) \Gscr_{\mu_F}^{\pm}\left(\frac{1 - s}{2}\right) \, ds,\]
for $0 < \sigma < 1$. First we show that $\widetilde{h}^{\hol}(k) \ll (kT)^{-100}$ for $k > t_g^{1000}$, by shifting the contour to the left, noting that for $k \in 2\N$ and $\ell \in \N_0$,
\begin{align*}
\Res_{s = 1 - k - 2\ell} \widehat{H^{-}}(s) \widehat{\JJ_k^{\hol}}(s) \sum_{\pm} G^{\mp}\left(\frac{1 - s}{2}\right) \Gscr_{\mu_F}^{\pm}\left(\frac{1 - s}{2}\right) \ll UT^{1 - k - 2\ell} t_g^{k - 1 + 2\ell} \left(\frac{k - 1}{2\pi e}\right)^{1 - k} k^{-1/2}.
\end{align*}
In the remaining range $k \leq t_g^{1000}$, we use \hyperref[lem:mellink-h]{Lemma \ref*{lem:mellink-h}} in order to restrict to the integral to $|\tau| < T^{1 + \e}/U$ up to an error of $O(T^{-A})$ for any $A > 0$. Since $k \leq t_g^{1000}$, this error is $O((kT)^{-100})$. 

For $k > T^{\e}$, we shift the $|\tau| < T^{1 + \e}/U$ integral to the left to see that we can assume $k \leq T^{1 + \e}/U$. For such $k$, we trivially estimate the integral at $\sigma=\e$ to get that $\widetilde{h}^{\hol}(k) \ll T^{1 + \e}$.

For $k \leq T^{\e}$, we trivially estimate the integral at $\sigma=\e$ to get that $\widetilde{h}^{\hol}(k) \ll U^{1 + \e}$.
\end{proof}

\subsubsection{Strong Bounds for $\widetilde{h}^{+}(t)$}

Our goal is to prove refined estimates for $\widetilde{h}^+(t)$ in the range $T^{\e} \leq |t| \leq T^{1 + \e}/U$. In this range, any error term $O(T^{-A})$ for arbitrarily large $A$ may also be written as $O((|t| T)^{-A'})$ for arbitrarily large $A'$. We begin by noting that
\begin{multline}
\label{eqn:GtoJ}
\sum_{\pm} G^{\mp}\left(\frac{1 - s}{2}\right) \Gscr_{\mu_F}^{\pm}\left(\frac{1 - s}{2}\right)	\\
= 4 (\cosh 2\pi t_g + 1) (2\pi)^{2(s - 1)} \Gamma\left(\frac{1 - s}{2}\right)^2 \Gamma\left(\frac{1 - s}{2} + 2it_g\right) \Gamma\left(\frac{1 - s}{2} - 2it_g\right) \sin \frac{\pi s}{2}	\\
= 4 \widehat{\JJ_0^{+}}(1 - s) \widehat{\JJ_{2t_g}^{-}}(1 - s) + 4 \widehat{\JJ_0^{-}}(1 - s) \widehat{\JJ_{2t_g}^{+}}(1 - s).
\end{multline}
The latter term turns out to give a negligible contribution, while for the former term, we make use of the following asymptotic formula.

\begin{lemma}
For $s = \sigma + i\tau$ with $|\tau| \leq T^{1 + \e}/U$ and $0 < \sigma < 1$, there exist constants $c_{j,j_1}$ such that
\begin{equation}
\label{eqn:stir4}
\widehat{\JJ_{2t_g}^-}(s) = \frac{1}{2} \left(\frac{t_g}{\pi}\right)^{s - 1} \left(1 + \sum_{\substack{2 \leq j \leq M\\0 \leq j_1\le3j/2}} c_{j,j_1} \frac{s^{j_1} }{t_g^j}\right) + O(T^{-A})
\end{equation}
for all $A > 0$ and for all $M$ sufficiently large with respect to $A$.
\end{lemma}

\begin{proof}
By Stirling's formula, there exist constants $c_{j,j_1}$ such that
\begin{multline*}
\log \Gamma\left(\frac{s}{2} + 2it_g\right) + \log \Gamma\left(\frac{s}{2} - 2it_g\right)	\\
= \left(\frac{s}{2} - \frac{1}{2} + 2it_g\right) \log \left(\frac{s}{2} + 2it_g\right) + \left(\frac{s}{2} - \frac{1}{2} - 2it_g\right) \log\left(\frac{s}{2} - 2it_g\right)	\\
- s + \log 2\pi + \sum_{\substack{2 \leq j \leq M\\ 0 \leq j_1 \leq j}} c_{j,j_1} \frac{s^{j_1}}{t_g^j} + O(T^{-A}),
\end{multline*}
for any $A > 0$, provided that $M$ is sufficiently large with respect to $A$. The complex logarithm $\log z \coloneqq \log|z| + i\arg z$ for $z \in \C \setminus (-\infty,0]$ satisfies the identity
\[\log\left(\frac{s}{2} \pm 2it_g\right) = \log (\pm 2it_g) +\log\left(1 \mp \frac{is}{4t_g}\right) = \log 2t_g \pm i \frac{\pi}{2} - \sum_{j = 1}^{\infty} \frac{1}{j} \left(\frac{\pm i s}{4t_g}\right)^j.\]
It follows that
\begin{multline}
\label{eqn:stir2}
\left(\frac{s}{2} - \frac{1}{2} + 2it_g\right) \log \left(\frac{s}{2} + 2it_g\right) + \left(\frac{s}{2} - \frac{1}{2} - 2it_g\right) \log \left(\frac{s}{2} - 2it_g\right) - s + \log 2\pi \\
= (s - 1) \log 2t_g - 2\pi t_g + \log 2\pi + \sum_{\substack{2 \leq j \leq M \\ 0 \leq j_1 \leq j + 1}} c_{j,j_1} \frac{s^{j_1}}{t_g^j} + O(T^{-A})
\end{multline}
for some constants $c_{j,j_1}$. We take the exponential of the right hand side of \eqref{eqn:stir2}. Writing 
\[\exp\left(\sum_{\substack{2 \leq j \leq M \\ 0 \leq j_1 \leq j + 1}} c_{j,j_1} \frac{s^{j_1}}{t_g^j} \right) = \sum_{\ell = 0}^{\infty} \frac{1}{\ell!} \left(\sum_{\substack{2 \leq j \leq M \\ 0 \leq j_1 \leq j + 1}} c_{j,j_1} \frac{s^{j_1}}{t_g^j}\right)^\ell\]
and expanding, we obtain the desired expansion \eqref{eqn:stir4} (for different values of $c_{j,j_1}$ and $M$) upon recalling the definition \eqref{eqn:JJr-Mellin} of $\widehat{\JJ_{2t_g}^{-}}(s)$. Thus we have written $\widehat{\JJ_{2t_g}^{-}}(s)$ as $\frac{1}{2} (t_g/\pi)^{s - 1}$ plus similar but smaller functions. This is because the largest term in the series in \eqref{eqn:stir4}, corresponding to $j = 2$ and $j_1 = 3$, is of size
\[\frac{|s|^3}{t_g^2}\ll \frac{T^{1 + \e}}{U^3} \ll T^{-3\delta + \e},\]
using the assumptions $|\tau| \ll T^{1 + \e} / U$ and $U \geq T^{1/3 + \delta}$. 
\end{proof}

We require the following bounds for $\JJ_t^+(x)$.

\begin{lemma}[{\cite[Proposition 9, Remark 6]{HM06}}]
For $t \in \R$, we have that
\begin{equation}
\label{eqn:Jt+bound}
\JJ_t^+(x) \ll_{\e} \begin{dcases*}
(1 + |t|)^{\e} (x^\e+x^{-\e}) & if $0 < 4\pi x \leq 1 + 4|t|^2$, \\
\frac{1}{\sqrt{x}} & if $4\pi x > 1 + 4|t|^2$.
\end{dcases*}
\end{equation}
\end{lemma}

\begin{proof}
This follows by using the integral representation of $Y_0(x)$ found in \cite[p.~26]{Wat44} in the case $x \geq 1$, and by using the power series representation \cite[8.402, 8.403.2]{GR15} in the case $0 < x < 10^6$.
\end{proof}

Now we work towards the improved estimates \eqref{eqn:tildeh+boundtransition} for $\widetilde{h}^{+}(t)$.

\begin{lemma}
\label{lem:tildeh+truncate}
For $T^{\e} \leq |t| \leq T^{1 + \e}/U$, there are constants $c_{j,j_1}$ such that
\begin{equation}
\label{eqn:newL} \widetilde{h}^{+}(t) = \sum_{\substack{0 \leq j \leq M \\ 0 \leq j_1 \leq 3j/2}} c_{j,j_1} (\Lscr^+ \widetilde{\phi}_{j,j_1})(t)   + O(T^{-A}),
\end{equation}
where
\begin{equation}
\label{eqn:newL'}
(\Lscr^+ \widetilde{\phi}_{j,j_1})(t) \coloneqq \int_0^\infty \JJ_t^+(x) \left( \frac{x}{t_g^{j+1} } \int_0^\infty  H_{j_1}^-(y)Y_0\left(\frac{4\pi^2 xy}{t_g} \right) \, dy \right) \, \frac{dx}{x}
\end{equation}
and 
\[H_{j_1}^-(y) \coloneqq y^{j_1}\frac{d^{j_1}}{dy^{j_1}} H^-(y), \]
for all $A > 0$ and for all $M$ sufficiently large with respect to $A$.
\end{lemma}

Here $Y_{\alpha}(x)$ denotes the Bessel function of the second kind.

\begin{proof}
Via \eqref{eqn:GtoJ}, we may write
\begin{equation}
\label{eqn:L+H+def}
\widetilde{h}^{+}(t) = \int_0^\infty \JJ_t^+(x) \phi(x) \, \frac{dx}{x} + \frac{2}{\pi i} \int_{\sigma - i\infty}^{\sigma + i\infty} \widehat{H^{-}}(1 - s) \widehat{\JJ_t^{+}}(1 - s) \widehat{\JJ_0^{-}}(s) \widehat{\JJ_{2t_g}^{+}}(s) \, ds,
\end{equation}
where
\begin{equation}
\label{eqn:phidefeq}
\phi(x) \coloneqq \frac{2}{\pi i} \int_{\sigma - i\infty}^{\sigma + i\infty} \widehat{H^{-}}(1 - s) \widehat{\JJ_0^{+}}(s) \widehat{\JJ_{2t_g}^{-}}(s) x^{1 - s} \, ds
\end{equation}
for $x > 0$ and $1/2 < \sigma < 1$. The second term on the right-hand side of \eqref{eqn:L+H+def} is negligibly small by Stirling's formula. The bounds \eqref{eqn:mellink-hline} for $\widehat{H^{-}}(s)$ ensure that the integral \eqref{eqn:phidefeq} converges absolutely, and by \eqref{eqn:Jt+bound}, the first term on the right-hand side of \eqref{eqn:L+H+def} is absolutely convergent for $t \in \R$ provided that $1/2 < \sigma < 1$.

Next, we may restrict the range of integration in \eqref{eqn:phidefeq} to $|\Im(s)| < T^{1 + \e}/U$ up to a negligibly small error term, since the bounds \eqref{eqn:mellink-hline} for $\widehat{H^{-}}(s)$ ensure that the remaining portion of the integral is negligibly small. In the range $|\Im(s)| < T^{1 + \e}/U$, we may replace $\widehat{\JJ_{2t_g}^-}(s)$ by the expansion in \eqref{eqn:stir4} and then extend the integral back to the whole line $\Re(s) = \sigma$ at the cost of a negligibly small error term, since once more \eqref{eqn:mellink-hline} ensures that $\widehat{H^{-}}(s)$ is negligibly small in the remaining portion of the integral. Thus in place of $\phi(x)$, we are led to studying linear combinations of terms of the form
\[\int_0^\infty H^-(y) \frac{1}{2\pi i} \int_{\sigma - i\infty}^{\sigma + i\infty} \widehat{\JJ_0^+}(s) \frac{s^{j_1}}{t_g^j} \left(\frac{\pi x y}{t_g} \right)^{1 - s} \, ds \, \frac{dy}{y}\]
for $1/2 < \sigma < 1$ and $j \geq 0$ fixed, where $0 \leq j_1 \leq 3j/2$. We may write $s^{j_1}$ as a linear combination of products $\prod_{k=1}^{k=\ell} (-s+k)$ for $0\le \ell \le j_1$, where the empty product is interpreted as being equal to $1$. Doing so and then integrating by parts $\ell$ times and using the fact that
\[\frac{1}{2\pi i} \int_{\sigma - i\infty}^{\sigma + i\infty} \widehat{\JJ_0^+}(s) w^{-s} \, ds = \JJ_0^+(w) = -2\pi Y_0(4\pi w),\]
we find that up to a negligibly small error, $\phi(x)$ equals a linear combination of
\[{\phi}_{j,\ell}(x) \coloneqq \frac{x}{t_g} \int_0^\infty \frac{y^{\ell}}{t_g^j} \frac{d^{\ell}H^-(y)}{dy^{\ell}}  Y_0\left(\frac{4\pi^2 xy}{t_g} \right) \, dy
\]
for $\ell\le j_1\le 3/2 j$.
\end{proof}

Now as a preliminary step, we show that, up to a negligible error term, we may restrict the outer integral in \eqref{eqn:newL'} to $x \geq 2$. This will be useful because we will be able to use of the following expression for $Y_0(x)$.

\begin{lemma}[{\cite[Lemma 2.7]{Kha22}}]
\label{lem:Y0}
For $x \geq 1$, we have that
\begin{equation}
\label{eqn:Y0long}
Y_0(x)= \sum_{\pm} \frac{e^{\pm ix}}{\sqrt{x}} W_{\pm}(x)
\end{equation}
for some smooth functions $W_{\pm}$ satisfying
\begin{equation}
\label{eqn:Wderivbounds}
W_{\pm}^{(j)}(x) \ll_j x^{-j}
\end{equation}
for $j \in \N_0$. For $0 < x < 10^6$, we have that
\begin{equation}
\label{eqn:Y0bound}
Y_0^{(j)}(x) \ll_{j,\e} x^{-j - \e}
\end{equation}
for $j \in \N_0$.
\end{lemma}

We shall make use of the following bounds for the transform $\Lscr^+$.

\begin{lemma}[{\cite[Lemma 1]{BHM07}, \cite[Lemma 3]{Jut99}}]
\label{lem:transf-bounds}
\hspace{1em}
\begin{enumerate}[leftmargin=*,label=\textup{(\arabic*)}]
\item\label{item:transf-bounds1} If $\phi(x)$ is a smooth function supported on $0 < X < x <2X$, with $\phi^{(j)}(x) \ll_j X^{-j}$ for all $j \in \N_0$, then for $t \in \R$, we have that
\[(\Lscr^+ \phi)(t) \ll_{\ell} \frac{1 + |\log X|}{1 + X} \left(\frac{1 + X}{1 + |t|}\right)^{\ell}\]
for any $\ell \in \N_0$.
\item\label{item:transf-bounds2} If $\phi(x) = e(\pm 2x) \psi(x)$ with $\psi(x)$ a smooth function supported on $1 < X < x < 2X$ satisfying $\psi^{(j)}(x) \ll_j X^{-j}$ for all $j \in \N_0$, then for $t \in \R$, we have that
\[(\Lscr^+ \phi)(t) \ll_{\e} X^{-\frac{1}{2} + \e}.\]
For $|t| > X^{1/2 + \e}$, we have that
\[(\Lscr^+ \phi)(t) \ll_{\ell,\e} (|t| + X)^{-\ell} X^{\e}\]
for any $\ell \in \N_0$.
\end{enumerate}
\end{lemma}

\begin{lemma}
For $|t| > T^{\e}$, we have that
\begin{equation}
\label{eqn:1integral}
 \widetilde{h}^{+}(t)   =  \sum_{\substack{0 \leq j \leq M \\ 0 \leq j_1 \leq 3j/2}} c_{j,j_1}  \int_{2}^\infty \JJ_t^+(x)  \int_0^\infty \frac{1}{t_g^{j+1}} H_{j_1}^-(y) Y_0\left(\frac{4\pi^2 xy}{t_g} \right) \, dy  \, dx + O(T^{-A})
\end{equation}
for any $A > 0$ and $M$ sufficiently large in terms of $A$.
\end{lemma}

\begin{proof}
First consider the contribution to \eqref{eqn:newL} of the range $x \leq T^{-100A}$ for $A > 0$. Inserting the bounds \eqref{eqn:Jt+bound} and \eqref{eqn:Y0bound}, we get that the contribution of this range is $O(T^{-A})$. 

Now consider the range $T^{-100A} < x < 2$. The portion of the integral over $y \in \R$ for which $|2\pi y - T| > U^{1 + \e}$ is $O(T^{-A})$ by the bounds \eqref{eqn:H-derivbound} for $H_{j_1}^-(y)$, \eqref{eqn:Jt+bound} for $\JJ_t^+$, and \eqref{eqn:Y0bound} for $Y_0$. For the remaining range $|2\pi y - T| < U^{1 + \e}$, we have that $H_{j_1}^-(y) \ll T^{j_1+1}$ by \eqref{eqn:H-derivbound} and that
\[\frac{d^j}{dx^j} Y_0\left( \frac{4 \pi^2 xy}{t_g}\right) \ll_j T^{\e} x^{-j}\]
by \eqref{eqn:Y0bound}. So by dividing the interval $T^{-100A} < x < 2 $ into dyadic intervals and using \hyperref[item:transf-bounds1]{Lemma \ref*{lem:transf-bounds} \ref*{item:transf-bounds1}} with $\ell$ sufficiently large, we get in the current range that when $|t| > T^{\e}$, the contribution to \eqref{eqn:newL} is $O(T^{-A})$.
\end{proof}

Next we show that up to a negligible error term, the inner integral in \eqref{eqn:1integral} can be restricted to $|2\pi y - T| < U^{1 + \e}$. To this end, let $W_0(y)$ denote a smooth function that approximates the characteristic function of the interval $|2\pi y - T| < U^{1 + \e}$, such that
\begin{equation}
\label{eqn:W0derivbounds} W_0^{(j)}(y)\ll_j (U^{-1 - \e})^j
\end{equation}
for all $j\ge 1$.
\begin{lemma}
\label{lem:restrict-y}
For $|t| > T^{\e}$, we have that
\begin{equation}
\label{eqn:newL2}
\widetilde{h}^{+}(t) = \sum_{\substack{0 \leq j \leq M \\ 0 \leq j_1 \leq 3j/2}} c_{j,j_1}  \int_{2}^\infty \JJ_t^+(x) \int\limits_0^\infty  \frac{W_0(y)  H_{j_1}^-(y)}{t_g^{j+1}} Y_0\left(\frac{4\pi^2 xy}{t_g} \right) \, dy  \, dx + O(T^{-A})
\end{equation}
for any $A > 0$ and $M$ large enough in terms of $A$.
\end{lemma}

\begin{proof}
We write the integral in \eqref{eqn:1integral} as
\[\int_2^\infty \JJ_t^+(x) \left(\int_0^{t_g/ 4\pi^2 x} + \int_{t_g/ 4\pi^2 x}^\infty \right) \frac{H_{j_1}^-(y)}{t_g^{j+1}} Y_0\left(\frac{4 \pi^2 xy}{t_g}\right) \, dy \, dx.\]
The first of these integrals is bounded, using \hyperref[lem:Y0]{Lemma \ref*{lem:Y0}} to estimate $Y_0( \frac{4 \pi^2 xy}{t_g})$, by
\begin{equation}
\label{eqn:firstintegral}
\int_2^\infty |\JJ_t^+(x) | \int_0^{t_g/ 4\pi^2 x} \frac{|H_{j_1}^-(y)|}{t_g^{j+1}} \left(\frac{4\pi^2 xy}{t_g}\right)^{-\e} \, dy \, dx.
\end{equation}
By inserting the bounds \eqref{eqn:H-derivbound} for $H_{j_1}^-(y)$, integrating over $y$, and then inserting the bound from \eqref{eqn:Jt+bound} for $\JJ_t^+$, we see that \eqref{eqn:firstintegral} converges. In this integral, we have $|2\pi y - T|\geq U^{1 + \e}$, since $y\le t_g/(4\pi^2 x) < T/(8\pi^2)$. Thus we can insert the bound \eqref{eqn:H-derivbound} for $H_{j_1}^-(y)$ to see that the contribution of \eqref{eqn:firstintegral} is $O(T^{-A})$ for any $A > 0$. The second integral is, using Lemma \ref{lem:Y0},
\[\sum_{\pm} \int_2^\infty \JJ_t^+(x) \int_{t_g/ 4\pi^2 x}^\infty \frac{H_{j_1}^-(y)}{t_g^{j+1}} \left( \frac{4 \pi^2 xy}{t_g^{j+1}}\right)^{-\frac{1}{2}} e\left(\pm \frac{2 \pi xy}{t_g}\right) W_{\pm}\left( \frac{4 \pi^2 xy}{t_g}\right) \, dy \, dx.\]
Inserting $W_0(y)+W_1(y)$ into the integrand, where $W_1(y)=1-W_0(y)$, we need to show that
\[\sum_{\pm} \int_2^\infty \JJ_t^+(x) \int_{t_g/ 4\pi^2 x}^\infty \frac{W_1(y)  H_{j_1}^-(y)}{t_g^{j+1}} \left( \frac{4 \pi^2 xy}{t_g}\right)^{-\frac{1}{2}} e\left(\pm \frac{2 \pi xy}{t_g}\right) W_{\pm}\left( \frac{4 \pi^2 xy}{t_g}\right) \, dy \, dx \ll T^{-A}.\]
By integrating by parts once, this double integral is equal to
\begin{align*}
&\pm \int_2^\infty \JJ_t^+(x) \left[\frac{t_g}{i 4 \pi^2 x} e\left(\pm \frac{2 \pi xy}{t_g}\right)  \frac{W_1(y)  H_{j_1}^-(y)}{ t_g^{j+1}} \left( \frac{4 \pi^2 xy}{t_g}\right)^{-\frac{1}{2}} W_{\pm}\left( \frac{4 \pi^2 xy}{t_g}\right)\right]_{y = \frac{t_g}{4\pi^2 x}}^{y = \infty} \, dx	\\
&\mp \int_2^\infty \JJ_t^+(x) \int_{t_g/ 4\pi^2 x}^\infty \frac{t_g}{i 4 \pi^2 x} e\left(\pm \frac{2 \pi xy}{t_g}\right) \\
&\hskip1in \times \frac{d}{dy}\left( \frac{W_1(y) H_{j_1}^-(y)}{ t_g^{j+1}} \left( \frac{4 \pi^2 xy}{t_g}\right)^{-\frac{1}{2}} W_{\pm}\left( \frac{4 \pi^2 xy}{t_g}\right)\right) \, dy \, dx.
\end{align*}
By \eqref{eqn:Jt+bound} and the rapid decay of $H_{j_1}^{-}$, these integrals converge. Moreover these integrals are bounded by $O(T^{-A})$ for any $A > 0$ by \eqref{eqn:H-derivbound} since the support of $W_1(y)$ restricts the integrals to the range $|2\pi y - T| \geq U^{1 + \e}$.
\end{proof}

Now we prove that $(\Lscr^+ \widetilde{\phi}_{j,j_1})(t)$ is in the form needed in order to apply \hyperref[lem:transf-bounds]{Lemma \ref*{lem:transf-bounds}}.

\begin{lemma}
We have that 
\begin{equation}
\label{eqn:Xellexpansion}
(\Lscr^+  \widetilde{\phi}_{j,j_1})(t) = U \sum_{\pm} \sum_{1 \leq \ell \leq T^{\e}} X_\ell^\frac{1}{2} (\Lscr^+ \phi_{j,j_1,\ell,\pm})(t) + O(T^{-A})
\end{equation}
for some functions $\phi_{j,j_1,\ell,\pm}$ of the form $\phi_{j,j_1,\ell,\pm}(x) = e(\pm 2x) \psi_{j,j_1,\ell,\pm}(x)$ with $\psi_{j,j_1,\ell,\pm}(x)$ supported on $1 \leq X_\ell < x < 2X_\ell <(T/U)^{1 + \e}$ and satisfying $\psi_{j,j_1,\ell,\pm}^{(j)}(x) \ll_{j,k} X_\ell^{-k}$ for all $k \geq 0$.
\end{lemma}

\begin{proof}
By \eqref{eqn:Y0long}, we have that the inner integral in \eqref{eqn:newL2} equals
\begin{equation}
\label{eq:innerint} \sum_\pm \int\limits_0^\infty  \frac{1}{2\pi \sqrt{xy}}\frac{W_0(y)  H_{j_1}^-(y)}{t_g^{j+\frac12}}   e\left(\frac{\pm 2\pi xy}{t_g} \right)  W_\pm \left(\frac{4\pi^2 xy}{t_g} \right)\, dy. 
\end{equation}
For $x \geq T^{1 + \e} / U$, we repeatedly integrate by parts with respect to $y$, integrating $e(\pm Uxy/t_g)$ and differentiating the rest. Via the bounds \eqref{eqn:Wderivbounds} for the derivatives of $W_\pm$, the bounds \eqref{eqn:W0derivbounds} for the derivatives of $W_0$, and the bounds \eqref{eqn:H-derivbound} for the derivatives of $H_{j_1}^-$, we see that the portion of the outer integral in \eqref{eqn:newL2} for which $x \geq T^{1 + \e} / U$ is negligibly small.

We make the substitution $y \mapsto \frac{Uy + T}{2\pi}$ in the inner integral \eqref{eq:innerint} and recall that $T = 2t_g - 2U + 2$, to see that up to a negligible error, $(\Lscr^+  \widetilde{\phi}_{j,j_1})(t)$ is proportional to
\begin{multline}
\label{eqn:delicate}
\sum_\pm U \int_{2}^{\frac{T^{1 + \e}}{ U}} \JJ_t^+(x) \Bigg( \frac{x^\frac12 e(\pm 2x)  }{ (t_g T)^{\frac12} t_g^{j} } e\left(\mp \frac{2(U - 1)x}{t_g}\right)
 \int\limits_0^\infty \left(1 + \frac{Uy}{T}\right)^{-\frac12} W_0\left( \frac{Uy + T}{2\pi} \right) \\
\left(\frac{Uy + T}{2\pi}\right)^{j_1} (H^{-})^{(j_1)}\left(\frac{Uy + T}{2\pi}\right) W_\pm \left(\frac{2 \pi x (Uy + T)}{t_g}\right) e\left(\pm \frac{Uxy}{t_g}\right) \, dy \Bigg) \frac{dx}{x}.
\end{multline}
Note that in \eqref{eqn:delicate}, the $y$-integral is restricted to $|y|<T^\varepsilon$ by the support of $W_0$, and that
\[\frac{1}{(t_g T)^{\frac12} t_g^{j}} \left(\frac{Uy + T}{2\pi}\right)^{j_1} (H^{-})^{(j_1)}\left(\frac{Uy + T}{2\pi}\right) \ll 1\]
using \eqref{eqn:H-derivbound}, since $j_1 \le 3/2 j$ and $U>T^{1/3+\delta}$. We obtain the desired decomposition upon applying a smooth partition of unity to split the interval $[2, T^{1 + \e} / U]$ into dyadic intervals.
\end{proof}

Finally, we complete the proof of the bounds \eqref{eqn:tildeh+boundtransition} for $\widetilde{h}^{+}(t)$.

\begin{proof}[Proof of {\eqref{eqn:tildeh+boundtransition}}]
Due to the bounds attained for $\widetilde{h}^{+}(t)$ in \hyperref[lem:+crude]{Lemma \ref*{lem:+crude}}, it remains only to show that
\[\widetilde{h}^{+}(t) \ll_{\e} \begin{dcases*}
U^{1 + \e} & if $T^{\e} \leq |t| \leq \left(\frac{T}{U}\right)^{\frac{1}{2} + \e}$,	\\
(|t|T)^{-100} & if $\left(\frac{T}{U}\right)^{\frac{1}{2} + \e} \leq |t| \leq \frac{T^{1 + \e}}{U}$.
\end{dcases*}\]
Via \hyperref[lem:tildeh+truncate]{Lemma \ref*{lem:tildeh+truncate}}, it suffices to prove these bounds for $(\Lscr^+  \widetilde{\phi}_{j,j_1})(t)$ in place of $\widetilde{h}^+(t)$. The desired bounds then follow by combining the expansion \eqref{eqn:Xellexpansion} for $(\Lscr^+  \widetilde{\phi}_{j,j_1})(t)$ together with the bounds for $(\Lscr^+ \phi_{j,j_1,\ell,\pm})(t)$ given in \hyperref[item:transf-bounds2]{Lemma \ref*{lem:transf-bounds} \ref*{item:transf-bounds2}}.
\end{proof}

\begin{remark}
The calculation of $\widetilde{h}^{+}(t)$ is delicate because in order to apply \hyperref[item:transf-bounds2]{Lemma \ref*{lem:transf-bounds} \ref*{item:transf-bounds2}}, one needs $\phi$ to be of the form $\phi(x) = e(ax) \psi(x)$ with $a \in \{2,-2\}$, which is what we arrived at in \eqref{eqn:delicate}; no other constants $a$ will suffice.
\end{remark}

\subsubsection{Strong Bounds for $\widetilde{h}^{\hol}(k)$}

The bounds \eqref{eqn:tildehholboundtransition} for $\widetilde{h}^{\hol}(k)$ are deduced in exactly the same way.

\begin{proof}[Proof of {\eqref{eqn:tildehholboundtransition}}]
Due to the bounds attained for $\widetilde{h}^{\hol}(t)$ in \hyperref[lem:holcrude]{Lemma \ref*{lem:holcrude}}, it remains only to show that
\[\widetilde{h}^{\hol}(k) \ll_{\e} \begin{dcases*}
U^{1 + \e} & if $T^{\e} \leq k \leq \left(\frac{T}{U}\right)^{\frac{1}{2} + \e}$,	\\
(kT)^{-100} & if $\left(\frac{T}{U}\right)^{\frac{1}{2} + \e} \leq k \leq \frac{T^{1 + \e}}{U}$.
\end{dcases*}\]
This follows by precisely the same method as for $\widetilde{h}^{+}(t)$; an analogue of \hyperref[lem:tildeh+truncate]{Lemma \ref*{lem:tildeh+truncate}} holds for $\widetilde{h}^{\hol}(k)$ in place of $\widetilde{h}^{+}(t)$ by using in place of \eqref{eqn:Jt+bound} the bounds \cite[Proposition 8]{HM06}
\[\JJ_k^\hol(x) \ll \begin{dcases*}
\frac{(4\pi x)^{k - 1}}{2^{k - 1} \Gamma\left(k - \frac{1}{2}\right)} & if $0 < 4\pi x \leq 1$, \\
\frac{k}{\sqrt{x}} & if $4\pi x > 1$,
\end{dcases*}\]
while the analogues of \hyperref[item:transf-bounds1]{Lemma \ref*{lem:transf-bounds} \ref*{item:transf-bounds1}} and \ref{item:transf-bounds2} hold with $\Lscr^+$ replaced by $\Lscr^{\hol}$ by \cite[Lemma 1]{BHM07} and \cite[Remark 1]{Jut99}.
\end{proof}

\section{Bounds for Mixed Moments of \texorpdfstring{$L$}{L}-Functions in the Short Transition Range}

\subsection{Bounds via the Spectral Large Sieve}

When bounding the mixed moment of $L$-functions \eqref{eqn:transition} in the short transition range, we can no longer perform a dyadic subdivision, since the analytic conductors of $L(1/2,\ad g \otimes f)$ and $L(1/2 + it,\ad g)$ exhibit conductor-dropping in this range. Instead, we must divide into shorter intervals. After an application of H\"{o}lder's inequality, this leads us to studying the second moment of $L(1/2,\ad g \otimes f)$ in short intervals. We can bound this via the spectral large sieve.

\begin{proposition}
\label{prop:largesievebounds3}
Let $g$ be a Hecke--Maa\ss{} cusp form on $\Gamma \backslash \Hb$ with spectral parameter $t_g$. For $t_g \leq T \leq 3t_g$ and $U = \left|t_g - \frac{T}{2}\right| + 1$, we have the bounds
\begin{equation}
\label{eqn:largesievebounds3}
\begin{drcases*}
\sum_{\substack{f \in \BB_0 \\ T - U \leq t_f \leq T + U}} \frac{L\left(\frac{1}{2},\ad g \otimes f\right)^2}{L(1,\ad f)} & \\
\frac{1}{2\pi} \int\limits_{T - U \leq |t| \leq T + U} \left|\frac{L\left(\frac{1}{2} + it,\ad g\right)^2}{\zeta(1 + 2it)}\right|^2 \, dt &
\end{drcases*} \ll_{\e} t_g^{2 + \e} U.
\end{equation}
\end{proposition}

\begin{proof}
Just as in the proof of \hyperref[prop:largesievebounds1]{Proposition \ref*{prop:largesievebounds1}}, this follows via the approximate functional equation and the spectral large sieve, noting that the conductor of $L(1/2,\ad g \otimes f)$ and of $|L(1/2 + it,\ad g)|^2$ in these ranges is $O(t_g^4 U^2)$.
\end{proof}

\subsection{Bounds via the Kuznetsov Formula}

When applying H\"{o}lder's inequality, we also are led to the study of the first moment of $L(1/2,\ad g \otimes f)$ in short intervals close to $2t_g$. While we expect that this can be achieved via $\GL_3 \times \GL_2 \leftrightsquigarrow \GL_4 \times \GL_1$ spectral reciprocity, we instead approach this problem in a more traditional fashion via the Kuznetsov formula. As in \hyperref[sect:GL4xGL2transforms]{Section \ref*{sect:GL4xGL2transforms}}, the statement and the proof of the following proposition uses the $\e$-convention: $\e$ denotes an arbitrarily small constant whose value may change from occurrence to occurrence.

\begin{proposition}
\label{prop:firstmomentconductordrop}
Let $g$ be a Hecke--Maa\ss{} cusp form on $\Gamma \backslash \Hb$ with spectral parameter $t_g$. For $2t_g - t_g^{1/3 + \e} \leq T \leq 2t_g$ and $U = t_g - \frac{T}{2} + 1$, we have that
\begin{equation}
\label{eqn:firstmomentconductordrop}
\begin{drcases*}
\sum_{\substack{f \in \BB_0 \\ T - U \leq t_f \leq T + U}} \frac{L\left(\frac{1}{2},\ad g \otimes f\right)}{L(1,\ad f)} & \\
\frac{1}{2\pi} \int\limits_{T - U \leq |t| \leq T + U} \left|\frac{L\left(\frac{1}{2} + it,\ad g\right)}{\zeta(1 + 2it)}\right|^2 \, dt & \end{drcases*} \ll_{\e} \begin{dcases*}
t_g^{\frac{3}{2} + \e} U^{\frac{1}{2}} & if $1 \leq U \leq t_g^{\frac{1}{5}}$,	\\
t_g^{\frac{7}{4} + \e} U^{-\frac{5}{4}} & if $t_g^{\frac{1}{5}} \leq U \leq t_g^{\frac{1}{3} + \e}$.
\end{dcases*}
\end{equation}
\end{proposition}

We postpone the proof of \hyperref[prop:firstmomentconductordrop]{Proposition \ref*{prop:firstmomentconductordrop}} to \hyperref[sect:firstmomentconductordrop]{Section \ref*{sect:firstmomentconductordrop}}.

\subsection{Bounds via Spectral Reciprocity}

The application of H\"{o}lder's inequality that we use in the proof of \hyperref[item:transition]{Proposition \ref*{prop:fourranges} \ref*{item:transition}} leads us to a short interval third moment of $L(1/2,f)$ for $f \in \BB_0$. For this, we have the following well-known bound, which is essentially a consequence of $\GL_3 \times \GL_2 \leftrightsquigarrow \GL_4 \times \GL_1$ spectral reciprocity in the guise of Motohashi's formula.

\begin{proposition}[{Ivi\'{c} \cite[Theorem]{Ivi01}}]
\label{prop:thirdmoment}
For $1 \leq U \leq T$, we have that
\begin{equation}
\label{eqn:thirdmoment}
\begin{drcases*}
\sum_{\substack{f \in \BB_0 \\ T - U \leq t_f \leq T + U}} \frac{L\left(\frac{1}{2},f\right)^3}{L(1,\ad f)} & \\
\frac{1}{2\pi} \int\limits_{T - U \leq |t| \leq T + U} \left|\frac{\zeta\left(\frac{1}{2} + it\right)^3}{\zeta(1 + 2it)}\right|^2 \, dt &
\end{drcases*} \ll_{\e} T^{1 + \e} U.
\end{equation}
\end{proposition}

It is important to note that the terms on the left-hand side of \eqref{eqn:thirdmoment} are nonnegative, as $L(1/2,f) \geq 0$ for $f \in \BB_0$ by \cite[Corollary 0.1]{KaSa93}.

With these collections of bounds in hand, we are able to show the following bounds for mixed moments of $L$-functions in the short transition range.

\begin{proposition}
\label{prop:momentsboundstransition}
Let $g$ be a Hecke--Maa\ss{} cusp form on $\Gamma \backslash \Hb$ with spectral parameter $t_g$. Fix $\alpha > 0$. Then for $2t_g - t_g^{1 - \alpha} \leq T \leq 2t_g$ and $U = t_g - \frac{T}{2} + 1$, we have that
\begin{multline}
\label{eqn:momentsboundstransition}
\begin{drcases*}
\sum_{\substack{f \in \BB_0 \\ T - U \leq t_f \leq T + U}} \frac{L\left(\frac{1}{2},f\right) L\left(\frac{1}{2},\ad g \otimes f\right)}{L(1,\ad f)} &	\\
\frac{1}{2\pi} \int\limits_{T - U \leq |t| \leq T + U} \left|\frac{\zeta\left(\frac{1}{2} + it\right) L\left(\frac{1}{2} + it,\ad g\right)}{\zeta(1 + 2it)}\right|^2 \, dt &
\end{drcases*}
\ll_{\e} \begin{dcases*}
t_g^{\frac{3}{2} + \e} U^{\frac{5}{6}} & if $1 \leq U \leq t_g^{\frac{1}{5}}$,	\\
t_g^{\frac{19}{12} + \e} U^{\frac{1}{4}} & if $t_g^{\frac{1}{5}} \leq U \leq t_g^{\frac{1}{3} + \e}$,	\\
t_g^{\frac{47}{38} + \e} U & if $t_g^{\frac{1}{3} + \e} \leq U \leq t_g^{\frac{13}{19}}$,	\\
t_g^{\frac{7}{4} + \e} U^{\frac{1}{4}} & if $t_g^{\frac{13}{19}} \leq U \leq t_g^{1 - \alpha}$.
\end{dcases*}
\end{multline}
\end{proposition}

\begin{proof}
Our first approach, valid for $1 \leq U \leq t_g^{1/3 + \e}$, is to write
\[\frac{L\left(\frac{1}{2},f\right) L\left(\frac{1}{2},\ad g \otimes f\right)}{L(1,\ad f)} = \left(\frac{L\left(\frac{1}{2},f\right)^3}{L(1,\ad f)}\right)^{\frac{1}{3}} \left(\frac{L\left(\frac{1}{2},\ad g \otimes f\right)^2}{L(1,\ad f)}\right)^{\frac{1}{3}} \left(\frac{L\left(\frac{1}{2},\ad g \otimes f\right)}{L(1,\ad f)}\right)^{\frac{1}{3}}\]
(which crucially uses the nonnegativity of these $L$-functions), apply H\"{o}lder's inequality with exponents $(1/3,1/3,1/3)$, and use the bounds from \hyperref[prop:largesievebounds3]{Propositions \ref*{prop:largesievebounds3}}, \ref{prop:firstmomentconductordrop}, and \ref{prop:thirdmoment}; that is, as well as bounds arising from the spectral large sieve, we use third moment bounds and bounds from the Kuznetsov formula. This shows that the left-hand side of \eqref{eqn:momentsboundstransition} is
\[\ll_{\e} \begin{dcases*}
t_g^{\frac{3}{2} + \e} U^{\frac{5}{6}} & if $1 \leq U \leq t_g^{\frac{1}{5}}$,	\\
t_g^{\frac{19}{12} + \e} U^{\frac{1}{4}} & if $t_g^{\frac{1}{5}} \leq U \leq t_g^{\frac{1}{3} + \e}$.
\end{dcases*}\]

Our second approach, valid for $t_g^{1/3 + \e} \leq U \leq t_g^{1 - \alpha}$, is to use \hyperref[thm:4x2reciprocity]{Theorem \ref*{thm:4x2reciprocity}}, namely $\GL_4 \times \GL_2 \leftrightsquigarrow \GL_4 \times \GL_2$ spectral reciprocity. We take the triple of test functions $(h^{+},h^{-},h^{\hol})$ given by \eqref{eqn:fourthtriple}. With this choice of test functions, the left-hand side of \eqref{eqn:4x2identity} provides an upper bound for the left-hand side of \eqref{eqn:momentsboundstransition} by positivity, as $h^{+}(t) = 0$ and $h^{\hol}(k) = 0$, while $h^{-}(t) \geq 0$ for all $t \in \R$ and $h^{-}(t) \asymp_M 1$ for $t \in [-T - U, -T + U] \cup [T - U,T + U]$. The first term on the right-hand side of \eqref{eqn:4x2identity} is $O_{\e}(t_g^{1 + \e} U)$ as $L(1,\ad g) \ll_{\e} t_g^{\e}$. The second term is equal to zero since $h^{\hol}(k) = 0$. Finally, for the third, fourth, and fifth terms, we divide the terms into dyadic ranges and use the bounds \eqref{eqn:tildeh+boundtransition}, \eqref{eqn:tildeh-boundtransition}, and \eqref{eqn:tildehholboundtransition} for the transforms $\widetilde{h}^{+}(t)$, $\widetilde{h}^{-}(t)$, and $\widetilde{h}^{\hol}(t)$. Due to the rapid decay of $\widetilde{h}^{+}(t)$ for $|t| \geq (T/U)^{1/2 + \e}$, $\widetilde{h}^{-}(t)$ for $|t| \geq T^{\e}$, and $\widetilde{h}^{\hol}(k)$ for $k \geq (T/U)^{1/2 + \e}$, we are left with bounding the quantities
\begin{gather*}
\sum_{\substack{f \in \BB_0 \\ V \leq t_f \leq 2V}} \frac{L\left(\frac{1}{2},f\right) L\left(\frac{1}{2},\ad g \otimes f\right)}{L(1,\ad f)}, \qquad \frac{1}{2\pi} \int\limits_{V \leq |t| \leq 2V} \left|\frac{\zeta\left(\frac{1}{2} + it\right) L\left(\frac{1}{2} + it,\ad g\right)}{\zeta(1 + 2it)}\right|^2 \, dt,	\\
\sum_{\substack{f \in \BB_{\hol} \\ V \leq k_f \leq 2V}} \frac{L\left(\frac{1}{2},f\right) L\left(\frac{1}{2},\ad g \otimes f\right)}{L(1,\ad f)}
\end{gather*}
for $V \leq (T/U)^{1/2 + \e}$, for which we may apply the bounds from \hyperref[prop:momentsboundsinitial]{Proposition \ref*{prop:momentsboundsinitial}}. This shows that the left-hand side of \eqref{eqn:momentsboundstransition} is
\[\ll_{\e} \begin{dcases*}
t_g^{\frac{47}{38} + \e} U & if $t_g^{\frac{1}{3} + \e} \leq U \leq t_g^{\frac{13}{19}}$,	\\
t_g^{\frac{7}{4} + \e} U^{\frac{1}{4}} & if $t_g^{\frac{13}{19}} \leq U \leq t_g^{1 - \alpha}$.
\end{dcases*}\qedhere\]
\end{proof}

\subsection{Proof of \texorpdfstring{\hyperref[item:transition]{Proposition \ref*{prop:fourranges} \ref*{item:transition}}}{Proposition \ref{prop:fourranges} \ref{item:transition}}}

We now prove \hyperref[item:transition]{Proposition \ref*{prop:fourranges} \ref*{item:transition}}, namely the bound \eqref{eqn:transition} for the short transition range, via \hyperref[prop:momentsboundstransition]{Proposition \ref*{prop:momentsboundstransition}}.

\begin{proof}[Proof of {\hyperref[item:transition]{Proposition \ref*{prop:fourranges} (3)}}]
By the lower bound $L(1,\ad g) \gg_{\e} t_g^{-\e}$ and the asymptotic formula \eqref{eqn:Htasymp} for $H(t)$, it suffices to show that
\[\begin{drcases*}
\sum_{\substack{f \in \BB_0 \\ 2t_g - t_g^{1 - \alpha} \leq t_f \leq 2t_g}} \frac{1}{1 + (2t_g - t_f)^{1/2}} \frac{L\left(\frac{1}{2},f\right) L\left(\frac{1}{2},\ad g \otimes f\right)}{L(1,\ad f)} &	\\
\frac{1}{2\pi} \int\limits_{2t_g - t_g^{1 - \alpha} \leq |t| \leq 2t_g} \frac{1}{1 + (2t_g - |t|)^{1/2}} \left|\frac{\zeta\left(\frac{1}{2} + it\right) L\left(\frac{1}{2} + it,\ad g\right)}{\zeta(1 + 2it)}\right|^2 \, dt & \end{drcases*}
\ll_{\e} t_g^{\frac{30}{19} + \e}.\]
We dyadically decompose both the sum over $f$ and the integral over $t$, so that we are left with proving the bounds
\begin{equation}
\label{eqn:shorttransitiontobeproved}
\begin{drcases*}
\sum_{\substack{f \in \BB_0 \\ T - U \leq t_f \leq T + U}} \frac{L\left(\frac{1}{2},f\right) L\left(\frac{1}{2},\ad g \otimes f\right)}{L(1,\ad f)} &	\\
\frac{1}{2\pi} \int\limits_{T - U \leq |t| \leq T + U} \left|\frac{\zeta\left(\frac{1}{2} + it\right) L\left(\frac{1}{2} + it,\ad g\right)}{\zeta(1 + 2it)}\right|^2 \, dt & \end{drcases*}
\ll_{\e} t_g^{\frac{30}{19} + \e} U^{\frac{1}{2}}
\end{equation}
for $2t_g - t_g^{1 - \alpha} \leq T \leq 2t_g$ and $U = t_g - \frac{T}{2} + 1$. These desired bounds are a consequence of \hyperref[prop:momentsboundstransition]{Proposition \ref*{prop:momentsboundstransition}}, which gives these bounds when $U \asymp t_g^{13/19}$ and gives stronger bounds otherwise.
\end{proof}

\begin{remark}
\label{rem:shorttransitionoptimal}
Obtaining the essentially optimal bound $\|g\|_4 \ll_{\e} t_g^{\e}$ for the $L^4$-norm would require us to improve the bounds for the left-hand side of \eqref{eqn:shorttransitiontobeproved} to $O_{\e}(t_g^{3/2 + \e} U^{1/2})$ for $2t_g - t_g^{1 - \alpha} \leq T \leq 2t_g$ and $U = t_g - \frac{T}{2} + 1$.
\end{remark}

\section{Bounds for the First Moment of \texorpdfstring{$L(\frac{1}{2},\ad g \otimes f)$}{L(\80\275,ad g\9042\227f)} via the Kuznetsov Formula}
\label{sect:firstmomentconductordrop}

The method of proof of \hyperref[prop:firstmomentconductordrop]{Proposition \ref*{prop:firstmomentconductordrop}} involves replacing the $L$-functions $L(1/2,\ad g \otimes f)$ and $|L(1/2 + it,\ad g \otimes f)|^2$ with Dirichlet polynomials via the approximate functional equation, interchanging the order of summation, and applying the Kuznetsov formula. In this way, the problem is reduced to bounding a sum of Kloosterman sums weighted by Hecke eigenvalues.

\subsection{Approximate Functional Equations}

For $\Re(s) > 1$, we have that
\[L(s,\ad g \otimes f) = \sum_{m,n = 1}^{\infty} \frac{A_F(m,n)\lambda_f(n)}{m^{2s} n^s},\]
where $F = \ad g$. This has functional equation
\[L(s, \ad g \otimes f) G(s,t_f,\epsilon_f) = \epsilon_f L(1 - s, \ad g \otimes f) G(1 - s,t_f, \epsilon_f),\]
where $\epsilon_f \in \{1,-1\}$ is the parity of $f$, and
\[G(s,t, \epsilon) \coloneqq \prod_\pm \Gamma_{\R}\left(s + \frac{1}{2} - \frac{\epsilon}{2} \pm it + 2it_g\right) \Gamma_{\R}\left(s + \frac{1}{2} - \frac{\epsilon}{2} \pm it\right) \Gamma_{\R}\left(s + \frac{1}{2} - \frac{\epsilon}{2} \pm it - 2it_g\right).\]
From this, we get the following approximate functional equation for the central $L$-values, which are known to be nonnegative. We also give a related approximate functional equation for $|L(1/2 + it,F)|^2$.

\begin{lemma}[{\cite[Theorem 5.3]{IK04}}]
\label{lem:afeFxf}
For $t \in \R$, $\epsilon \in \{1,-1\}$, $x > 0$, $\sigma>0$ and $X \geq 1$ a fixed parameter of our choice, let
\begin{equation}
\label{eqn:vdef}
V_\pm (x,t,\epsilon) \coloneqq \frac{1}{2\pi i} \int_{\sigma - i\infty}^{\sigma + i\infty} e^{s^2} \left(\frac{X^{\pm 1}}{x}\right)^s \frac{G\left(\frac{1}{2}+s,t,\epsilon\right)}{G\left(\frac{1}{2},t,\epsilon\right)} \, \frac{ds}{s}.
\end{equation}
We have that
\begin{equation}
\label{eqn:afe}
L\left(\frac{1}{2}, \ad g \otimes f\right) = S_+(t_f, \epsilon_f) + \epsilon_f S_-(t_f, \epsilon_f),
\end{equation}
where
\[S_\pm(t_f,\epsilon_f) \coloneqq \sum_{m,n = 1}^{\infty} \frac{A_F(m,n)\lambda_f(n)}{m \sqrt{n}} V_\pm(m^2 n,t_f, \epsilon_f).\]
Similarly
\[\left|L\left(\frac{1}{2} + it,\ad g\right)\right|^2 = E_+(t)+ E_-(t),\]
where
\[E_\pm(t) \coloneqq \sum_{m,n = 1}^{\infty} \frac{A_F(m,n)\lambda(n,t)}{m \sqrt{n}} V_{\pm} (m^2 n,t,1).\]
\end{lemma}

Fix $\e > 0$, and let $t_g^{\e} \leq U \leq t_g^{1/3 + \e}$. We will be interested in the cases
\begin{equation}
\label{eqn:interest}
|t \pm 2t_g| \asymp U, \qquad X = U^{1-\e}.
\end{equation}
By Stirling's estimates and a standard contour shifting argument, we have that
\begin{equation}
\label{eqn:vdecay}
V_\pm (x,t,\epsilon) \ll_{\sigma} \left(\frac{X^{\pm 1}Ut_g^2}{x}\right)^\sigma.
\end{equation}
Thus the sums in the approximate functional equations \eqref{eqn:afe} have different lengths: the sums $S_+(t_f,\epsilon_f)$ and $E_+(t)$ are of length about $U^{2-\e} t_g^{2}$, while the sums $S_-(t_f,\epsilon_f)$ and $E_-(t)$ are of length about $U^\e t_g^{2}$. 

We will also need a version of the approximate functional equation in which the weight functions $V_\pm(x,t, \epsilon)$ do not depend on $\epsilon$.

\begin{lemma}
\label{lem:afeFxf2}
Suppose $|t_f|, |t_f-2t_g|, |t_f+2t_g|>t_g^\e$. We have that
\[L\left(\frac{1}{2}, \ad g \otimes f\right) = S_+(t_f, 1) + \epsilon_f S_-(t_f, 1) + \delta_{\epsilon_f,-1} \left(s_+(t_f, 1) + \epsilon_f s_-(t_f, 1) + O(t_g^{-50})\right), \]
where
\[s_\pm(t_f,1) \coloneqq \sum_{m,n = 1}^{\infty} \frac{A_F(m,n)\lambda_f(n)}{m \sqrt{n}} v_\pm(m^2 n, t_f, 1)\]
with
\[v_\pm(n^2 m, t_f, 1) \coloneqq \frac{1}{2\pi i} \int_{\sigma - i\infty}^{\sigma + i\infty} e^{s^2} \left(\frac{X^{\pm 1} }{m^2 n}\right)^s \frac{G\left(\frac{1}{2} + s,t,1\right)}{G\left(\frac{1}{2},t,1\right)} P_s\left(\frac{1}{t^2}, \frac{1}{(t - 2t_g)^2}, \frac{1}{(t + 2t_g)^2}\right) \, \frac{ds}{s}\]
for $\sigma>0$ and some polynomial $P_s(x_1,x_2,x_3)$ whose coefficients depend on $s$ and are bounded by $t_g^{\e}$, satisfying $P_s(0,0,0)=0$.
\end{lemma}

\begin{remark}
\label{rem:newafe}
In the cases of interest \eqref{eqn:interest} with $t = t_f$, the sums $s_\pm$ are of the same length as $S_\pm$ but their weight functions are smaller by a factor of $U^{-2 + \e}$; that is, $v_\pm(n^2m, t_f, 1) \ll U^{-2 + \e}$.
\end{remark}

\begin{proof}
For $|s| < t_g^\e$, we have that
\[\frac{G\left(\frac{1}{2} + s,t,-1\right)}{G\left(\frac{1}{2},t,-1\right)} = \frac{G\left(\frac{1}{2} + s,t,1\right)}{G\left(\frac{1}{2},t,1\right)} \left(1 + P_s\left(\frac{1}{t^2}, \frac{1}{(t - 2t_g)^2}, \frac{1}{(t + 2t_g)^2}\right)\right) + O(t_g^{-100}),\]
for $P_s(x_1,x_2,x_3)$ as described above. This is obtained by applying Stirling's formula to each gamma factor and using that $G(1/2 + s,t,\epsilon)$ is an even function of $t$, $t - 2t_g$, and $t + 2t_g$. This, together with \hyperref[lem:afeFxf]{Lemma \ref*{lem:afeFxf}}, completes the proof.
\end{proof}

\subsection{Bounds for the First Moment of Dirichlet Polynomials}

As shorthand, we write
\[S_\pm(t) \coloneqq S_\pm(t,1), \qquad V_\pm(x,t) \coloneqq V_\pm(x,t,1).\]
The chief input in the proof of \hyperref[prop:firstmomentconductordrop]{Proposition \ref*{prop:firstmomentconductordrop}} is the following.

\begin{proposition}
\label{prop:shorttransition}
Let
\begin{equation}
\label{eqn:UXhtdefeq}
t_g^{\e} \leq U \leq t_g^{\frac{1}{3} + \e}, \quad X=U^{1-\e}, \quad h(t) = \exp\left(-\left(\frac{t - 2t_g - U}{U^{1 - \e}}\right)^2\right) + \exp\left(-\left(\frac{t + 2t_g + U}{U^{1-\e}}\right)^2\right).
\end{equation}
\begin{enumerate}[leftmargin=*,label=\textup{(\arabic*)}]
\item\label{item:shorttransition1}
We have that
\begin{equation}
\label{eqn:thml1}
\sum_{f \in \BB_0} \frac{S_+(t_f)}{L(1,\ad f)} h(t_f) + \frac{1}{2\pi} \int_{-\infty}^{\infty} \frac{E_+(t)}{|\zeta(1 + 2it)|^2} h(t) \, dt \ll_{\e} t_g^{1 + \e} U.
\end{equation}
\item\label{item:shorttransition2}
If we further restrict to $U \geq t_g^{1/5}$, then we have that
\begin{equation}
\label{eqn:thml2}
\sum_{f \in \BB_0} \epsilon_f \frac{S_-(t_f)}{L(1,\ad f)} h(t_f) + \frac{1}{2\pi} \int_{-\infty}^{\infty} \frac{E_-(t)}{|\zeta(1 + 2it)|^2} h(t) \, dt \ll_{\e} t_g^{\frac{7}{4} + \e} U^{-\frac{5}{4}}.
\end{equation}
\end{enumerate}
\end{proposition}

The proof of \hyperref[prop:shorttransition]{Proposition \ref*{prop:shorttransition}}, which we give in \hyperref[sect:shorttransitionproof]{Section \ref*{sect:shorttransitionproof}}, proceeds via a series of steps. Taking this result for granted for the time being, we proceed to the proof of \hyperref[prop:firstmomentconductordrop]{Proposition \ref*{prop:firstmomentconductordrop}}.

\begin{proof}[Proof of {\hyperref[prop:firstmomentconductordrop]{Proposition \ref*{prop:firstmomentconductordrop}}}]
We use \hyperref[lem:afeFxf]{Lemmata \ref*{lem:afeFxf}} and \ref{lem:afeFxf2} to express each of the $L$-functions $L(1/2, \ad g \otimes f)$ and $|L(1/2 + it,\ad g)|^2$ as Dirichlet series. The sums involving $S_+(t_f)$ and $E_+(t)$ are dealt with using \eqref{eqn:thml1}. The sums involving $s_\pm(t_f)$ are seen to be $O_{\e}(t_g^{3/2 + \e} U^{-1/2})$ by using the spectral large sieve after an application of the Cauchy--Schwarz inequality, keeping in mind \hyperref[rem:newafe]{Remark \ref*{rem:newafe}}. In this way, we are left to deal with the sums involving $S_-(t_f)$ and $E_-(t)$. With $h(t)$ as in \eqref{eqn:UXhtdefeq}, we deduce that
\begin{multline*}
\sum_{\substack{f \in \BB_0 \\ 2t_g - U - U^{1 - \e} \leq t_f \leq 2t_g - U + U^{1 - \e}}} \frac{L\left(\frac{1}{2}, \ad g \otimes f\right)}{L(1,\ad f)} + \frac{1}{2\pi} \int\limits_{2t_g - U - U^{1 - \e} \leq |t| \leq 2t_g - U + U^{1 - \e}} \left|\frac{L\left(\frac{1}{2} + it,\ad g\right)}{\zeta(1 + 2it)}\right|^2 \, dt	\\
\ll_{\e} \left|\sum_{f \in \BB_0} \epsilon_f \frac{S_-(t_f)}{L(1,\ad f)} h(t_f) + \frac{1}{2\pi} \int_{-\infty}^{\infty} \frac{E_-(t)}{|\zeta(1 + 2it)|^2} h(t) \, dt\right| + t_g^{\frac{3}{2} + \e} U^{-\frac{1}{2}}.
\end{multline*}
Recall that $S_-(t_f)$ and $E_-(t)$ are sums of length $O_{\e}(t_g^{2 + \e})$. We proceed differently according to the size of $U$. For $t_g^{\e} \leq U \leq t_g^{1/5}$, we apply the spectral large sieve to get the bound
\[\sum_{f \in \BB_0} \epsilon_f \frac{S_-(t_f)}{L(1,\ad f)} h(t_f) + \frac{1}{2\pi} \int_{-\infty}^{\infty} \frac{E_-(t)}{|\zeta(1 + 2it)|^2} h(t) \, dt \ll_{\e} t_g^{\frac{3}{2} + \e} U^{\frac{1}{2}}.\]
For $t_g^{1/5} \leq U \leq t_g^{1/3 + \e}$, we apply \eqref{eqn:thml2}. This yields \eqref{eqn:firstmomentconductordrop} for $t_g^{\e} \leq U \leq t_g^{1/3 + \e}$, while for $1 \leq U \leq t_g^{\e}$, we simply invoke the convexity bound.
\end{proof}

\subsection{Proof of \texorpdfstring{\hyperref[prop:shorttransition]{Proposition \ref*{prop:shorttransition}}}{Proposition \ref{prop:shorttransition}}}
\label{sect:shorttransitionproof}

Let $t_g^{\e} \leq U \leq t_g^{1/3 + \e}$. By the Kuznetsov formula \eqref{eqn:Kuznetsovformula}, we have that
\begin{align}
\nonumber
\sum_{f \in \BB_0} \epsilon_f^{\frac{1 \mp 1}{2}} \frac{S_\pm(t_f)}{L(1,\ad f)} h(t_f) + & \frac{1}{2\pi} \int_{-\infty}^{\infty} \frac{E_\pm(t)}{|\zeta(1 + 2it)|^2} h(t) \, dt	\\
\label{eqn:diag}
& = \delta_{\pm,+} \sum_{m = 1}^{\infty} \frac{A_f(m,1)}{m} \int_{-\infty}^{\infty} V_{+}(m^2,t) h(t) \, d_{\spec}(t) \\
\label{eqn:offdiag}
& \qquad + \sum_{c,m,n = 1}^{\infty} \frac{A_F(m,n)}{m\sqrt{n}} \frac{S(1,\pm n;c)}{c} (\Kscr^\pm h V_\pm (m^2 n,\cdot) ) \left(\frac{\sqrt{n}}{c}\right).
\end{align}
The diagonal contribution \eqref{eqn:diag} is $O_{\e}(t_g^{1 + \e} U)$ by bounding trivially. The main challenge is to treat the off-diagonal \eqref{eqn:offdiag}. We note first of all that we can crudely truncate the off-diagonal sum to $c \leq t_g^{1000}$, up to a negligible error term of size $O(t_g^{-1000})$, in a standard way by using Weil's bound for Kloosterman sums and shifting the $t$-integral in the definition of $(\Kscr^\pm h V_\pm (m^2 n,\cdot) ) (\frac{\sqrt{n}}{c})$ to $\Im(t) = -\frac{1}{2} + \e$ or $\Im(t) = \frac{1}{2} - \e$ and then bounding absolutely, as done in \cite[Lemma 5]{Blo12}.

\subsubsection{The Positive Sign Case of the Off-Diagonal}

We first prove the bound \eqref{eqn:thml1}. The main idea is that the approximate functional equation sum is short due to conductor dropping while the spectral family we average over is relatively large. We show the following.

\begin{lemma}
\label{lem:possignKscrbound}
We have that
\begin{equation}
\label{eqn:possignKscrbound}
(\Kscr^+ h V_+ (m^2n,\cdot))(x) \ll t_g^{-500}
\end{equation}
unless $x \geq U^{1-\e/3} t_g$.
\end{lemma}

We use this to deduce the bound \eqref{eqn:thml1}.

\begin{proof}[Proof of {\hyperref[item:shorttransition1]{Proposition \ref*{prop:shorttransition} (1)}}]
We must show that the off-diagonal term \eqref{eqn:offdiag} is $O_{\e}(t_g^{1 + \e} U)$. Since we may restrict to $m^2 n \ll U^{1-\e} t_g^2$ by the decay properties \eqref{eqn:vdecay} of $V_+ (m^2n,\cdot) (x)$, no value of $c \in \N$ can satisfy $\frac{\sqrt{n}}{c} \geq U^{1-\e/3} t_g$ once $t_g$ is sufficiently large. Coupled with \eqref{eqn:possignKscrbound}, this shows that the off-diagonal is negligibly small, and so the only contribution comes from the diagonal, which trivially satisfies the required bound \eqref{eqn:thml1}.
\end{proof}

\begin{proof}[Proof of {\hyperref[lem:possignKscrbound]{Lemma \ref*{lem:possignKscrbound}}}]
It suffices to prove this for $h(t)$ redefined as
\[h(t) = \exp\left(-\left(\frac{t - 2t_g - U}{U^{1 - \e}}\right)^2\right).\]
We require the derivative bounds
\begin{equation}
\label{eqn:deriv}
\frac{d^j}{dt^j} h(t) V_\pm (m^2n,t) \ll_j (U^{-1 + \e})^j
\end{equation}
for $|t - 2t_g| \asymp U$ and $j \in \N_0$. Assuming this for the time being, we proceed. In the integral defining $(\Kscr^+ h V_+ (m^2n,\cdot))(x)$ given by \eqref{eqn:NscrpmKscrpmdefeq}, we insert a smooth bump function $W(t)$ compactly supported on
\begin{equation}
\label{eqn:assump}
|t - 2t_g - U|\leq U^{1 - \e/2}
\end{equation}
and satisfying $\|W^{(j)}\| \ll_j (U^{-1 + \e/2})^j$ for $j \in \N_0$, since $h(t)$ decays rapidly outside this interval. Then using the identity \cite[p.~180]{Wat44}
\[\frac{J_{2it}(2\pi x) - J_{-2it}(2\pi x)}{\cosh \pi t} = -2i\tanh(\pi t)\int_{-\infty}^{\infty} \cos(2\pi x \cosh \pi v) \cos(2\pi tv) \, dv\]
and the fact that $\tanh \pi t = 1 + O(e^{-\pi t})$ for $t \geq 0$, we see that it suffices to show that 
\[\int_{-\infty}^{\infty} e(x\cosh \pi v)\int_{-\infty}^{\infty} h(t) V_+ (m^2n,t ) W(t) e(\pm vt) \, dt \, dv \ll t_g^{-100}\]
unless $x \geq U^{1 - \e/3} t_g$. By integrating by parts multiple times in the inner integral and using \eqref{eqn:deriv}, we deduce that the required bound holds for the portion of the outer integral for which $|v| \gg U^{-1 + 2\e}$. To treat the remaining portion of the outer integral, we insert a smooth bump function $\Omega(Uv)$ such that $\Omega$ is supported on $(-U^{2\e}, U^{2\e})$ with derivatives satisfying $\| \Omega^{(j)}\| \ll_j (U^{-2\e})^j$ for $j \in \N_0$. Thus it suffices to show that under the assumption \eqref{eqn:assump}, we have that
\begin{equation}
\label{eqn:vint}
\int_{-\infty}^{\infty} e\left(x\cosh \frac{\pi v}{U} \pm \frac{vt}{U}\right) \Omega(v) \, dv \ll t_g^{-300}
\end{equation}
if $x < U^{1 - \e/3} t_g$. This follows by integration by parts after observing that the phase
\[\phi(v) \coloneqq x\cosh \frac{\pi v}{U} \pm \frac{vt}{U}\]
satisfies
\[|\phi'(v)| = \left|\frac{x}{U}\sinh \frac{\pi v}{U} \pm \frac{t}{U} \right| \gg \frac{t_g}{U},\]
and
\[|\phi^{(j)}(v)| \ll_j x U^{-j + \e}\]
for $j \geq 2$. We apply \cite[Lemma 8.1]{BKY13} with $R = t_g/U$, $Y = x$, and $Q = U$ being the key parameters, which tells us that the integral \eqref{eqn:vint} is negligible provided $R\gg t_g^\e$ and $QRY^{-\frac{1}{2}}\gg t_g^{\e}$, which is the case here.

It remains to establish \eqref{eqn:deriv}. It is clear that the derivatives of $h(t)$ satisfy the bound, so we just need to consider the derivatives of $V_\pm (m^2 n,t)$. By taking $\sigma = \e$ in \eqref{eqn:vdef} and using the rapid decay of $e^{s^2}$, it suffices to show that
\begin{equation}
\label{eqn:gder}
\frac{d^j}{dt^j} \frac{G\left(\frac{1}{2} + s,t\right)}{G\left(\frac{1}{2},t\right)} \ll_j (U^{-1 + \e})^j
\end{equation}
for $|s| \leq t_g^\e$ and $j \in \N_0$. By Stirling's estimates (see \cite[(2.4)]{BlK19b}), for $y \geq t_g^\e$ and $\kappa > 0$ a fixed constant, there exists some function $\psi_{s,M}(y)$ satisfying $\psi_{s,M}^{(j)}(y) \ll_j y^{-j}$ for any $j \in \N_0$ such that
\[\frac{\Gamma_{\R}(s + \kappa + iy)}{\Gamma_{\R}(\kappa + iy)} = y^{\frac{s}{2}} \psi_{s,M}(y) + O_{s,M}(y^{-M})\]
for any $M > 0$. This gives \eqref{eqn:gder} (noting that the error term is the same under differentiation by Cauchy's integral formula) since each gamma ratio in $\frac{G(1/2 + s,t)}{G(1/2,t)}$ can be written as $\frac{\Gamma_{\R}(s + \kappa + iy)}{\Gamma_{\R}(\kappa + iy)}$ with $y\gg U$ under the assumption \eqref{eqn:assump}.
\end{proof}

\subsubsection{The Negative Sign Case of the Off-Diagonal}

We next prove the bound \eqref{eqn:thml2}. In this case, the off-diagonal is not empty but it is very nearly so; that is, we shall show that we may restrict the sum over $c \in \N$ to $c \ll t_g^\e$. An analysis of the $\Kscr^-$-transform then shows that the off-diagonal sum may be restricted to a short interval (see \eqref{eqn:xranget0} and \eqref{eqn:xranget}). We would then like to bound trivially, but two problems arise that we must circumvent. First, we do not know the Ramanujan conjecture, or efficient versions of this on average over \emph{short} intervals; we instead use $L$-functions and the spectral large sieve to obtain reasonable averaged bounds for sums of Hecke eigenvalues. Second, even if we were able to appeal to the Ramanujan conjecture, a trivial bound is insufficient in the range $U\leq t_g^{1/3 - \e}$. We get further savings by obtaining some stationary phase cancellation in the $\Kscr^-$-transform in \hyperref[lem:stat]{Lemma \ref*{lem:stat}}.

\begin{proof}[Proof of {\hyperref[item:shorttransition2]{Proposition \ref*{prop:shorttransition} (2)}}]
We assume that 
\[t_g^{\frac{1}{5}} \leq U \leq t_g^{\frac{1}{3} + \e}.\]
It suffices to consider $h(t)$ redefined as
\[h(t) = \exp\left(-\left(\frac{t - 2t_g - U}{U^{1 - \e}}\right)^2\right).\]
We must show that
\[\sum_{c,m,n = 1}^{\infty} \frac{A_F(m,n)}{m\sqrt{n}} \frac{S(1,-n;c)}{c} \int_{-\infty}^{\infty} \cosh \pi t K_{2it}\left(\frac{4 \pi \sqrt{n}}{c}\right) h(t) V_-(m^2 n, t) W(t) t \tanh \pi t \, dt\]
is $O_{\e}(t_g^{7/4 + \e} U^{-5/4})$, where $W(t)$ is a bump function defined exactly as in the proof of \hyperref[lem:possignKscrbound]{Lemma \ref*{lem:possignKscrbound}} with support as in \eqref{eqn:assump}. Using the identity \cite[8.432.4]{GR15}
\[\sinh \pi t K_{2it}(2\pi x) = \frac{\pi \tanh \pi t}{2} \int_{-\infty}^{\infty} \cos(2\pi x \sinh \pi v) e(tv) \, dv,\]
it suffices to prove this bound for
\[t_g \sum_{c,m,n = 1}^{\infty} \frac{A_F(m,n)}{m\sqrt{n}} \frac{S(1,-n;c)}{c} \int_{-\infty}^{\infty} e\left(\pm \frac{2\pi \sqrt{n}}{c} \sinh \pi v\right) \int_{-\infty}^{\infty} e(tv) h(t) V_-(m^2 n, t) W(t) \, dt \, dv.\]
The desired bound holds for the portion of the outer integral for which $|v| \geq U^{-1 + \e}$ by \eqref{eqn:deriv} and repeated integration by parts in the inner integral. To treat the remaining portion of the outer integral, we insert a smooth bump function $\Omega(Uv)$ as defined exactly as in the proof of \hyperref[lem:possignKscrbound]{Lemma \ref*{lem:possignKscrbound}}, so that we are left with
\begin{multline*}
t_g \sum_{c,m,n = 1}^{\infty} \frac{A_F(m,n)}{m \sqrt{n}} \frac{S(1,-n;c)}{c}	\\
\times \int_{-\infty}^{\infty} e\left(\pm \frac{2\pi \sqrt{n}}{c} \sinh \frac{\pi v}{U} + \frac{tv}{U}\right) \Omega(v) \int_{-\infty}^{\infty} h(t) V_-(m^2 n,t) W(t) \, \frac{dt}{U} \, dv.
\end{multline*}
If $\frac{\sqrt{n}}{c} \leq t_g^{1 - \e}$, then the outer integral is readily seen to be negligible via repeated integration by parts using \cite[Lemma 8.2]{BKY13} with $R = t_g/U$, $Y = \sqrt{n}/c$, and $Q = U$. Thus we may restrict to $\frac{\sqrt{n}}{c} > t_g^{1 - \e}$; as we may additionally restrict to $m^2 n \ll U^{\e} t_g^2$ by \eqref{eqn:vdecay}, we are left with the ranges
\[t_g^{2 - \e} < n < t_g^{2 + \e}, \qquad 1 \leq c, m \leq t_g^\e.\]

A further simplification we can make is the following. Since $U \geq t_g^{1/5}$, we may take the power series expansion of $ \sinh \frac{\pi v}{U}$ and absorb the exponential of all terms in the expansion beyond the cubic term into the weight function $\Omega(v)$. After opening up the Kloosterman sum $S(1,-n;c) = \sum_{a \in (\Z/c\Z)^{\times}} e(\frac{-\overline{a} + an}{c})$, using the Hecke relations \eqref{eqn:GL3Heckerelations}, and making the change of variables $n \mapsto dn$, we are thereby left with showing that
\[t_g \sum_{n = 1}^{\infty} \frac{A_F(1,n)}{\sqrt{n}} e\left(\frac{adn}{c}\right) \Psi\left(\frac{n}{t_g^2}\right) \int_{-\infty}^{\infty} e\left(\pm \frac{2\pi \sqrt{dn}}{c} \left(\frac{\pi v}{U} + \frac{\pi^3 v^3}{3! U^3}\right) + \frac{tv}{U}\right) \Omega(v) \, dv \ll_{\e} t_g^{\frac{7}{4} + \e} U^{-\frac{5}{4}}\]
for any $t\asymp 2t_g$, $1 \leq c,d \leq t_g^\e$, and $a \in (\Z/c\Z)^{\times}$, where $\Psi$ is a smooth function, compactly supported on $(t_g^{-\e},t_g^\e)$, with derivatives satisfying $\| \Psi^{(j)} \| \ll_j (t_g^\e)^{j}$ for any $j \in \N_0$. By Mellin inversion, this expression is equal to
\[\frac{1}{2\pi i} \int_{\sigma - i\infty}^{\sigma + i\infty} t_g^{1 + 2s} \Phi_F\left(\frac{c}{(c,d)},\overline{\frac{ad}{(c,d)}},1;\frac{1}{2} + s\right) \widetilde{I}(s) \, ds,\]
where $\sigma>0$, $\Phi_F$ is the Vorono\u{\i} series \eqref{eqn:PhiFdefeq}, and
\begin{equation}
\label{eqn:I-mell}
\widetilde{I}(s) \coloneqq \int_0^\infty \left(\int_{-\infty}^{\infty} e\left(\pm \frac{2\pi t_g \sqrt{dx}}{c} \left(\frac{\pi v}{U} + \frac{\pi^3 v^3}{3!U^3}\right) + \frac{tv}{U}\right) \Omega(v) \, dv \right) \Psi(x) x^{s - 1} \, dx.
\end{equation}
We may shift the line of integration to $\sigma = 0$ and repeatedly integrate by parts in \eqref{eqn:I-mell} in order to restrict the line of integration to the range
\[|s| \leq \frac{t_g^{1 + \e}}{U}\]
at the cost of a negligible error term. To bound the remaining integral, we show in \hyperref[lem:stat]{Lemma \ref*{lem:stat}} that $\widetilde{I}(iy) \ll_{\e} t_g^{-1/2 + \e} U^{-1/2}$, at which point we are left with showing that
\begin{equation}
\label{eqn:firstmomentPhiF}
\int_{|y| \leq \frac{t_g^{1 + \e}}{U}} \left|\Phi_F\left(\frac{c}{(c,d)},\overline{\frac{ad}{(c,d)}},1;\frac{1}{2} + iy\right)\right| \, dy \ll_{\e} t_g^{\frac{5}{4} + \e} U^{-\frac{3}{4}}.
\end{equation}
We achieve this via the Cauchy--Schwarz inequality coupled with \hyperref[lem:Phisecondmoment]{Lemma \ref*{lem:Phisecondmoment}} below.
\end{proof}

\subsubsection{A Stationary Phase Estimate}

In the course of the proof of \hyperref[item:shorttransition2]{Proposition \ref*{prop:shorttransition} \ref*{item:shorttransition2}}, we invoked the following estimate, which we now prove.

\begin{lemma}
\label{lem:stat}
We have that
\begin{equation}
\label{eqn:lemmbound}
\int_{t_g^{-\e}}^{t_g^\e} \left| \int_{-\infty}^{\infty} e\left(v\left(\frac{t}{U} \pm \frac{2\pi^2 t_g x}{Uc}\right) - v^3\frac{\pi^4 t_g x}{3U^3c} \right) \Omega(v) \, dv \right| \, dx \ll t_g^{-\frac{1}{2} + \e} U^{-\frac{1}{2}}.
\end{equation}
\end{lemma}

\begin{proof}
First we consider the contribution of the small values of $v$. Let $V_0 \coloneqq (\frac{t_g}{U^3})^{-1/3 + \e}$. The range $|v| < V_0$ may essentially be picked out using a bump function $\Omega_0(v/V_0)$, where $\Omega_0$ is smooth, compactly supported on $[-1,1]$ and has bounded derivatives. In this range, the exponential of the cubic term in the phase may be absorbed in the weight function $\Omega(v)$, since this exponential has derivatives bounded by powers of $t_g^\e$. In turn, we may then replace $\pm x$ by $-x$ via a change of variables. By integrating by parts repeatedly in the integral
\[\int_{-\infty}^{\infty} e\left(v\left(\frac{t}{U}-\frac{2\pi^2 t_g x}{Uc}\right) \right) \Omega(v) \Omega_0\left(\frac{v}{V_0}\right) \, dv,\]
we see that we may restrict to
\begin{equation}
\label{eqn:xranget0} \left|\frac{t}{U} - \frac{2\pi^2 t_g x}{Uc}\right|\ll t_g^{\e} \frac{1}{V_0}, 
\end{equation}
which implies that the outer integral is restricted to an interval of size at most $t_g^{\e} \frac{U}{t_gV_0}$. Using this and the fact that the inner integral is restricted to an interval of size $O(V_0)$, by the support of $\Omega_0$, we bound trivially to get the bound $O(t_g^\e \frac{U}{t_g})$ for the double integral, which is sufficient.

Now we consider the contribution of larger values of $v$; by symmetry, it suffices to consider $v$ positive. For 
\begin{equation}
\label{eqn:v1range}
\left(\frac{t_g}{U^3}\right)^{-\frac{1}{3} + \e} \leq V_1 \leq t_g^\e,
\end{equation}
we insert a bump function $\Omega_1(v/V_1)$, where $\Omega_1$ is smooth, compactly supported on $[1,2]$ and has $j$-th derivative bounded by $(t_g^\e)^j$. After a substitution, the inner integral is equal to
\begin{equation}
\label{eqn:vintg}
V_1 \int_{-\infty}^{\infty} e(\phi(v)) \Omega_1(v) \, dv,
\end{equation}
where $\Omega(vV_1)$ has been absorbed into $\Omega_1(v)$, and 
\[\phi(v) \coloneqq v \left(\frac{tV_1}{U}-\frac{2\pi^2 t_g xV_1}{Uc}\right) -v^3\frac{\pi^4t_g x V_1^3}{3U^3c}\]
is the phase, with derivatives
\begin{align*}
\phi'(v) & = \frac{1}{U}\left(tV_1 -\frac{2\pi^2 x t_g V_1}{c}\right)-\frac{t_g x}{U^3} \frac{\pi^4 V_1^3 v^2}{c},\\
\phi''(v) & = -\frac{t_g x}{U^3} \frac{2 \pi^4 V_1^3 v}{c}.
\end{align*}
First, by repeated integration by parts, we see that we may restrict to
\begin{equation}
\label{eqn:xranget}
\left|\frac{tV_1}{U}-\frac{2\pi^2 x t_g V_1 }{Uc}\right| \ll t_g^\e \frac{t_g V_1^3}{U^3},
\end{equation}
To see this, one may use \cite[Lemma 8.1]{BKY13} with the parameters therein being $R = \frac{t_g V_1^3}{U^3}$, $Y = t_g$, and $Q = (\frac{U}{V_1})^{3/2}$. By \eqref{eqn:xranget}, the outer integral is restricted to an interval of size at most $t_g^\e \frac{Uc}{t_g} \frac{t_g V_1^3}{U^3}$. Now we will bound \eqref{eqn:vintg} by $V_1 |\phi''|^{-\frac{1}{2}} \asymp (\frac{t_g x V_1 }{U^3 c})^{-\frac{1}{2}}$. This is expected by stationary phase analysis in the inner integral, but we give the details below. Once we have this bound, the double integral is seen to be bounded by $t_g^\e \frac{U}{t_g} (\frac{t_g}{U^3})^\frac{1}{2} V_1^\frac52$. By taking the maximum over $V_1$ in the range \eqref{eqn:v1range}, we obtain the required bound in \eqref{eqn:lemmbound}.

\textit{Case 1: $tV_1 -\frac{2\pi^2 x t_g V_1}{c}> 0$.} In this range, there are two stationary points, possibly outside the interval $[1,2]$. Let $\nu$ be the positive stationary point, so that $\nu > 0$ and $\phi'( \nu)=0$. We split the integral in \eqref{eqn:vintg} into two integrals $I_1 + I_2$, where $I_1$ is an integral over
\[S \coloneqq \left\{v\in [1,2] : |v - \nu| \leq \left(\frac{t_g V_1^3}{U^3}\right)^{-\frac{1}{2}} \right\},\]
and $I_2$ is an integral over $[1,2] \backslash S$. Bounding trivially, we have $|I_1| \ll (\frac{t_g V_1}{U^3})^{-\frac{1}{2}}$. For $I_2$, we integrate by parts once as follows:
\begin{align}
\nonumber
I_2 & = V_1 \int_{[1,2]\backslash S} \frac{\phi'(v)}{\phi'(v)} e(\phi(v))\Omega_1(v) \, dv \\
\label{eqn:intbypartsonce}
& \ll V_1 \left|\int_{[1,2] \backslash S} \left(\frac{ \Omega_1'(v)}{\phi'(v)} - \frac{\phi''(v)\Omega_1(v)}{(\phi'(v))^2} \right) e(\phi(v)) \, dv \right| + t_g^\e V_1 \sup_{v\in [1,2] \backslash S} \left|\frac{1}{\phi'(v)}\right|.
\end{align}
Writing $v = \nu + u$ with $|u| \geq (\frac{t_g V_1^3}{U^3})^{-\frac{1}{2}}$, we have that for $v \in [1,2] \backslash S$,
\[|\phi'(v)| = \left|\frac{\pi^4 x t_g V_1^3}{cU^3}(2\nu u + u^2)\right| = \left|\frac{\pi^4 x t_g V_1^3}{cU^3} (\nu + v) u\right| \gg \left|\frac{\pi^4 x t_g V_1^3}{cU^3} u\right| \gg t_g^{-\e} \left(\frac{ t_g V_1^3}{U^3}\right)^{\frac{1}{2}}.\]
Thus 
\[V_1 \sup_{v \in [1,2]\backslash S} \left| \frac{1}{\phi'(v)}\right| \ll t_g^\e \left(\frac{t_g V_1}{U^3}\right)^{-\frac{1}{2}}, \qquad V_1 \int_{[1,2] \backslash S} \left| \frac{ \Omega'(v)}{\phi'(v)}\right| \, dv \ll t_g^\e \left(\frac{t_g V_1}{U^3}\right)^{-\frac{1}{2}},\]
and
\[V_1 \int_{[1,2] \backslash S} \left| \frac{\phi''(v)\Omega(v)}{(\phi'(v))^2} \right| \, dv \ll t_g^\e \left(\frac{t_g V_1}{U^3}\right)^{-\frac{1}{2}} \int_{(\frac{t_g V_1^3}{U^3})^{-\frac{1}{2}}}^2 \frac{1}{u} \, du \ll t_g^\e \left(\frac{t_g V_1}{U^3}\right)^{-\frac{1}{2}}.\]

\textit{Case 2: $tV_1 -\frac{2\pi^2 x t_g V_1}{c}\leq 0$.} In this range, there is no stationary point, and we have that
\begin{equation}
\label{eqn:nostatpoint}
|\phi'(v)| \geq t_g^{-\e}\frac{ t_g V_1^3}{U^3}
\end{equation}
for $v \in [1,2]$. We integrate by parts once as in \eqref{eqn:intbypartsonce}, and then the required bound follows easily by using \eqref{eqn:nostatpoint}.
\end{proof}

\subsubsection{The Second Moment of Vorono\u{\i} Series}

Finally, we must also prove the following estimate, which we invoked in the proof of \hyperref[item:shorttransition2]{Proposition \ref*{prop:shorttransition} \ref*{item:shorttransition2}}.

\begin{lemma}
\label{lem:Phisecondmoment}
For $1 \leq V \leq t_g$, $c \ll t_g^{\e}$, and $d \in (\Z/c\Z)^{\times}$, we have that
\[\int_{-V}^{V} \left|\Phi_F\left(c,d,1;\frac{1}{2} + iy\right)\right|^2 \, dy \ll_{\e} t_g^{1 + \e} V^{\frac{1}{2}}.\]
\end{lemma}

\begin{proof}
By character orthogonality, we may write
\[e\left(\frac{n\overline{d}}{c}\right) = \sum_{a \mid (c,n)} \frac{1}{\varphi\left(\frac{c}{a}\right)} \sum_{\chi \hspace{-.25cm} \pmod{\frac{c}{a}}} \tau(\chi) \chi(d) \overline{\chi}\left(\frac{n}{a}\right).\]
Upon making the change of variables $n \mapsto an$ and relabelling $a$ as $a_1$ and $c/a$ as $a_2$, we deduce that for $\Re(s) > 1$,
\[\Phi_F(c,d,1;s) = \sum_{a_1 a_2 = c} \frac{1}{a_1^s \varphi(a_2)} \sum_{\chi \hspace{-.25cm} \pmod{a_2}} \tau(\chi) \chi(d) \sum_{n = 1}^{\infty} \frac{A_F(1,a_1 n) \overline{\chi}(n)}{n^s}.\]
We next make use of the Hecke relations,
\[A_F(1,a_1 n) = \sum_{\substack{b_1 \mid (a_1,n) \\ b_2 \mid \left(b_1,\frac{n}{b_1}\right)}} \mu(b_1) \mu(b_2) A_F\left(\frac{b_1}{b_2},\frac{a_1}{b_1}\right) A_F\left(1,\frac{n}{b_1 b_2}\right),\]
which follows from \cite[Theorem 6.4.11]{Gol06} and the M\"{o}bius inversion formula. Inserting this and making the change of variables $n \mapsto b_1 b_2 n$, $a_1 \mapsto a_1 b_1$, and $b_1 \mapsto b_1 b_2$, we deduce that
\[\Phi_F(c,d,1;s) = \sum_{\substack{a_1 a_2 b_1 b_2 = c \\ (b_1,b_2) = 1}} \frac{\mu(b_1) \mu^2(b_2) \overline{\chi}(b_1) \overline{\chi}^2(b_2) A_F(b_1,a_1)}{a_1^s \varphi(a_2) b_1^{2s} b_2^{3s}} \sum_{\chi \hspace{-.25cm} \pmod{a_2}} \tau(\chi) \chi(d) L(s,\ad g \otimes \overline{\chi}).\]
Note furthermore that if $\chi^{\ast}$ is the primitive character modulo $a_2^{\ast}$ that induces $\chi$, where $a_2^{\ast} \mid a_2$, then
\[L(s,\ad g \otimes \overline{\chi}) = L(s,\ad g \otimes \overline{\chi}^{\ast}) \prod_{p \mid a_2} \left(1 - A_F(1,p) \overline{\chi}(p) p^{-s} + A_F(p,1) \overline{\chi}^2(p) p^{-2s} - \overline{\chi}^3(p) p^{-3s}\right).\]
This identity extends by analytic continuation to $s = 1/2 + iy$. Since $c \ll t_g^{\e}$, we deduce that
\[\int_{-V}^{V} \left|\Phi_F\left(c,d,1;\frac{1}{2} + iy\right)\right|^2 \, dy \ll_{\e} t_g^{\e} \sup_{a \mid c} \sup_{\substack{\chi \hspace{-.25cm} \pmod{a} \\ \chi \text{ primitive}}} \int_{-V}^{V} \left|L\left(\frac{1}{2} + iy,\ad g \otimes \chi\right)\right|^2 \, dy.\]
Just as in \hyperref[lem:L1/2F2ndmoment]{Lemma \ref*{lem:L1/2F2ndmoment}}, we may bound this by $O_{\e}(t_g^{1 + \e} U^{1/2})$ by using the approximate functional equation \cite[Theorem 5.3]{IK04} to write $L(1/2 + iy,F \otimes \chi)$ in terms of a Dirichlet polynomial and then invoking the Montgomery--Vaughan mean value theorem for Dirichlet polynomials \cite[Corollary 3]{MV74}, noting that the analytic conductor of $L(1/2 + iy,F \otimes \chi)$ is $O_{\e}(t_g^{2 + \e} (1 + |y|))$ as $|y| \leq V \leq t_g$.
\end{proof}

\section{Bounds for Mixed Moments of \texorpdfstring{$L$}{L}-Functions in the Tail Range}

\subsection{Proof of \texorpdfstring{\hyperref[item:tail]{Proposition \ref*{prop:fourranges} \ref*{item:tail}}}{Proposition \ref{prop:fourranges} \ref{item:tail}}}

The proof of \hyperref[item:tail]{Proposition \ref*{prop:fourranges} \ref*{item:tail}}, namely the bound \eqref{eqn:tail} for the tail range, follows in a straightforward manner from bounds attained via the spectral large sieve.

\begin{proof}[Proof of {\hyperref[item:tail]{Proposition \ref*{prop:fourranges} (4)}}]
By the lower bound $L(1,\ad g) \gg_{\e} t_g^{-\e}$ and the asymptotic formula \eqref{eqn:Htasymp} for $H(t)$, it suffices to show that
\[\begin{drcases*}
\sum_{\substack{f \in \BB_0 \\ t_f \geq 2t_g}} \frac{e^{-\pi(t_f - 2t_g)}}{t_f^{3/2}(1 + t_f - 2t_g)^{1/2}} \frac{L\left(\frac{1}{2},f\right) L\left(\frac{1}{2},\ad g \otimes f\right)}{L(1,\ad f)} &	\\
\frac{1}{2\pi} \int\limits_{|t| \geq 2t_g} \frac{e^{-\pi(|t| - 2t_g)}}{|t|^{3/2}(1 + |t| - 2t_g)^{1/2}} \left|\frac{\zeta\left(\frac{1}{2} + it\right) L\left(\frac{1}{2} + it,\ad g\right)}{\zeta(1 + 2it)}\right|^2 \, dt & \end{drcases*}
\ll_{\e} t_g^{\e}.\]
We dyadically decompose both the sum over $f$ and the integral over $t$, so that we are left with proving the bounds
\[\begin{drcases*}
\sum_{\substack{f \in \BB_0 \\ T - U \leq t_f \leq T + U}} \frac{L\left(\frac{1}{2},f\right) L\left(\frac{1}{2},\ad g \otimes f\right)}{L(1,\ad f)} &	\\
\frac{1}{2\pi} \int\limits_{T - U \leq |t| \leq T + U} \left|\frac{\zeta\left(\frac{1}{2} + it\right) L\left(\frac{1}{2} + it,\ad g\right)}{\zeta(1 + 2it)}\right|^2 \, dt & \end{drcases*}
\ll_{\e} t_g^{\frac{3}{2} + \e} U^{\frac{1}{2}} e^{\pi U}\]
for $2t_g \leq T \leq 3t_g$ and $U = \frac{T}{2} + 1 - t_g$, as well as the bounds
\[\begin{drcases*}
\sum_{\substack{f \in \BB_0 \\ T \leq t_f \leq 2T}} \frac{L\left(\frac{1}{2},f\right) L\left(\frac{1}{2},\ad g \otimes f\right)}{L(1,\ad f)} &	\\
\frac{1}{2\pi} \int\limits_{T \leq |t| \leq 2T} \left|\frac{\zeta\left(\frac{1}{2} + it\right) L\left(\frac{1}{2} + it,\ad g\right)}{\zeta(1 + 2it)}\right|^2 \, dt & \end{drcases*}
\ll_{\e} t_g^{\e} T^2 (\log T)^{-2} e^{\pi(T - 2t_g)}\]
for $T \geq 3t_g$. Via the Cauchy--Schwarz inequality, the former follows (with polynomial room to spare unless $U = o(\log t_g)$) from the bounds \eqref{eqn:largesievebounds2} and \eqref{eqn:largesievebounds3} from \hyperref[prop:largesievebounds1]{Propositions \ref*{prop:largesievebounds1}} and \ref{prop:largesievebounds3} arising from the spectral large sieve, while the latter follows (with exponential room to spare) from the bounds \eqref{eqn:largesievebounds1} and \eqref{eqn:largesievebounds2}.
\end{proof}

\section{Extensions and Improvements}

We finish by sketching how the methods in this paper extend to yield \hyperref[thm:L4modified]{Theorem \ref*{thm:L4modified}} and discussing some conditional approaches that lead to strengthenings of \hyperref[thm:L4]{Theorems \ref*{thm:L4}} and \ref{thm:L4modified}.

\subsection{A Sketch of the Proof of \texorpdfstring{\hyperref[thm:L4modified]{Theorem \ref*{thm:L4modified}}}{Theorem \ref{thm:L4modified}}}
\label{sect:proofmodifiedsketch}

The method of proof of \hyperref[thm:L4]{Theorem \ref*{thm:L4}} can readily be seen to extend to Hecke--Maa\ss{} newforms on $\Gamma_0(q) \backslash \Hb$. The key reason for this is that all of the tools used, such as the Watson--Ichino triple product formula and various spectral reciprocity formul\ae{}, remain applicable in this more general setting. Moreover, all of the estimates for various moments of $L$-functions given in this paper are purely archimedean in nature, and so the same estimates hold on $\Gamma_0(q) \backslash \Hb$ (albeit with unspecified dependence on $q$). We list below the major alterations required in order to extend \hyperref[thm:L4]{Theorem \ref*{thm:L4}} in this direction.
\begin{enumerate}[leftmargin=*,label=\textup{(\arabic*)}]
\item Via Parseval's identity for $L^2(\Gamma_0(q) \backslash \Hb)$, for $g$ a Hecke--Maa\ss{} newform on $\Gamma_0(q) \backslash \Hb$, we express $\|g\|_4^4$ in terms of a spectral expansion of triple products of automorphic forms. Choosing an explicit orthonormal basis of cusp forms and Eisenstein series in terms of newforms and oldforms, and then applying the Watson--Ichino triple product formula, we obtain a level $q$ analogue of the identity \eqref{eqn:L4toLfunctions}, namely
\begin{multline*}
\int_{\Gamma_0(q) \backslash \Hb} |g(z)|^4 \, \frac{3}{\pi [\Gamma : \Gamma_0(q)]} \frac{dx \, dy}{y^2} = 1 + \sum_{q_1 q_2 = q} \sum_{f \in \BB_0^{\ast}(\Gamma_0(q_1))} c_{f,g,q_2} \frac{L\left(\frac{1}{2},f\right) L\left(\frac{1}{2},\ad g \otimes f\right)}{L(1,\ad f) L(1,\ad g)^2} H(t_f)	\\
+ \frac{1}{2\pi} \int_{-\infty}^{\infty} c_{t,g,q_1}\left|\frac{\zeta\left(\frac{1}{2} + it\right) L\left(\frac{1}{2} + it,\ad g\right)}{\zeta(1 + 2it) L(1,\ad g)}\right|^2 H(t) \, dt.
\end{multline*}
(Cf.\ \cite[Propositions 1.13 and 1.16]{HK20}.) Here $\BB_0^{\ast}(\Gamma_0(q_1))$ denotes an orthonormal basis of Hecke--Maa\ss{} newforms of level $q_1$, while $c_{f,g,q_2},c_{t,g,q_1}$ are local constants arising from the Watson--Ichino triple product formula that are bounded by a constant dependent only on $q$.
\item Next, we derive level $q$ analogues of the $\GL_3 \times \GL_2 \leftrightsquigarrow \GL_4 \times \GL_1$ and $\GL_4 \times \GL_2 \leftrightsquigarrow \GL_4 \times \GL_2$ spectral reciprocity formul\ae{} given in \hyperref[thm:3x2reciprocity]{Theorems \ref*{thm:3x2reciprocity}} and \ref{thm:4x2reciprocity}. The methods of proof are essentially identical; the chief modifications are the usage the Kuznetsov and Petersson formul\ae{} for $\Gamma_0(q) \backslash \Hb$ associated to $(\infty,0)$ pair of cusps \cite[Theorems A.16 and A.19]{HK20}, which naturally introduces the root number into this formula, and the usage the $\GL_3$ Vorono\u{\i} summation formula for $\ad g$ with $g$ of level $q$ \cite{HL24}.
\item Once we have a level $q$ analogue of \hyperref[thm:4x2reciprocity]{Theorem \ref*{thm:4x2reciprocity}}, in order to prove the level $q$ analogue of \hyperref[prop:momentsboundsinitial]{Proposition \ref*{prop:momentsboundsinitial}} (which in turn yields the level $q$ analogue of \hyperref[item:initial]{Proposition \ref*{prop:fourranges} \ref*{item:initial}}), we require level $q$ analogues of \hyperref[prop:largesievebounds1]{Propositions \ref*{prop:largesievebounds1}}, \ref{prop:upperbounds}, and \ref{prop:Jutilabounds}. The former result is immediate since the spectral large sieve also holds for level $q$ cusp forms, the second result follows from the level $q$ analogue of \hyperref[thm:3x2reciprocity]{Theorem \ref*{thm:3x2reciprocity}}, while the latter result follows from \cite[Theorem 7.1]{HK24}.
\item The proof of the level $q$ analogue of \hyperref[item:bulk]{Proposition \ref*{prop:fourranges} \ref*{item:bulk}} is via the identical method except using the level $q$ Kuznetsov formula (cf.\ \cite[Proof of Proposition 1.21 (2)]{HK20}).
\item Finally, to prove the level $q$ analogue of \hyperref[item:transition]{Proposition \ref*{prop:fourranges} \ref*{item:transition}}, we require the level $q$ analogue of \hyperref[prop:momentsboundstransition]{Proposition \ref*{prop:momentsboundstransition}}. In turn, this requires level $q$ analogues of \hyperref[prop:largesievebounds3]{Propositions \ref*{prop:largesievebounds3}}, \ref{prop:firstmomentconductordrop}, and \ref{prop:thirdmoment}. The former result is again an immediate consequence of the spectral large sieve, the second result is via the same method of proof except using the level $q$ Kuznetsov formula, and the final result follows from \cite[Theorem 4.1]{AW23}.
\end{enumerate}

To prove \hyperref[thm:L4modified]{Theorem \ref*{thm:L4modified}} for Hecke--Maa\ss{} newforms on $\Gamma^D \backslash \Hb$, where $D$ is the indefinite quaternion division algebra over $\Q$ of squarefree discriminant $q$, we again begin via Parseval's identity for $L^2(\Gamma^D \backslash \Hb)$ coupled with the Watson--Ichino triple product formula, which yields an appropriate analogue of the identity \eqref{eqn:L4toLfunctions} of the form
\[\int_{\Gamma^D \backslash \Hb} |g(z)|^4 \, \frac{3}{\pi \left[\Gamma : \Gamma^D\right]} \frac{dx \, dy}{y^2} = 1 + \sum_{f \in \BB_0^{\ast}(\Gamma^D)} c_{f,g,q} \frac{L\left(\frac{1}{2},f\right) L\left(\frac{1}{2},\ad g \otimes f\right)}{L(1,\ad f) L(1,\ad g)^2} H(t_f).\]
Here $\BB_0^{\ast}(\Gamma^D)$ denotes an orthonormal basis of Hecke--Maa\ss{} cusp forms for $\Gamma^D \backslash \Hb$ (which are all newforms), while $c_{f,g,q_2}$ are once more local constants arising from the Watson--Ichino triple product formula that are bounded in absolute value by a constant dependent only on $q$; note that there is no integral over $t \in \R$ since the compactness of $\Gamma^D \backslash \Hb$ means that there is no continuous spectrum of the Laplacian. Via the Jacquet--Langlands correspondence, each $f \in \BB_0^{\ast}(\Gamma^D)$ corresponds bijectively with a Hecke--Maa\ss{} newform on $\Gamma_0(q) \backslash \Hb$ with identical spectral parameter and Hecke eigenvalues. Thus we in turn have that
\[\int_{\Gamma^D \backslash \Hb} |g(z)|^4 \, \frac{3}{\pi \left[\Gamma : \Gamma^D\right]} \frac{dx \, dy}{y^2} = 1 + \sum_{f \in \BB_0^{\ast}(\Gamma_0(q))} c_{f,g,q} \frac{L\left(\frac{1}{2},f\right) L\left(\frac{1}{2},\ad g \otimes f\right)}{L(1,\ad f) L(1,\ad g)^2} H(t_f),\]
at which point the desired result now follows by the same method sketched above for Hecke--Maa\ss{} newforms on $\Gamma_0(q) \backslash \Hb$.

\subsection{Conditional Improvements via the Generalised Lindel\"{o}f Hypothesis}
\label{sect:GLH}

As discussed in \hyperref[sect:improvementsimprovementsubsect]{Section \ref*{sect:improvementsimprovementsubsect}}, Watson observed that the essentially optimal bound $\|g\|_4 \ll_{\e} \lambda_g^{\e}$ follows from the generalised Lindel\"{o}f hypothesis for $\GL_3 \times \GL_2$ Rankin--Selberg $L$-functions and $\GL_2$ standard $L$-functions. Our method of proof of \hyperref[thm:L4]{Theorems \ref*{thm:L4}} and \ref{thm:L4modified} demonstrates that the same holds under a slightly weaker assumption.

\begin{proposition}
\label{prop:L4GLHconditional}
Let $g$ be a Hecke--Maa\ss{} newform of Laplacian eigenvalue $\lambda_g$ on either $\Gamma_0(q) \backslash \Hb$ or $\Gamma^D \backslash \Hb$, where $q$ is squarefree and fixed and $D$ is the indefinite quaternion division algebra over $\Q$ of discriminant $q$. Under the assumption of the generalised Lindel\"{o}f hypothesis for $\GL_2$ standard $L$-functions, $\|g\|_4 \ll_{\e} \lambda_g^{\e}$.
\end{proposition}

\begin{proof}[Sketch of proof]
From \hyperref[rem:shortinitialoptimal]{Remarks \ref*{rem:shortinitialoptimal}} and \ref{rem:shorttransitionoptimal}, all that is needed is the improved bound $O_{\e}(t_g^{1 + \e} T)$ for \eqref{eqn:shortinitialtobeproved} in the ranges $T \leq t_g^{3/13}$ and $t_g^{10/13} \leq T \leq t_g^{1 - \alpha}$ and the improved bound $O_{\e}(t_g^{3/2 + \e} U^{1/2})$ for \eqref{eqn:shorttransitiontobeproved}. The former holds immediately in the range $t_g^{10/13} \leq T \leq t_g^{1 - \alpha}$ by the assumption $L(1/2,f) \ll_{\e} t_f^{\e}$ and $|\zeta(1/2 + it)|^2 \ll_{\e} (1 + |t|)^{\e}$ in conjunction with the bounds \eqref{eqn:upperbounds}; $\GL_4 \times \GL_2 \leftrightsquigarrow \GL_4 \times \GL_2$ spectral reciprocity then yields the same result in the range $T \leq t_g^{3/13}$. The latter holds for $U \leq t_g^{1/3}$ by the same assumption in conjunction with the bounds \eqref{eqn:firstmomentconductordrop}; once more, $\GL_4 \times \GL_2 \leftrightsquigarrow \GL_4 \times \GL_2$ spectral reciprocity then yields the same result in the range $t_g^{1/3} \leq U \leq t_g^{1 - \alpha}$.
\end{proof}

Similar conditional analogues of \hyperref[prop:L4GLHconditional]{Proposition \ref*{prop:L4GLHconditional}} also hold for the $L^4$-norm of holomorphic Hecke cusp forms in the weight aspect \cite[Theorem 1.4]{BKY13} and in the level aspect \cite[Theorem 1.1]{BuK15}.

We may also obtain the improved bound $O_{\e}(t_g^{1 + \e} T)$ for \eqref{eqn:shortinitialtobeproved} in the ranges $T \leq t_g^{3/13}$ and $t_g^{10/13} \leq T \leq t_g^{1 - \alpha}$ and the improved bound $O_{\e}(t_g^{3/2 + \e} U^{1/2})$ for \eqref{eqn:shorttransitiontobeproved} in the range $t_g^{1/3} \leq U \leq t_g^{1 - \alpha}$ under a different assumption, namely the Lindel\"{o}f-on-average bound\footnote{In fact, we could make do with the weaker bound $\int_{U}^{2U} |L(1/2 + it,\ad g)|^2 \, dt \ll_{\e} t_g^{7/5 + \e} U^{3/10}$ uniformly for $t_g^{34/25} \leq U \leq t_g^2$, which is a Lindel\"{o}f-on-average bound only when $U \asymp t_g^2$.}
\begin{equation}
\label{eqn:L1/2F2ndmomentconditional}
\int_{U}^{2U} \left|L\left(\frac{1}{2} + it,\ad g\right)\right|^2 \, dt \ll_{\e} U^{1 + \e}
\end{equation}
uniformly for $t_g^{34/25} \leq U \leq t_g^2$. This conditional strengthening of \eqref{eqn:L1/2F2ndmoment} would yield the improved bound $O_{\e}(t_g^{1 + \e} + t_g^{7/10 + \e} T^2)$ for \eqref{eqn:upperbounds} in the range $T \leq t_g^{3/13}$, which would ensure the requisite bound $O_{\e}(t_g^{1 + \e} T)$ for \eqref{eqn:shortinitialtobeproved} in this range. An application of $\GL_4 \times \GL_2 \leftrightsquigarrow \GL_4 \times \GL_2$ spectral reciprocity would then yield this same requisite bound in the range $t_g^{10/13} \leq T \leq t_g^{1 - \alpha}$ as well as the requisite bound $O_{\e}(t_g^{3/2 + \e} U^{1/2})$ for \eqref{eqn:shorttransitiontobeproved} in the range $t_g^{1/3} \leq U \leq t_g^{1 - \alpha}$. This reduction to the assumption \eqref{eqn:L1/2F2ndmomentconditional} can be thought of as a ``reduction to Eisenstein observables'' akin to the work of Nelson \cite{Nel19a}. Unfortunately, while an unconditional proof of \eqref{eqn:L1/2F2ndmomentconditional} is not inconceivably unrealistic using current technology, the best known estimates in this regard fall shy of what is required (cf.~\cite{ALM22,Pal22}).

\subsection{Conditional Improvements via Fifth Moment Bounds}
\label{sect:fifthmoment}

An alternate conditional approach to improving \hyperref[thm:L4]{Theorems \ref*{thm:L4}} and \ref{thm:L4modified} would be to appeal to the conditional fifth moment bounds
\begin{equation}
\label{eqn:fifthmomentconditional}
\begin{drcases*}
\sum_{\substack{f \in \BB_0 \\ T \leq t_f \leq 2T}} \frac{L\left(\frac{1}{2},f\right)^5}{L(1,\ad f)} & \\
\frac{1}{2\pi} \int\limits_{T \leq |t| \leq 2T} \left|\frac{\zeta\left(\frac{1}{2} + it\right)^5}{\zeta(1 + 2it)}\right|^2 \, dt & \\
\sum_{\substack{f \in \BB_{\hol} \\ T \leq k_f \leq 2T}} \frac{L\left(\frac{1}{2},f\right)^5}{L(1,\ad f)} & 
\end{drcases*} \ll_{\e} T^{2 + \e}.
\end{equation}

\begin{proposition}
\label{prop:L45thmomentconditional}
Let $g$ be a Hecke--Maa\ss{} newform of Laplacian eigenvalue $\lambda_g$ on either $\Gamma_0(q) \backslash \Hb$ or $\Gamma^D \backslash \Hb$, where $q$ is squarefree and fixed and $D$ is the indefinite quaternion division algebra over $\Q$ of discriminant $q$. Under the assumption of \eqref{eqn:fifthmomentconditional}, $\|g\|_4 \ll_{\e} \lambda_g^{1/104 + \e}$.
\end{proposition}

\begin{proof}[Sketch of proof]
Using H\"{o}lder's inequality with exponents $(1/5,1/5,3/5)$ and combining \eqref{eqn:largesievebounds1}, \eqref{eqn:upperbounds}, and the assumption \eqref{eqn:fifthmomentconditional}, we obtain the improved bounds $O_{\e}(t_g^{2/5 + \e} T^{9/5})$ for \eqref{eqn:shortinitialtobeproved} in the range $t_g^{14/17} \leq T \leq t_g^{11/13}$. Via $\GL_4 \times \GL_2 \leftrightsquigarrow \GL_4 \times \GL_2$ spectral reciprocity, we similarly improve \eqref{eqn:shortinitialtobeproved} to $O_{\e}(t_g^{6/5 + \e} T^{1/5})$ in the range $t_g^{2/13} \leq T \leq t_g^{3/17}$ and \eqref{eqn:shorttransitiontobeproved} to $O_{\e}(t_g^{13/10 + \e} U^{9/10})$ in the range $t_g^{11/17} \leq U \leq t_g^{9/13}$. These in turn imply the improved bounds $O_{\e}(t_g^{1/13 + \e})$ for \eqref{eqn:initial} and \eqref{eqn:transition}.
\end{proof}

The second author \cite[Theorem 1.1]{Kha20} has shown that the third term on the left-hand side of \eqref{eqn:fifthmomentconditional} is $O_{\e}(T^{2 + 2\vartheta + \e})$, where $\vartheta$ denotes the current best bound towards the Selberg eigenvalue conjecture; the same method yields the same bound for the first and second terms on the left-hand side of \eqref{eqn:fifthmomentconditional}. Thus the Selberg eigenvalue conjecture implies the improved $L^4$-norm bound $\|g\|_4 \ll_{\e} \lambda_g^{1/104 + \e}$.


\begin{thebibliography}{BHKM20}

\bibitem[ALM22]{ALM22} Keshav Aggarwal, Wing Hong Leung, and Ritabrata Munshi, \href{https://doi.org/10.4171/jems/1667}{``Short Second Moment Bound and Subconvexity for $\GL(3)$ $L$-Functions''}, to appear in \textit{Journal of the European Mathematical Society} (2022), 58 pages.

\bibitem[AK18]{AK18} Nickolas Andersen and Eren Mehmet K\i{}ral, \href{https://doi.org/10.1112/S0025579318000256}{``Level Reciprocity in the Twisted Second Moment of Rankin--Selberg $L$-Functions''}, \textit{Mathematika} \textbf{64}:3 (2018), 770--784.

\bibitem[AW23]{AW23} Nickolas Andersen and Han Wu, \href{https://doi.org/10.1016/j.jnt.2022.07.012}{``Hybrid Subconvexity and the Partition Function''}, \textit{Journal of Number Theory} \textbf{242} (2023), 154--180.

\bibitem[BS17]{BS17} Matthew D.~Blair and Christopher D.~Sogge, \href{https://doi.org/10.1007/s00220-017-2977-8}{``Refined and Microlocal Kakeya--Nikodym Bounds of Eigenfunctions in Higher Dimensions''}, \textit{Communications in Mathematical Physics} \textbf{356}:2 (2017), 501--533.

\bibitem[BS18]{BS18} Matthew D.~Blair and Christopher D.~Sogge, \href{https://doi.org/10.4310/jdg/1527040871}{``Concerning Toponogov's Theorem and Logarithmic Improvement of Estimates of Eigenfunctions''}, \textit{Journal of Differential Geometry} \textbf{109}:2 (2018), 189--221.

\bibitem[BS19]{BS19} Matthew D.~Blair and Christopher D.~Sogge, \href{https://doi.org/10.1007/s00222-019-00873-6}{``Logarithmic Improvements in $L^p$ Bounds for Eigenfunctions at the Critical Exponent in the Presence of Nonpositive Curvature''}, \textit{Inventiones Mathematicae} \textbf{217}:2 (2019), 703--748.

\bibitem[Blo12]{Blo12} Valentin Blomer, \href{https://doi.org/10.1353/ajm.2012.0032}{``Subconvexity for Twisted $L$-Functions on $\GL(3)$''}, \textit{American Journal of Mathematics} \textbf{134}:5 (2012), 1385--1421.

\bibitem[Blo13]{Blo13} Valentin Blomer, \href{https://doi.org/10.4171/JEMS/405}{``On the $4$-Norm of an Automorphic Form''}, \textit{Journal of the European Mathematical Society} \textbf{15}:5 (2013), 1825--1852.

\bibitem[BHM07]{BHM07} Valentin Blomer, Gergely Harcos, and Philippe Michel, \href{https://doi.org/10.1016/j.ansens.2007.05.003}{``Bounds for Modular $L$-Functions in the Level Aspect''}, \textit{Annales Scientifiques de l'\'{E}cole Normale Sup\'{e}rieure, $4^{\mathrm{e}}$ s\'{e}rie} \textbf{40}:5 (2007), 697--740.

\bibitem[BHKM20]{BHKM20} Valentin Blomer, Peter Humphries, Rizwanur Khan, and Micah B.~Milinovich, \href{https://doi.org/10.1112/S0010437X20007101}{``Motohashi's Fourth Moment Identity for Dirichlet $L$-Functions for Non-Archimedean Test Functions and Applications''}, \textit{Compositio Mathematica} \textbf{156}:5 (2020), 1004--1038.

\bibitem[BlK19a]{BlK19a} Valentin Blomer and Rizwanur Khan, \href{https://doi.org/10.1016/j.jfa.2018.11.009}{``Uniform Subconvexity and Symmetry Breaking Reciprocity''}, \textit{Journal of Functional Analysis} \textbf{276}:7 (2019), 2315--2358.

\bibitem[BlK19b]{BlK19b} Valentin Blomer and Rizwanur Khan, \href{https://doi.org/10.1215/00127094-2018-0060}{``Twisted Moments of $L$-Functions and Spectral Reciprocity''}, \textit{Duke Mathematical Journal} \textbf{168}:6 (2019), 1109--1177.

\bibitem[BKY13]{BKY13} Valentin Blomer, Rizwanur Khan, and Matthew Young, \href{https://doi.org/10.1215/00127094-2380967}{``Distribution of Mass of Holomorphic Cusp Forms''}, \textit{Duke Mathematical Journal} \textbf{162}:14 (2013), 2609--2644.

\bibitem[BLM19]{BLM19} Valentin Blomer, Xiaoqing Li, and Stephen D.~Miller, \href{https://doi.org/10.1016/j.jnt.2019.05.011}{``A Spectral Reciprocity Formula and Non-Vanishing for $L$-Functions on $\GL(4) \times \GL(2)$''}, \textit{Journal of Number Theory Prime} \textbf{205} (2019), 1--43.

\bibitem[Bou09]{Bou09} Jean Bourgain, \href{https://doi.org/10.1090/trans2/226}{``Geodesic Restrictions and $L^p$-Estimates for Eigenfunctions of Riemannian Surfaces''} in \textit{Linear and Complex Analysis: Dedicated to V.~P.~Havin on the Occasion of His 75th Birthday}, editors Alexei Alexandrov, Anton Baranov, and Sergey Kislyakov, American Mathematical Society Translations: Series 2 \textbf{226}, Advances in the Mathematical Sciences, American Mathematical Society, Providence, RI, 2009, 27--35.

\bibitem[BD15]{BD15} Jean Bourgain and Ciprian Demeter, \href{https://doi.org/10.4007/annals.2015.182.1.9}{``The Proof of the $l^2$ Decoupling Conjecture''}, \textit{Annals of Mathematics} \textbf{182}:1 (2015), 351--389.

\bibitem[BuK15]{BuK15} Jack Buttcane and Rizwanur Khan, \href{https://doi.org/10.1007/s00208-014-1142-3}{``$L^4$-Norms of Hecke Newforms of Large Level''}, \textit{Mathematische Annalen} \textbf{362}:3--4 (2015), 699--715.

\bibitem[BuK17]{BuK17} Jack Buttcane and Rizwanur Khan, \href{https://doi.org/10.1112/S0010437X17007199}{``On the Fourth Moment of Hecke Maass Forms and the Random Wave Conjecture''}, \textit{Compositio Mathematica} \textbf{153}:7 (2017), 1479--1511.

\bibitem[CG23]{CG23} Yaiza Canzani and Jeffrey Galkowski, \href{https://doi.org/10.2140/apde.2023.16.2267}{``Growth of High $L^p$ Norms for Eigenfunctions: an Application of Geodesic Beams''}, \textit{Analysis \& PDE} \textbf{16}:10 (2023), 2267--2325.

\bibitem[Dem20]{Dem20} Ciprian Demeter, \href{https://doi.org/10.1017/9781108584401}{\textit{Fourier Restriction, Decoupling, and Applications}}, Cambridge Studies in Advanced Mathematics \textbf{184}, Cambridge University Press, Cambridge, 2020.

\bibitem[DK20]{DK20} Goran Djankovi\'{c} and Rizwanur Khan, \href{https://doi.org/10.1093/imrn/rny266}{``On the Random Wave Conjecture for Eisenstein Series''}, \textit{International Mathematics Research Notices} \textbf{2020}:23 (2020), 9694--9716.

\bibitem[GHLN24]{GHLN24} Soumendra Ganguly, Peter Humphries, Yongxiao Lin, and Ramon Nunes, \href{https://arxiv.org/abs/2408.00596v1}{``Strong Hybrid Subconvexity for Twisted Selfdual $\GL_3$ $L$-Functions''}, preprint (2024), 46 pages.

\bibitem[Gol06]{Gol06} Dorian Goldfeld, \href{https://doi.org/10.1017/CBO9780511542923}{\textit{Automorphic Forms and $L$-Functions for the Group $\GL(n,\R)$}}, Cambridge Studies in Advanced Mathematics \textbf{99}, Cambridge University Press, Cambridge, 2006.

\bibitem[GR15]{GR15} I.~S.~Gradshteyn and I.~M.~Ryzhik, \href{https://doi.org/10.1016/C2010-0-64839-5}{\textit{Table of Integrals, Series, and Products, Eighth Edition}}, editors Daniel Zwillinger and Victor Moll, Academic Press, Burlington, 2015.

\bibitem[HM06]{HM06} Gergely Harcos and Philippe Michel, \href{https://doi.org/10.1007/s00222-005-0468-6}{``The Subconvexity Problem for Rankin--Selberg $L$-Functions and Equidistribution of Heegner Points. II''}, \textit{Inventiones Mathematicae} \textbf{163}:3 (2006), 581--655.

\bibitem[HT15]{HT15} Andrew Hassell and Melissa Tacy, \href{https://doi.org/10.1515/forum-2012-0176}{``Improvement of Eigenfunction Estimates on Manifolds of Nonpositive Curvature''}, \textit{Forum Mathematicum} \textbf{27}:3 (2015), 1435--1451.

\bibitem[H-B78]{H-B78} D.~R.~Heath-Brown, \href{https://doi.org/10.1093/qmath/29.4.443}{``The Twelfth Power Moment of the Riemann-Function''}, \textit{The Quarterly Journal of Mathematics} \textbf{29}:4 (1978), 443--462.

\bibitem[HL24]{HL24} Fei Hou and GuangShi L\"{u}, \href{https://doi.org/10.1142/S1793042124500404}{``An Explicit Vorono\u{\i} Formula for $SL_3(\R)$ Newforms Underlying the Symmetric Lifts in the Level Aspect''}, \textit{International Journal of Number Theory} \textbf{20}:3 (2024), 797--809.

\bibitem[HS20]{HS20} Yueke Hu and Abhishek Saha, \href{http://doi.org/10.1112/S0010437X20007460}{``Sup-norms of Eigenfunctions in the Level Aspect for Compact Arithmetic Surfaces, II: Newforms and Subconvexity''}, \textit{Compositio Mathematica} \textbf{156}:11 (2020), 2368--2398.

\bibitem[Hum18]{Hum18} Peter Humphries, \href{https://doi.org/10.1007/s00208-018-1677-9}{``Equidistribution in Shrinking Sets and $L^4$-Norm Bounds for Automorphic Forms''}, \textit{Mathematische Annalen} \textbf{371}:3--4 (2018), 1497--1543.

\bibitem[HK20]{HK20} Peter Humphries and Rizwanur Khan, \href{https://doi.org/10.1007/s00039-020-00526-4}{``On the Random Wave Conjecture for Dihedral Maa\ss{} Forms''}, \textit{Geometric and Functional Analysis} \textbf{30}:1 (2020), 34--125.

\bibitem[HK24]{HK24} Peter Humphries and Rizwanur Khan, \href{https://doi.org/10.1007/s00208-023-02747-y}{``The Twelfth Moment of Hecke $L$-Functions in the Weight Aspect''}, \textit{Mathematische Annalen} \textbf{389}:4 (2024), 3935--3974.

\bibitem[Ich08]{Ich08} Atsushi Ichino, \href{https://doi.org/10.1215/00127094-2008-052}{``Trilinear Forms and the Central Values of Triple Product $L$-Functions''}, \textit{Duke Mathematical Journal} \textbf{145}:2 (2008), 281--307.

\bibitem[Ivi01]{Ivi01} Aleksandar Ivi\'{c}, \href{https://doi.org/10.5802/jtnb.333}{``On Sums of Hecke Series in Short Intervals''}, \textit{Journal de Th\'{e}orie des Nombres de Bordeaux} \textbf{13}:2 (2001), 453--468.

\bibitem[Ivi03]{Ivi03} Aleksandar Ivi\'{c}, \textit{The Riemann Zeta-Function}, Dover Publications, Inc., Mineola, New York, 2003.

\bibitem[Iwa80]{Iwa80} Henryk Iwaniec, \href{https://www.jstor.org/stable/44165358}{``Fourier Coefficients of Cusp Forms and the Riemann Zeta-Function''}, \textit{Seminaire de Th\'{e}orie des Nombres de Bordeaux} \textbf{18} (1980), 1--36.

\bibitem[Iwa02]{Iwa02} Henryk Iwaniec, \href{https://doi.org/10.1090/gsm/053}{\textit{Spectral Methods of Automorphic Forms, Second Edition}}, Graduate Studies in Mathematics \textbf{53}, American Mathematical Society, Providence, 2002.

\bibitem[IK04]{IK04} Henryk Iwaniec and Emmanuel Kowalski, \href{https://doi.org/10.1090/coll/053}{\textit{Analytic Number Theory}}, American Mathematical Society Colloquium Publications \textbf{53}, American Mathematical Society, Providence, 2004.

\bibitem[IS95]{IS95} H.~Iwaniec and P.~Sarnak, \href{https://doi.org/10.2307/2118522}{``$L^{\infty}$ Norms of Eigenfunctions of Arithmetic Surfaces''}, \textit{Annals of Mathematics} \textbf{141}:2 (1995), 301--320.

\bibitem[JN21]{JN21} Subhajit Jana and Ramon Nunes, \href{https://arxiv.org/abs/2111.02297v2}{``Spectral Reciprocity for $\GL(n)$ and Simultaneous Non-vanishing of Central $L$-Values''}, to appear in \textit{American Journal of Mathematics} (2021), 62 pages.

\bibitem[Jut87]{Jut87} M.~Jutila, \href{http://www.math.tifr.res.in/~publ/ln/tifr80.pdf}{\textit{Lectures on a Method in the Theory of Exponential Sums}}, Tata Institute of Fundamental Research Lectures on Mathematics and Physics \textbf{80}, Springer--Verlag, Berlin, 1987.

\bibitem[Jut99]{Jut99} M.~Jutila, \href{https://www.emis.de/journals/PIMB/079/n079p031.pdf}{``Convolutions of Fourier Coefficients of Cusp Forms''}, \textit{Publications de l'Institut Math\'{e}matique, Nouvelle s\'{e}rie} \textbf{65}:79 (1999), 31--51.

\bibitem[Jut04a]{Jut04a} M.~Jutila, \href{https://doi.org/10.2298/PIM0476041J}{``The Spectral Mean Square of Hecke $L$-Functions on the Critical Line''}, \textit{Publications de l'Institut Math\'{e}matique, Nouvelle s\'{e}rie} \textbf{76}:90 (2004), 41--55.

\bibitem[Jut04b]{Jut04b} Matti Jutila, \href{https://doi.org/10.1016/j.jnt.2004.05.012}{``The Twelfth Moment of Central Values of Hecke Series''}, \textit{Journal of Number Theory} \textbf{108}:1 (204), 157--168.

\bibitem[Kan22]{Kan22} Ikuya Kaneko, \href{https://doi.org/10.1017/fms.2022.33}{``Motohashi's Formula for the Fourth Moment of Individual Dirichlet $L$-Functions and Applications''}, \textit{Forum of Mathematics, Sigma} \textbf{10}:e41 (2022), 1--43.

\bibitem[KaSa93]{KaSa93} Svetlana Katok and Peter Sarnak, \href{https://doi.org/10.1007/BF02761700}{``Heegner Points, Cycles and Maass Forms''}, \textit{Israel Journal of Mathematics} \textbf{84}:1--2 (1993), 193--227.

\bibitem[Kha14]{Kha14} Rizwanur Khan, \href{https://doi.org/10.1007/s11139-013-9505-z}{``On the Fourth Moment of Holomorphic Hecke Cusp Forms''}, \textit{The Ramanujan Journal} \textbf{34}:1 (2014), 83--107.

\bibitem[Kha20]{Kha20} Rizwanur Khan, \href{https://doi.org/10.1017/S0305004118000944}{``The Fifth Moment of Hecke $L$-Functions in the Weight Aspect''}, \textit{Mathematical Proceedings of the Cambridge Philosophical Society} \textbf{168}:3 (2020), 543--566.

\bibitem[Kha22]{Kha22} Rizwanur Khan, \href{https://doi.org/10.1090/tran/8647}{``Subconvexity Bounds for Twisted $L$-Functions, II''}, \textit{Transactions of the American Mathematical Society} \textbf{375}:10 (2022), 6769--6796.

\bibitem[KhSt24]{KhSt24} Ilya Khayutin and Raphael S.~Steiner, \href{https://doi.org/10.1112/S0010437X24007437}{``Theta Functions, Fourth Moments of Eigenforms, and the Sup-Norm Problem I''}, \textit{Compositio Mathematica} \textbf{160}:12 (2024), 2916--2969.

\bibitem[Ki23]{Ki23} Haseo Ki, \href{https://arxiv.org/abs/2302.02625v1}{``$L^4$-Norms and Sign Changes of Maass Forms''}, preprint (2023), 27 pages.

\bibitem[Kuz89]{Kuz89} N.~V.~Kuznetsov, ``Sums of Kloosterman Sums and the Eighth Power Moment of the Riemann Zeta-Function'', in \textit{Number Theory and Related Topics}, Tata Institute of Fundamental Research: Studies in Mathematics \textbf{12}, Oxford University Press, 1989, 57--117.

\bibitem[Kuz99]{Kuz99} N.~V.~Kuznetsov, ``The Hecke Series at the Center of the Critical Strip'', preprint (1999), 27 pages.

\bibitem[Kwa24]{Kwa24} Chung-Hang Kwan, \href{https://arxiv.org/abs/2112.08568v2}{``Spectral Moment Formulae for $\GL(3) \times \GL(2)$ $L$-Functions I: The Cuspidal Case''}, \textit{Algebra \&{} Number Theory} \textbf{18}:10 (2024), 1817--1862.

\bibitem[Lap03]{Lap03} Erez Lapid, \href{https://doi.org/10.1155/S1073792803204013}{``On the Nonnegativity of Rankin-Selberg $L$-Functions at the Center of Symmetry''}, \textit{International Mathematics Research Notices} \textbf{2003}:2 (2003), 65--75.

\bibitem[Li10]{Li10} Xiannan Li, \href{https://doi.org/10.1093/imrn/rnp148}{``Upper Bounds on $L$-Functions at the Edge of the Critical Strip''}, \textit{International Mathematics Research Notices} \textbf{2010}:4 (2010), 727--755.

\bibitem[Li09]{Li09} Xiaoqing Li, \href{https://doi.org/10.1007/s00039-008-0692-5}{``The Central Value of the Rankin--Selberg $L$-Functions''}, \textit{Geometric and Functional Analysis} \textbf{18}:5 (2009), 1660--1695.

\bibitem[Li11]{Li11} Xiaoqing Li, \href{http://doi.org/10.4007/annals.2011.173.1.8}{``Bounds for $\GL(3) \times \GL(2)$ $L$-Functions and $\GL(3)$ $L$-Functions''}, \textit{Annals of Mathematics} \textbf{173}:1 (2011), 301--336.

\bibitem[LNQ23]{LNQ23} Yongxiao Lin, Ramon Nunes, and Zhi Qi, \href{https://doi.org/10.1093/imrn/rnac153}{``Strong Subconvexity for Self-Dual $\GL(3)$ $L$-Functions''}, \textit{International Mathematics Research Notices} \textbf{2023}:13 (2023), 11453--11470.

\bibitem[Liu15]{Liu15} Sheng-Chi Liu, \href{https://doi.org/10.1007/s12188-014-0100-z}{``$L^4$-Norms of the Holomorphic Dihedral Forms of Large Level''}, \textit{Abhandlungen aus dem Mathematischen Seminar der Universit\"{a}t Hamburg} \textbf{85}:1 (2015), 53--57.

\bibitem[Luo14]{Luo14} Wenzhi Luo, \href{https://doi.org/10.1093/imrn/rns298}{``$L^4$-Norms of the Dihedral Maass Forms''}, \textit{International Mathematics Research Notices} \textbf{2014}:8 (2014), 2294--2304.

\bibitem[Mar16a]{Mar16a} Simon Marshall, \href{https://doi.org/10.1215/00127094-3166736}{``Geodesic Restrictions of Arithmetic Eigenfunctions''}, \textit{Duke Mathematical Journal} \textbf{165}:3 (2016), 463--508.

\bibitem[Mar16b]{Mar16b} Simon Marshall, \href{https://doi.org/10.4171/JEMS/619}{``$L^p$-Norms of Higher Rank Eigenfunctions and Bounds for Spherical Functions''}, \textit{Journal of the European Mathematical Society} \textbf{18}:7 (2016), 1437--1493.

\bibitem[Mar16c]{Mar16c} Simon Marshall, \href{https://doi.org/10.2140/ant.2016.10.803}{``Local Bounds for $L^p$ Norms of Maass Forms in the Level Aspect''}, \textit{Algebra \&{} Number Theory} \textbf{10}:4 (2016), 803--812.

\bibitem[Mia21]{Mia21} Xinchen Miao, \href{https://arxiv.org/abs/2110.11529v1}{``Spectral Reciprocity for the Product of Rankin--Selberg $L$-Functions''}, preprint (2021), 31 pages.

\bibitem[Mil10]{Mil10} Djordje Mili\'{c}evi\'{c}, \href{https://doi.org/10.1215/00127094-2010-058}{``Large Values of Eigenfunctions on Arithmetic Hyperbolic Surfaces''}, \textit{Duke Mathematical Journal} \textbf{155}:2 (2010), 365--401.

\bibitem[Mil11]{Mil11} Djordje Mili\'{c}evi\'{c}, \href{https://doi.org/10.1007/s00039-011-0144-5}{``Large Values of Eigenfunctions on Arithmetic Hyperbolic $3$-Manifolds''}, \textit{Geometric and Functional Analysis} \textbf{21}:6 (2011), 1375--1418.

\bibitem[MV74]{MV74} H.~L.~Montgomery and R.~C.~Vaughan, \href{https://doi.org/10.1112/jlms/s2-8.1.73}{``Hilbert's Inequality''}, \textit{Journal of the London Mathematical Society} \textbf{8} (1974), 73--82.

\bibitem[Mot97]{Mot97} Yoichi Motohashi, \href{https://doi.org/10.1017/CBO9780511983399}{\textit{Spectral Theory of the Riemann Zeta-Function}}, Cambridge Tracts in Mathematics \textbf{127}, Cambridge University Press, Cambridge, 1997.

\bibitem[Mot99]{Mot99} Yoichi Motohashi, \href{https://doi.org/10.3792/pjaa.75.147}{``A Note on the Mean Value of the Zeta and $L$-Functions. IX''}, \textit{Proceedings of the Japan Academy, Series A, Mathematical Sciences} \textbf{75}:8 (1999), 147--149.

\bibitem[Mot03]{Mot03} Yoichi Motohashi, \href{https://arxiv.org/abs/math/0310105v1}{``A Functional Equation for the Spectral Fourth Moment of Modular Hecke $L$-Functions''}, in \textit{Proceedings of the Session in Analytic Number Theory and Diophantine Equations}, Bonner mathematische Schriften \textbf{360}, Universit\"{a}t Bonn, Bonn, 2003, 19 pages.

\bibitem[Nel19a]{Nel19a} Paul D.~Nelson, \href{https://doi.org/10.1215/00127094-2019-0005}{``Subconvex Equidistribution of Cusp Forms: Reduction to Eisenstein Observables''}, \textit{Duke Mathematical Journal} \textbf{168}:9 (2019), 1665--1722.

\bibitem[Nel19b]{Nel19b} Paul D.~Nelson, \href{https://arxiv.org/abs/1911.06310v3}{``Eisenstein Series and the Cubic Moment for $\PGL(2)$''}, preprint (2019), 77 pages.

\bibitem[Nun23]{Nun23} Ramon M.~Nunes, \href{https://doi.org/10.2140/ant.2023.17.1381}{``Spectral Reciprocity via Integral Representations''}, \textit{Algebra \&{} Number Theory} \textbf{17}:8 (2023), 1381--1409.

\bibitem[Pal22]{Pal22} Sampurna Pal, \href{https://arxiv.org/abs/2212.14620v2}{``Second Moment of Degree Three $L$-Functions''}, preprint (2022), 50 pages.

\bibitem[Pet15]{Pet15} Ian N.~Petrow, \href{https://doi.org/10.1007/s00208-014-1166-8}{``A Twisted Motohashi Formula and Weyl-Subconvexity for $L$-Functions of Weight Two Cusp Forms''}, \textit{Mathematische Annalen} \textbf{363}:1--2 (2015), 175--216.

\bibitem[Sah17]{Sah17} Abhishek Saha, \href{http://doi.org/10.2140/ant.2017.11.1009}{``Hybrid Sup-norm Bounds for Maass Newforms of Powerful Level''}, \textit{Algebra \&{} Number Theory} \textbf{11}:5 (2017), 1009--1045.

\bibitem[Sar03]{Sar03} Peter Sarnak, \href{http://doi.org/10.1090/S0273-0979-03-00991-1}{``Spectra of Hyperbolic Surfaces''}, \textit{Bulletin of the American Mathematical Society} \textbf{40}:4 (2003), 441--478.

\bibitem[Sog88]{Sog88} Christopher D.~Sogge, \href{https://doi.org/10.1016/0022-1236(88)90081-X}{``Concerning the $L^p$ Norm of Spectral Clusters for Second-Order Elliptic Operators on Compact Manifolds''}, \textit{Journal of Functional Analysis} \textbf{77}:1 (1988), 123--138.

\bibitem[Spi03]{Spi03} Florin Spinu, \href{https://search.proquest.com/docview/288239392/}{\textit{The $L^4$ Norm of the Eisenstein Series}}, Ph.D.~Thesis, Princeton University, 2003.

\bibitem[Ste94]{Ste94} Gunther Steil, \href{http://cds.cern.ch/record/260472/files/P00022028.pdf}{``Eigenvalues of the Laplacian and of the Hecke Operators for $\PSL(2,\Z)$''}, Technical Report DESY 94-028, DESY, Hamburg (1994), 25 pages.

\bibitem[Wal85]{Wal85} J.-L.~Waldspurger, \href{http://www.numdam.org/item/CM_1985__54_2_173_0/}{``Sur les valeurs de certaines fonctions $L$ automorphes en leur centre de sym\'{e}trie''}, \textit{Compositio Mathematica} \textbf{54}:2 (1985), 173--242.

\bibitem[Wat44]{Wat44} G.~N.~Watson, \textit{A Treatise on the Theory of Bessel Functions}, Cambridge University Press, Cambridge, 1944.

\bibitem[Wat08]{Wat08} Thomas C.~Watson, \href{https://arxiv.org/abs/0810.0425v3}{\textit{Rankin Triple Products and Quantum Chaos}}, Ph.D.~Thesis, Princeton University, 2002 (revised 2008).

\bibitem[Wu22]{Wu22} Han Wu, \href{https://doi.org/10.1090/tran/8750}{``On Motohashi's Formula''}, \textit{Transactions of the American Mathematical Society} \textbf{375}:11 (2022), 8033--8081.

\bibitem[Xia07]{Xia07} Honggang Xia, \href{https://doi.org/10.1016/j.jnt.2006.09.003}{``On $L^{\infty}$ Norms of Holomorphic Cusp Forms''}, \textit{Journal of Number Theory} \textbf{124}:2 (2007), 325--327.

\bibitem[Zac19]{Zac19} Rapha\"{e}l Zacharias, \href{https://arxiv.org/abs/1912.01512}{``Periods and Reciprocity II''}, preprint (2019), 30 pages.

\bibitem[Zac21]{Zac21} Rapha\"{e}l Zacharias, \href{https://doi.org/10.1093/imrn/rnz100}{``Periods and Reciprocity I''}, \textit{International Mathematics Research Notices} \textbf{2021}:3 (2021), 2191--2209.

\bibitem[Zen23]{Zen23} Peter Zenz, \href{https://doi.org/10.1093/imrn/rnac199}{``Sharp Bound for the Fourth Moment of Holomorphic Hecke Cusp Forms''}, \textit{International Mathematics Research Notices} \textbf{2023}:16 (2023), 13562--13600.

\bibitem[Zyg74]{Zyg74} A.~Zygmund, \href{https://doi.org/10.4064/sm-50-2-189-201}{``On Fourier Coefficients and Transforms of Functions of Two Variables''}, \textit{Studia Mathematica} \textbf{50} (1974), 189--201.

\end{thebibliography}
\end{document}